\documentclass{amsart}
\usepackage[utf8]{inputenc}
\usepackage{amsmath,amsthm,amssymb,amscd,url,enumerate}
\usepackage[utf8]{inputenc}
\usepackage{hyperref}
\usepackage{pdflscape}
\usepackage{caption}
\usepackage{comment}
\usepackage[noadjust]{cite}
\usepackage{xypic}
\usepackage{tikz}
\usetikzlibrary{decorations.markings,arrows}
\usepackage{tikz-cd}
\usepackage{amsfonts}
\usepackage{amsmath}
\usepackage{graphicx}
\usepackage{amsthm}
\usepackage{xcolor}
\usepackage{mathrsfs}
\usepackage[ruled,vlined,linesnumbered,algosection]{algorithm2e}
\usepackage{ragged2e}
\usepackage{algpseudocode}
\usepackage[inline]{enumitem}

\theoremstyle{plain} 
\newtheorem{theorem}{Theorem}[section] 

\newtheorem{proposition}[theorem]{Proposition}
\newtheorem{lemma}[theorem]{Lemma} 
\newtheorem{corollary}[theorem]{Corollary}
\newtheorem{heuristic}[theorem]{Heuristic}

\newtheorem{definition}[theorem]{Definition}
\newtheorem{remark}[theorem]{Remark}

\newtheorem{example}[theorem]{Example}
\newtheorem{problem}[theorem]{Problem}

\DeclareMathOperator{\End}{\operatorname{End}}
\DeclareMathOperator{\Cl}{\operatorname{Cl}}
\let\SS\relax
\DeclareMathOperator{\SS}{\operatorname{SS}}

\DeclareMathOperator{\disc}{disc}
\DeclareMathOperator{\nrd}{nrd}
\DeclareMathOperator{\Tr}{Tr}
\DeclareMathOperator{\Leaves}{Leaves}
\DeclareMathOperator{\Diag}{Diag}
\DeclareMathOperator{\Aut}{Aut}
\DeclareMathOperator{\polylog}{\operatorname{polylog}}

\newcommand{\IsOrientable}[1]{\textsf{\upshape IsOrientable}\ensuremath{_{#1}}}
\newcommand{\FindSmoothGen}[1]{\textsf{\upshape Find\-Smooth\-Gen}\ensuremath{_{#1}}}
\newcommand{\SmoothFact}[1]{\textsf{\upshape Smooth\-Fact}\ensuremath{_{#1}}}
\newcommand{\EfficientRep}[1]{\textsf{\upshape Efficient\-Rep}\ensuremath{_{#1}}}
\newcommand{\TreeFill}[1]{\textsf{\upshape Tree\-Fill}\ensuremath{_{#1}}}

\newcommand{\Z}{\mathbb{Z}}
\newcommand{\Q}{\mathbb{Q}}

\newcommand{\C}{\mathbb{C}}
\newcommand{\F}{\mathbb{F}}
\newcommand{\Fpbar}{\overline{\mathbb{F}_p}}

\title{Finding Orientations of Supersingular Elliptic Curves and Quaternion Orders}
\author[S. Arpin, J. Clements, P. Dartois, J. K. Eriksen, P. Kutas, B. Wesolowski]{Sarah Arpin, James Clements, Pierrick Dartois, Jonathan Komada Eriksen, P\'eter Kutas, Benjamin Wesolowski}

\address{Mathematics Institute,
Universiteit Leiden,
Leiden, The Netherlands}
\email{s.a.arpin@math.leidenuniv.nl}

\address{School of Computer Science, University of Bristol, Bristol, United Kingdom}
\email{james.clements@bristol.ac.uk}

\address{Centre Inria de l'Universit\'{e} de Bordeaux, Institut de Math\'{e}matiques de Bordeaux, UMR 5251, Bordeaux, France}
\email{pierrick dot dartois at u-bordeaux dot fr}

\address{Department of Information Security and Communication Technology, Norwegian University of Science and Technology, Trondheim, Norway}
\email{jonathan.k.eriksen@ntnu.no}

\address{Faculty of Informatics, Eötvös Loránd University, Hungary and School of Computer Science, University of Birmingham, UK}
\email{p.kutas@bham.ac.uk}

\address{ENS de Lyon, CNRS, UMPA, UMR 5669, Lyon, France}
\email{benjamin.wesolowski@math.u-bordeaux.fr}

\thanks{This research was funded in part by the UK Engineering and Physical Sciences Research Council (EPSRC) (grant number EP/V011324/1.), by the Agence Nationale de la Recherche under grant ANR MELODIA (ANR-20-CE40-0013), and the France 2030 program under grant agreement No. ANR-22-PETQ-0008 PQ-TLS. This research is also supported by the Hungarian Ministry of Innovation and Technology NRDI Office within the framework of the Quantum Information National Laboratory Program, the János Bolyai Research Scholarship of the Hungarian Academy of Sciences.}

\date{\today}

\begin{document}

\maketitle

\begin{abstract}
    Orientations of supersingular elliptic curves encode the information of an endomorphism of the curve. Computing the full endomorphism ring is a known hard problem, so one might consider how hard it is to find one such orientation. We prove that access to an oracle which tells if an elliptic curve is $\mathfrak{O}$-orientable for a fixed imaginary quadratic order $\mathfrak{O}$ provides non-trivial information towards computing an endomorphism corresponding to the $\mathfrak{O}$-orientation. We provide explicit algorithms and in-depth complexity analysis.

    We also consider the question in terms of quaternion algebras. We provide algorithms which compute an embedding of a fixed imaginary quadratic order into a maximal order of the quaternion algebra ramified at $p$ and $\infty$. We provide code implementations in Sagemath \cite{sage} which is efficient for finding embeddings of imaginary quadratic orders of discriminants up to $O(p)$, even for cryptographically sized $p$. 
\end{abstract}

\tableofcontents
\section{Introduction}

Isogeny-based cryptography is a relatively new branch of post-quantum cryptography which is based on hard problems supposedly intractable even for quantum computers. The underlying hard problems were first introduced publicly in 2006 by the hash-function proposal of Charles-Goren-Lauter \cite{CGL}, and the works of Couveignes \cite{couveignes2006hard} and Rostovtsev-Stolbunov \cite{RostStol}. Since then, this field has blossomed with the introductions of new schemes such as SIDH \cite{SIDH} (now broken by \cite{CastryckDecruSIDH,Maino_et_al_SIDH,robert2023breaking}), CSIDH \cite{castryck2018csidh}, and SQISign \cite{de2020sqisign}. The hardness of all isogeny-based schemes is based on some variant of the path finding problem, which asks to find an isogeny between two given supersingular elliptic curves. The quaternion analogue of this hard problem has been efficiently solved \cite{KLPT}, but the problem remains hard for supersingular elliptic curves. Path finding in the supersingular isogeny graph is equivalent to endomorphism ring computation, which was first heuristically proven in \cite{eisentrager2018supersingular} and then rigorously (assuming GRH) proven in \cite{Wesolowski_EquivProbs}.
The key recovery of CSIDH was reduced to endomorphism ring computations in \cite{Bweso_orientations}. 

To study the hardness of the path finding problem it is natural to add some data to the elliptic curves and study how this data interacts with the graph structure. One way to do this is to add the information of an orientation to the elliptic curve vertices. Informally, an orientation on an elliptic curve $E$ is an embedding of an imaginary quadratic order $\mathfrak{O}$ into the endomorphism ring of $E$ which cannot be extended to a superorder of $\mathfrak{O}$. The resulting isogeny graph admits an abelian group action, which is used in cryptographic protocols such as CSIDH \cite{castryck2018csidh}, Scallop \cite{SCALLOP}, OSIDH \cite{OSIDH}, and SETA \cite{SETA}. The group action is crucial for defining the Uber isogeny problem, whose hardness underlies all isogeny-based schemes. 
One might suspect that being given the information of an orientation could weaken the difficulty of the path finding problem, but this depends heavily on the given orientation and does not typically weaken the hardness of the path finding problem \cite{Bweso_orientations,arpin2022orienteering}. A natural question to consider would be how to find an orientation on a curve, given that one exists. This is the $\mathfrak{O}$-Orienting Problem. It is also natural to consider the decisional version of this problem: given a supersingular elliptic curve $E$ and a quadratic order $\mathfrak O$, can one decide whether $E$ is orientable by $\mathfrak O$?  Solving either of these problems would completely break OSIDH \cite{OSIDH,OSIDH_Onuki,OSIDH_security}. Interestingly, they are not even efficiently solved on the quaternion side. In this work, we give reductions between the search and decision variants of these problems, and provide algorithms for the quaternion variant of these problems. 
\medskip

\newpage
\subsection{Our contributions}

\subsubsection{Reduction from Search to Decision $\mathfrak{O}$-Orienting Problem}

When the discriminant of $\mathfrak{O}$ is smaller than the characteristic $p$ of the base field, we prove a subexponential reduction from the computational to the decisional version of the $\mathfrak{O}$-Orienting Problem. In particular, we provide an explicit algorithm (Algorithm~\ref{algo: reduction search to decision}) to find an $\mathfrak{O}$-orientation of an orientable elliptic curve in subexponential time and space when given access to an oracle deciding whether any elliptic curve is $\mathfrak{O}$-orientable. In Section~\ref{sec: complexity analysis} we provide an in-depth analysis and proof of the complexity of the algorithm (Theorem~\ref{thm: complexity}). This proves that such an oracle gives non-trivial information since finding an orientation automatically yields a non-scalar endomorphism and the best known algorithms to find a non-scalar endomorphism on a supersingular elliptic curve are exponential \cite{DelfsGalbraith2016,Eisentrager_al_2018}.

Before treating the general case, we prove a polynomial reduction when $\mathfrak{O}$ is the maximal order of $\mathbb{Q}(\sqrt{-d})$ and $d$ is the product of small distinct primes in Section~\ref{sec: smooth disc}. This allows us to illustrate the spirit of the more general algorithm in a less complicated setting. We provide an explicit algorithm for this case (Algorithm~\ref{algo: special reduction search to decision}) and prove in Theorem~\ref{thm: special_red} that this algorithm runs in polynomial time.

\subsubsection{Quaternion Order Embedding Problem}

In Section~\ref{sec: quaternion}, we consider the Quaternion Order Embedding Problem (Problem~\ref{prob: quat}) which is the quaternion analogue of the $\mathfrak{O}$-Orienting Problem. That is, given a maximal quaternion order $\mathcal{O} \subset B_{p, \infty}$ and a quadratic order $\mathfrak{O}$ which embeds into $\mathcal{O}$, find an embedding $\iota: \mathfrak{O} \hookrightarrow \mathcal{O}$ that cannot be extended to a superorder of $\mathfrak{O}$. 
In Section~\ref{ssec: quat_findingembeddings} we present a general algorithm to solve the problem of finding embeddings using a factorization oracle. We provide a complexity analysis based on several heuristics in Section~\ref{ssec: quat_complexity}. In Section~\ref{primitive_orientations} we show that finding embeddings which cannot be extended (i.e., orientations), only adds a small factor to the running time. We prove efficiency for the curve with $j$-invariant $1728$, and describe a practical method of removing the dependence on the factorization oracle. When the discriminant $\disc(\mathfrak{O})$ is small, our algorithm improves the state of the art being efficient up to $\disc(\mathfrak{O}) = O(p)$. We provide an implementation in Sagemath \cite{sage} which, for small discriminant orders, is fast for cryptographically sized $p$.

Code is available at: \texttt{https://github.com/jtcc2/finding-orientations}

\subsection*{Acknowledgements} This project began at KU Leuven Isogeny Days in 2022, and the authors extend their gratitude to the organizers.

\section{Preliminaries}

We provide a concise summary of the necessary background and the state of the art algorithms which we use in this paper. 

\subsection{Supersingular elliptic curves and quaternion algebras}
Let $p$ be a prime. An elliptic curve over $\Fpbar$ is called \textit{supersingular} if any one of the following equivalent conditions holds :
\begin{enumerate}
    \item $\End(E)$ is isomorphic to a maximal order in a quaternion algebra
    \item $E[p^r] = 0_E$ for all $r\geq 1$
    \item $j(E)\in\mathbb{F}_{p^2}$ and the multiplication-by-$p$ map $[p]$ is purely inseparable
    \item The dual to the $p^r$-power Frobenius is purely inseparable for all $r\geq 1$. 
\end{enumerate}
See \cite{Silverman} for additional properties and proofs of equivalence.

We use the endomorphism ring heavily in what follows, so we describe here the necessary definitions and properties of quaternion objects. For more generality and more detail, we encourage the reader to see \cite{Voight}.

A (definite) quaternion algebra $\mathcal{A}$ is a noncommutative algebra which has rank~4 over $\mathbb{Q}$, and can be specified by generators $i,j$ such that:
\[\mathcal{A} = \Q + \Q i + \Q j + \Q k: \quad i^2,j^2\in\mathbb{Q},\quad i^2,i^2<0,\quad k:= ij = -ji.\]
An order $\mathcal{O}$ in $\mathcal{A}$ is a $\Z$-submodule of $\mathcal{A}$ of rank~$4$ which is also a subring. An order is said to be maximal if it is not properly contained in any other order. 
For any lattice $I$ in $\mathcal{A}$, we define its left order 
\[O_L(I):=\{\alpha\in\mathcal{A}: \alpha I\subseteq I\}.\]
The right order $O_R(I)$ is defined analogously.
A lattice $I$ of $\mathcal{A}$ is said to be invertible if there exists a lattice $I'$ such that $II' = O_L(I) = O_R(I')$ and $I'I = O_R(I) = O_L(I')$. 
A lattice in $\mathcal{A}$ is said to be a left (resp. right) $\mathcal{O}$-ideal if $\mathcal{O}\subseteq O_L(I)$ (resp. $\mathcal{O}\subseteq O_R(I)$). For every order $\mathcal{O}$ of $\mathcal{A}$ we can define a left class set of equivalence classes of invertible ideals: 
Two invertible left-$\mathcal{O}$ ideals are equivalent in the left class set of $\mathcal{O}$ if they differ by a unit of $\mathcal{A}$. The left class set of invertible ideals is finite. The right class set of invertible ideals is analogously defined and is also finite.

For a fixed prime $p$, we define the (unique up to isomorphism) quaternion algebra $B_{p,\infty}$ to be the definite quaternion algebra ramified precisely at $p$ and $\infty$. 
The endomorphism rings of supersingular elliptic curves over $\Fpbar$ are isomorphic to maximal orders in $B_{p,\infty}$:
\begin{theorem}[Deuring \cite{Deu41}]
Fix a maximal order $M$ of the quaternion algebra $B_{p,\infty}$ ramified precisely at $p$ and $\infty$. 
There is a bijection between isomorphism classes of supersingular elliptic curves over $\Fpbar$ and the left class set of the order $M$. 
\end{theorem}

Given a supersingular elliptic curve $E/\Fpbar$, one might ask to compute $\End(E)$ in different forms: to compute endomorphisms of $E$ which generate $\End(E)$, or to compute the isomorphism class of $\End(E)$ abstractly in the quaternion algebra $B_{p,\infty}$. This problem is computationally difficult in all formulations. The information of one endomorphism $\omega$ of $E$ reveals an imaginary quadratic order $\mathbb{Z}[\omega]$ embedded within $\End(E)$. In Section~\ref{sec: orientations}, we provide more background information on such embeddings.

\subsection{Orientations}\label{sec: orientations}

\begin{definition}[Orientation]\label{def: orientation}
Let $\mathfrak{O}$ be an imaginary quadratic order. An $\mathfrak{O}$-orientation of a supersingular elliptic curve $E/\Fpbar$ is an embedding $\iota:\mathfrak{O}\hookrightarrow \End(E)$ which cannot be extended to a larger order containing $\mathfrak{O}$. The pair $(E,\iota)$ is called an $\mathfrak{O}$-oriented supersingular elliptic curve.
\end{definition}

Definition~\ref{def: orientation} corresponds to the definition of \textit{primitive} $\mathfrak{O}$-orientation found elsewhere in the literature \cite{OSIDH,OSIDH_Onuki,arpin2022orienteering}. We omit the word ``primitive" in our definition, as almost all of our $\mathfrak{O}$-orientations are primitive. When we want to discuss an embedding $\mathfrak{O}\hookrightarrow\End(E)$ which can be extended to a superorder of $\mathfrak{O}$, we highlight this by using the term ``imprimitive".

The notion of an orientation as in Definition~\ref{def: orientation} was recently introduced to isogeny-based cryptography by Col\`o and Kohel \cite{OSIDH} and was subsequently studied \cite{OSIDH_Onuki,Bweso_orientations,SCALLOP,arpin2022orienteering,ACLSST2022_orientations}. The quaternion counterpart of this notion has a longer history, dating back to Chevalley, Hasse, and Noether and often referred to as the theory of \textit{optimal embeddings}.

Supersingular elliptic curves which admit an $\mathfrak{O}$-orientation are called $\mathfrak{O}$-orientable. There is an action of the class group $\Cl(\mathfrak{O})$ on the set of $\mathfrak{O}$-oriented supersingular elliptic curves induced by the following action of an invertible $\mathfrak{O}$-ideal $\mathfrak{a}$:
\[\mathfrak{a}*(E,\iota) := (E_\mathfrak{a},(\varphi_{\mathfrak{a}})_*\iota),\]
where $E_\mathfrak{a}$ is the codomain of the degree-$N(\mathfrak{a})$ isogeny $\varphi_\mathfrak{a}:E\longrightarrow E_\mathfrak{a}$ with kernel $\cap_{\alpha\in\mathfrak{a}}\ker\alpha$. The orientation $(\varphi_{\mathfrak{a}})_*\iota:\mathfrak{O}\hookrightarrow\End(E_\mathfrak{a})$ is given via $(\varphi_{\mathfrak{a}})_*\iota(-) := \frac{1}{N(\mathfrak{a})}\varphi_\mathfrak{a}\circ\iota(-)\circ\widehat{\varphi_{\mathfrak{a}}}$.

For an imaginary quadratic field $K$, deciding if $K$ embeds into the quaternion algebra $B_{p,\infty}$ is the simple matter of the splitting behavior of $p$ in $K$. However, for a particular imaginary quadratic order $\mathfrak{O}$ and a particular supersingular elliptic curve $E$, it is generally difficult to decide if $E$ is $\mathfrak{O}$-orientable. Naturally we are inclined to study the following problems and the relationship between them: 

\begin{problem}[Decision $\mathfrak{O}$-Orienting Problem]\label{prob:OOrientingDecision}
Given an elliptic curve $E$ and an imaginary quadratic order $\mathfrak{O}$, determine if $E$ is orientable by $\mathfrak{O}$.
\end{problem}

\begin{problem}[$\mathfrak{O}$-Orienting Problem]\label{prob:OOrienting}
Given an elliptic curve $E$ oriented by an imaginary quadratic order $\mathfrak{O}$, find the orientation.
\end{problem}

We explore the following quaternion variant of Problem~\ref{prob:OOrienting} in Section~\ref{sec: quaternion}. 

\begin{problem}[Quaternion Order Embedding Problem]\label{Problem 4}

Given a maximal quaternion order $\mathcal{O}$ and an imaginary quadratic order $\mathfrak{O}$ which embeds into $\mathcal{O}$, find the embedding.
\end{problem}

One may also consider the group action variant of the Uber-isogeny problem, originally introduced in \cite{SETA}, although we do not pursue this perspective in this work:

\begin{problem}[$\mathfrak{O}$-Uber Isogeny Problem]
Given a supersingular elliptic curve $E$ with an $\mathfrak{O}$-orientation $\iota:\mathfrak{O}\hookrightarrow \End(E)$ and an $\mathfrak{O}$-orientable supersingular elliptic curve $F$, find an ideal $\mathfrak{a}\in\Cl(\mathfrak{O})$ such that $\mathfrak{a}*E = F$.
\end{problem}



\subsection{Computing modular polynomials and $j$-invariants}\label{sec: j-invariants}

Given a prime number $\ell\ll p$ and the $j$-invariant $j(E)\in\mathbb{F}_{p^2}$ of a supersingular elliptic curve, we explain how to find all $\ell$-isogenous $j$-invariants $j(E')\in\mathbb{F}_{p^2}$ using modular polynomials $\Phi_\ell(X,Y)$. By \cite[Theorem 6.3]{Rob_mod_pol}, $\Phi_\ell(j(E),Y)\in\mathbb{F}_{p^2}[Y]$ can be computed with $\tilde{O}(\ell^2\log(p))$ operations over $\mathbb{F}_{p^2}$, where the $\tilde{O}$ means that polynomial factors in $\log(\ell)$ are omitted\footnote{The algorithm is provided over $\mathbb{F}_p$ but the techniques of \cite{Rob_mod_pol} easily extend to $\mathbb{F}_{p^2}$.}. \cite[Section 5]{Leroux_mod_pol} also provides an algorithm with similar complexity. We then find all the roots over $\mathbb{F}_{p^2}$ of the degree-$(\ell+1)$ polynomial $\Phi_\ell(j(E),Y)$ in $\tilde{O}(\ell^2\log(p))$ operations over $\mathbb{F}_{p^2}$ \cite[Theorem 14.14]{Modern_comput_alg} to find the $j$-invariants $j(E')\in\mathbb{F}_{p^2}$ that are $\ell$-isogenous to $j(E)$. On the whole, the computation has time complexity $\tilde{O}(\ell^2\log(p))$.

\subsection{Computing an $\ell$-isogeny between two $j$-invariants.}\label{sec: isogenies from j-invariants}

Given two supersingular $j$-invariants $j(E)\in\mathbb{F}_{p^2}$ and $j(E')\in\mathbb{F}_{p^2}$ we explain how to find an $\ell$-isogeny $\phi: E\longrightarrow E'$ in $\tilde{O}(\ell^2\log(p))$ operations over $\mathbb{F}_{p^2}$. 

By \cite[Theorem 2]{Isogeny_comp}, given Weierstrass equations of $E$ and $E'$, we can find (if it exists) a normalized $\ell$-isogeny $\phi: E\longrightarrow E'$ with only $\tilde{O}(\ell)$ arithmetic operations over $\mathbb{F}_{p^2}$. By \emph{normalized}, we mean that $\phi$ pulls back the invariant differential $\omega':=dx'/2y'$ of $E'$ to the invariant differential $\omega:=dx/2y$ of $E$ ($\phi^*\omega'=\omega$).

The existence of such a normalized isogeny $\phi$ only depends on the choice of Weierstrass equations for $E$ and $E'$ which determine the constant $\lambda:=\phi^*\omega'/\omega$. Knowing only $j(E)$ and $j(E')$, we have multiple choices of Weierstrass equations and we have to pick one so that $\lambda=1$. We fix an equation for $E: y^2=x^3+Ax+B$, then find an equation for $E'$ so that $\lambda = 1$. Following the method given by \cite[Section~7]{Schoof_1995} (referring to ideas introduced in \cite[Section 3]{Elkies_1997}), we take $E':y^2=x^3+A'x+B'$, with
\begin{equation}A':=-\cfrac{j'(E')^2}{48j(E')(j(E')-1728)} \quad B':=-\cfrac{j'(E')^3}{864j(E')^2(j(E')-1728)},\label{eq: A' and B'}\end{equation}
and
\begin{equation}j'(E'):=-\cfrac{j'(E)}{\ell}\cfrac{\partial\Phi_\ell}{\partial X}(j(E), j(E'))\left(\cfrac{\partial\Phi_\ell}{\partial Y}(j(E), j(E'))\right)^{-1}\label{eq: j'(E')},\end{equation}
where
\begin{equation}
    j'(E):=\left\{\begin{array}{cr}
         \cfrac{18B j(E)}{A}& \mbox{ if } A\neq 0  \\
         0 & \mbox{ if } A=j(E)=0
    \end{array}
    \right.\label{eq: j'(E)}
\end{equation}
The derivatives $\partial\Phi_\ell/\partial X$ and $\partial\Phi_\ell/\partial Y$ can be precomputed with $\tilde{O}(\ell^2\log(p))$ operations over $\mathbb{F}_{p^2}$ using the techniques in Section \ref{sec: j-invariants} (see \cite[Remark 5.3.10]{Rob_HDR}). Hence, in total, computing an isogeny $\phi: E\longrightarrow E'$ costs $\tilde{O}(\ell^2\log(p))$ operations over $\mathbb{F}_{p^2}$ when $j(E')\neq 0, 1728$ and $\partial\Phi_\ell/\partial Y(j(E), j(E'))\neq 0$. 

The cases $j(E')=0, 1728$ are very unlikely (probability $O(1/p)$ in the supersingular isogeny graph). The latter case we split into two: First, suppose $\partial\Phi_\ell/\partial Y(j(E), j(E'))=0$ and $\partial\Phi_\ell/\partial X(j(E), j(E'))\neq0$. By symmetry of $\Phi_\ell(X,Y)$, we can fix a Weierstrass equation for $E'$ and find a normalized Weierstrass equation for $E$ by Equations \ref{eq: A' and B'}, \ref{eq: j'(E')} and~\ref{eq: j'(E)} (after swapping $E$ and $E'$). The remaining case is $\partial\Phi_\ell/\partial X(j(E), j(E'))=\partial\Phi_\ell/\partial Y(j(E), j(E'))=0$, i.e. when $(j(E),j(E'))$ is a singular point of the modular curve $\Phi_\ell(X,Y)=0$ over $\F_{p^2}$. Following \cite[Section~7]{Schoof_1995}, we prove in Appendix \ref{sec: modular singular points} that this is very unlikely when $\log(\ell)\ll\log(p)$ (which will be the case in our~paper).

We can still handle singular cases at a higher cost of $\tilde{O}(\ell^{7/2})$ operations over $\F_{p^2}$ with a naive algorithm. We enumerate all the cyclic subgroups of order $\ell$ of $E[\ell]$ (there are $\ell+1$ of them) and use \cite{SqrtVelu} to compute each $\ell$-isogeny with $O(\sqrt{\ell})$ operations over the field extension $K/\F_{p^2}$ where $E[\ell]$ is defined. As will be proved in Lemma \ref{lemma: field definition torsion}, $K$ has degree $O(\ell)$ over $\F_{p^2}$ so one arithmetic operation over $K$ is equivalent to at most $O(\ell^2)$ operations over $\F_{p^2}$. Since singular cases are very unlikely when $\log(\ell)\ll\log(p)$, we may assume throughout this paper that computing $\ell$-isogenies between $j$-invariants costs $\tilde{O}(\ell^2\log(p))$ operations over $\mathbb{F}_{p^2}$ on average by Lemma \ref{lemma: expected isogeny computation time}.


\subsection{Efficiently representing an isogeny of any degree with Kani's lemma}\label{sec: efficient rep}

Let $\varphi: E\longrightarrow E'$ be an isogeny of degree $d$ between supersingular elliptic curves $E,E'/\F_{p^2}$. In general, we can represent $\varphi$ with data of size $O(d)$. We either have direct formulas to evaluate $\varphi$ (given by rational fractions) or equivalently, generators of the kernel (defined over an $\F_{p^2}$-extension of degree $O(d)$) from which we can derive these formulas by \cite{Velu}. In this case, evaluating $\varphi$ on a point takes linear time in $d$. We can do much better when $d$ is smooth by representing $\varphi$ as a product of small degree isogenies. This is an efficient representation, in the sense of the following definition. 

\begin{definition}\cite[Definition 1.1.1]{SQISignHD}\label{def: efficient rep}
  An \emph{efficient representation} of an isogeny $\varphi: E\longrightarrow E'$ defined over a finite field $\mathbb{F}_q$ is given by a couple $(D,\mathscr{A})$ where:

\begin{description}
\item[(i)] $D$ is some data of size polynomial in $\log(\deg(\varphi))$ and $\log(q)$ determining the isogeny $\varphi$ in a unique way.

\item[(ii)] $\mathscr{A}$ is a universal algorithm independent of $\varphi$ returning $\varphi(P)$ as input $D$ and $P\in E(\mathbb{F}_{q^k})$ in polynomial time in $k\log(q)$ and $\log(\deg(\varphi))$.
\end{description}
\end{definition}

We can also efficiently represent $\varphi$ when $d$ is not smooth. Provided we can evaluate $\varphi$ on some torsion points, we know that we can ``embed" $\varphi$ in a smooth degree higher dimensional isogeny $F$. Knowing $F$, we can evaluate $\varphi$ everywhere in polynomial time. This provides an efficient representation of $\varphi$. This idea was first introduced in the attacks against SIDH \cite{CastryckDecruSIDH,Maino_et_al_SIDH,RobSIDH} and then reused for several other applications \cite{Rob_mod_pol,Robert2022evaluating,SQISignHD}.

\begin{definition}[$d$-isogeny in higher dimension]
Let $\alpha : (A,\lambda_A)\longrightarrow (B,\lambda_B)$ be an isogeny between principally polarized abelian varieties (PPAV). We denote by $\widetilde{\alpha}$ the isogeny
\[B\overset{\lambda_B}{\longrightarrow}\widehat{B}\overset{\widehat{\alpha}}{\longrightarrow}\widehat{A}\overset{\lambda_A^{-1}}{\longrightarrow}A,\]
where $\widehat{\alpha}$ is the dual isogeny of $\alpha$.

We say that $\alpha$ is a \emph{$d$-isogeny} if $\widetilde{\alpha}\circ\alpha=[d]_A$, or equivalently if $\alpha\circ\widetilde{\alpha}=[d]_B$.
\end{definition}

We use the following result due to Kani \cite[Theorem 2.3]{Kani1997}. A concise expression of this result may be found in \cite[Lemma 3.6]{RobSIDH}. 

\begin{lemma}[Kani]
    Consider a commutative diagram of isogenies between PPAV:
\[
\xymatrix{
A' \ar[r]^{\varphi'} & B'\\
A\ar[u]^{\psi} \ar[r]^{\varphi} & B \ar[u]_{\psi'}
}
\]
where $\varphi$ and $\varphi'$ are $a$-isogenies and $\psi$ and $\psi'$ are $b$-isogenies. 

Then, the isogeny $F: A\times B'\longrightarrow B\times A'$ given in matrix notation by
\[F:=\left(\begin{matrix}
\varphi & \widetilde{\psi'}\\
-\psi & \widetilde{\varphi'}
\end{matrix}\right)\]
is a $d$-isogeny with $d:=a+b$, for the product polarizations. 

If $a$ and $b$ are coprime, the kernel of $F$ is
\[\ker(F)=\{(\widetilde{\varphi}(x),\psi'(x))\mid x\in B[d]\}.\]
\end{lemma}

Let $N>d$ be a powersmooth integer coprime with $d$. We can always write $N=d+a_1^2+a_2^2+a_3^2+a_4^2$ for some $a_1, a_2, a_3, a_4\in\mathbb{Z}$, by Legendre's four square theorem. Let $\alpha\in\End(E^4)$ be the isogeny written in matrix form as follows:
\begin{equation}\alpha:=\left(\begin{array}{cccc}
 a_1 & -a_2 & -a_3 & -a_4  \\
 a_2 & a_1 & a_4 & -a_4 \\
 a_3 & -a_4 & a_1 & a_2 \\
 a_4 & a_3 & -a_2 & a_1
\end{array}\right),\label{eq: matrix alpha}\end{equation}
and $\alpha'$ be its analogue in $\End(E')$. Let $\Phi:=\Diag(\varphi,\varphi,\varphi,\varphi): E^4\longrightarrow E'^4$. Then, $\Phi$ is a $d$-isogeny, $\alpha$ and $\alpha'$ are $(N-d)$-isogenies and we have a commutative diagram:
\[\xymatrix{
E^4 \ar[r]^{\Phi} & E'^4 \\
E^4 \ar[u]^{\alpha} \ar[r]^{\Phi} & E'^4 \ar[u]^{\alpha'}
}
\]
that yields an $8$-dimensional $N$-isogeny:
\[F:=\left(\begin{array}{cc}
   \alpha  & \widetilde{\Phi} \\
    -\Phi & \widetilde{\alpha}
\end{array}\right)\in\End(E^4\times E'^4),\]
with kernel:
\begin{equation}\ker(F):=\{(\widetilde{\alpha}(P),\Phi(P))\mid P\in E^4[N]\},\label{eq: ker(F)}\end{equation}
since $N$ and $d$ are coprime. By the above formula, we can compute $\ker(F)$ if we can evaluate $\varphi$ on generators of $E[N]$. We can then compute $F$ as a product of small degree isogenies (as in dimension $1$) and evaluate $\varphi$ efficiently everywhere as a component of $F$. 

\begin{lemma}\label{lemma: decompose F}
    Let the setup be that of the previous paragraph, and suppose $N=\prod_{i=1}^s q_i^{e_i}$, where $q_1, \cdots, q_s$ are distinct primes. Then, we can decompose $F$ as:
    \[\mathcal{A}_0\overset{F_1}{\longrightarrow} \mathcal{A}_1 \overset{F_2}{\longrightarrow} \cdots\mathcal{A}_{s-1}\overset{F_s}{\longrightarrow}\mathcal{A}_s,\]
    with $\mathcal{A}_0=\mathcal{A}_s:=E^4\times E'^4$ and where $F_i$ is a $q_i^{e_i}$-isogeny for all $i\in\{1,\cdots, s\}$. 
    
    Let $K:=\ker(F)$. Moreover, 
    \[\ker(F_1)=K[q_1^{e_1}] = \{P\in K: [q_1^{e_1}] P = 0\} \quad \mbox{and}\] 
    \[\quad \ker(F_i)=F_{i-1}\circ\cdots\circ F_1(K[q_i^{e_i}]),\, \text{for } 2\leq i\leq s.\]
\end{lemma}

\begin{proof}
    The decomposition $F=F_s\circ\cdots\circ F_1$ was proven in \cite[Proposition 5.4.1]{SQISignHD}. 
    
    Now, let $K_1:=K[q_1^{e_1}]$ and $K_i:=F_{i-1}\circ\cdots\circ F_1(K[q_i^{e_i}])$ for all $2\leq i\leq s$. Then, $F_s\circ\cdots\circ F_i(K_i)=F(K[q_i^{e_i}])=\{0\}$, so that 
    \[F_i(K_i)\subseteq\ker(F_s\circ\cdots\circ F_{i+1})\subseteq\mathcal{A}_i\left[\prod_{j\geq i+1}q_j^{e_j}\right].\]
    But $F_i(K_i)\subset\mathcal{A}_i[q_i^{e_i}]$ and the primes $q_j\neq q_i$ for $j> i$, so we must have $F_i(K_i)=\{0\}$ and $K_i\subseteq\ker(F_i)$. Now, $\# K_i=\# K[q_i^{e_i}]=q_i^{8e_i}=\deg(F_i)$ since $F_{i-1}\circ\cdots\circ F_1$ has degree coprime with $q_i$. It follows that $K_i=\ker(F_i)$.
\end{proof}

Each of the $F_i$ in Lemma \ref{lemma: decompose F} can be decomposed into a chain of $q_i$-isogenies. As suggested in \cite[§ 5.5]{SQISignHD}, those isogeny chains can be computed in quasi-linear time using the optimal strategy introduced for SIDH \cite[§ 4.2.2]{SIDH}. Each $q_i$-isogeny of the chain can be computed in time $O(q_i^8)$ with the theta model \cite{DRfastisogenies}. We summarize this computation in Algorithm \ref{algo: efficient representation}.

\begin{algorithm}
\SetAlgoLined
\KwData{An isogeny $\varphi: E\longrightarrow E'$ between supersingular elliptic curves and a smoothness bound $D>0$.}
\KwResult{An $8$-dimensional isogeny $F$ of $D$-powersmooth degree representing $F$.}

Let $d:=\deg(\varphi)$\;

Select prime powers $q_1^{e_1}, \cdots, q_s^{e_s}\leq D$ coprime with $d$ such that $N:=\prod_{i=1}^s q_i^{e_i}>d$\;\label{line: coprime with degree}

Find $a_1, a_2, a_3, a_4\in\mathbb{Z}$ such that $a_1^2+a_2^2+a_3^2+a_4^2=N-d$ using Pollack and Trevi\~{n}o's algorithm \cite{Pollack2018}\;\label{line: find ais}

Let $\alpha\in\End(E)$ and $\alpha'\in\End(E')$ as in equation \ref{eq: matrix alpha} and $\Phi:=\Diag(\varphi,\varphi,\varphi,\varphi)$\; 

\For{i=1 \KwTo s}{\label{line: for start1}
Generate a basis $(P_{i,1},P_{i,2})$ of $E[q_i^{e_i}]$\;

Compute $\varphi(P_{i,1})$ and $\varphi(P_{i,2})$\;\label{line: phi(Pi12)}

For all $j\in\{1,2\}$ and $k\in\{1,2,3,4\}$, let $\underline{P}_{i,j,k}\in E^4[q_i^{e_i}]$ be the tuple with $P_{i,j}$ in position $k$ and $0$ elsewhere\;

$\mathcal{B}_i\longleftarrow\{(\widetilde{\alpha}(\underline{P}_{i,j,k}),\Phi(\underline{P}_{i,j,k}))\mid 1\leq j\leq 2, \ 1\leq k\leq 4\}$\;

}\label{line: for end1}

$\mathcal{C}_i\longleftarrow \mathcal{B}_i$ for $1\leq i\leq s$\;
\For{i=1 \KwTo s-1}{\label{line: for start2}
Compute $F_i$ of kernel $\langle \mathcal{C}_i\rangle$\;
\For{j=i+1 \KwTo s}{
$\mathcal{C}_j\longleftarrow F_i(\mathcal{C}_j)$\;\label{line: Update Cj}
}
}\label{line: for end2}
Compute $F_s$ of kernel $\langle \mathcal{C}_s\rangle$\;

Return $F:=F_s\circ\cdots\circ F_1$\;

\caption{\EfficientRep{} returning an efficient representation of an isogeny.}\label{algo: efficient representation}
\end{algorithm}

\begin{lemma}\label{lemma: field definition torsion}
    Let $E/\mathbb{F}_{p^2}$ be a supersingular elliptic curve and $n\in\mathbb{Z}_{>0}$. Then $E[n]$ is defined over an extension of degree at most $6\phi(n)$ of $\mathbb{F}_{p^2}$, where $\phi$ is Euler's totient function.
\end{lemma}

\begin{proof}
    We first compute the characteristic polynomial of iterates of the Frobenius. Let $\chi_{p^2}:=(X-\alpha)(X-\beta)$ be the characteristic polynomial of the $p^2$ Frobenius $\pi_{p^2}$. Then $\chi_{p^{2\delta}}:=(X-\alpha^\delta)(X-\beta^\delta)$ is the characteristic polynomial of the $p^{2\delta}$ Frobenius $\pi_{p^{2\delta}}$, for any $\delta\in\mathbb{Z}_{>0}$. 

     Since $E$ is supersingular $p\mid\Tr(\pi_{p^2})=\alpha+\beta$, and $|\Tr(\pi_{p^2})|\leq 2p$ so $\Tr(\pi_{p^2})\in\{0,\pm p,\pm 2p\}$ \cite[Proposition 3.6]{Schoof_1995}. We consider these possibilities in three cases: 

     If $\Tr(\pi_{p^2})=0$, then $\chi_{p^2}=X^2+p^2=(X-ip)(X+ip)$ so
     \[\Tr(\pi_{p^{2\delta}})=(ip)^\delta+(-ip)^\delta=\left\{ 
     \begin{array}{cl}
        2p^\delta  & \mbox{if } \delta\equiv 0 \mod 4 \\
        -2p^\delta & \mbox{if } \delta\equiv 2 \mod 4 \\
        0 & \mbox{otherwise}
     \end{array}\right.\]

     If $\Tr(\pi_{p^2})=\pm p$, then $\chi_{p^2}=X^2\mp pX+p^2=(X\mp pe^{i\pi/3})(X\mp pe^{-i\pi/3})$ so
     \[\Tr(\pi_{p^{2\delta}})=(\pm 1)^\delta p^\delta(e^{i\delta\pi/3}+e^{-i\delta\pi/3})=\left\{ 
     \begin{array}{cl}
        2p^\delta  & \mbox{if } \delta\equiv 0 \mod 6 \\
        \mp 2p^\delta & \mbox{if } \delta\equiv 3 \mod 6 \\
        \pm p^\delta & \mbox{if } \delta\equiv \pm 1 \mod 6 \\
        -p^\delta & \mbox{if } \delta\equiv \pm 2 \mod 6 
     \end{array}\right.\]

     Finally, if $\Tr(\pi_{p^2})=\pm 2p$, then $\chi_{p^2}=X^2\mp 2pX+p^2=(X\mp p)^2$ so
     \[\Tr(\pi_{p^{2\delta}})=2(\pm p)^\delta \]
     
     In all cases, if $\delta\equiv 0 \mod 12$, we have $\Tr(\pi_{p^{2\delta}})=2p^\delta$, so $\chi_{p^{2\delta}}=(X-p^\delta)^2$. It follows that $\pi_{p^{2\delta}}=[p^\delta]$.

     Now, if we assume $\phi(n) \mid \delta$, then $p^\delta\equiv 1 \mod n$. For all $P,Q\in E[n]$:
     \[e_n(\pi_{p^{2\delta}}(P),Q)=e_n([p^\delta]P,Q)=e_n(P,Q)^{p^{\delta}}=e_n(P,Q).\]
     In particular, for all $P\in E[n]$, $e_n(\pi_{p^{2\delta}}(P)-P,Q)=1$ for every $Q\in E[n]$. By properties of the Weil pairing, $\pi_{p^{2\delta}}(P)-P=\mathcal{O}_E$ and $P\in E(\mathbb{F}_{p^{2\delta}})$. 
     
     Hence, $E[n]$ is defined over $\mathbb{F}_{p^{2\delta}}$ provided that $12$ and $\phi(n)$ divide $\delta$. If $n$ has an odd prime factor or $n=2^k$ with $k\geq 2$, then $\phi(n)$ is even so $12 \mid 6\phi(n)$ so $\delta=6\phi(n)$ satisfy the desired conditions. If $n=2$, then $E[n]$ is formed by $0$ and the $(x,0)$ where $x$ is the root of a cubic polynomial equation over $\mathbb{F}_{p^2}$ (Weierstrass equation of $E$). Hence,  $E[2]$ is defined over an extension of degree at most $3$ of $\mathbb{F}_{p^2}$. 
\end{proof}

\begin{proposition}\label{prop: complexity efficient rep}
    Algorithm \ref{algo: efficient representation} terminates and is correct. It requires $O(\log(d))$ evaluations of the input $d$-isogeny $\varphi$ on points defined over an extension of degree $O(D)$ of $\mathbb{F}_{p^2}$ and 
    \[O(D^3\log^3(p)\log(d)+D^{10}M(p)\log^2(d))\] 
    elementary (bitwise) arithmetic operations, where $M(p)$ is the complexity of the multiplication over $\mathbb{F}_p$.
\end{proposition}

\begin{proof}
    Correctness has been justified by Equation \ref{eq: ker(F)} and Lemma \ref{lemma: decompose F}. Termination is clear. 

    Now we compute the complexity. If $N$ is only slightly bigger than $d$, then $s=O(\log(d))$ and finding a suitable $N$ on line~\ref{line: coprime with degree} of the algorithm takes $O(\log(d))$ arithmetic operations. 
    
    Finding the $a_i$ on line~\ref{line: find ais} costs $O(\log(N)/\log\log(N))=O(\log(d))$ with Pollack and Trevi\~{n}o's algorithm \cite{Pollack2018}.

    For $i\in\{1,\cdots, s\}$, a basis of $E[q_i^{e_i}]$ is defined over a field extension of degree $\delta=O(q_i^{e_i})=O(D)$ of $\mathbb{F}_{p^2}$ by Lemma \ref{lemma: field definition torsion}. To generate such a basis, we first sample a random point $P'\in E(\mathbb{F}_{p^{2\delta}})$ and compute $P_{i,1}:=[M/q_i^{e_i}]P'$, where $M:=\# E(\mathbb{F}_{p^{2\delta}})$ until $P$ has order $q_i^{e_i}$. By the same method, we sample $P_{i,2}\in E[q_i^{e_i}]$ until $(P_{i,1},P_{i,2})$ is a basis of $E[q_i^{e_i}]$.
    
    To sample $P'\in E(\mathbb{F}_{p^{2\delta}})$, we first sample $x\in\mathbb{F}_{p^{2\delta}}$ in time $O(D\log(p))$ repeatedly ($O(1)$ times at most) until we find $y\in\mathbb{F}_{p^{2\delta}}$ such $P'=(x,y)\in E$. Computing $y$ requires a square root computation in $\mathbb{F}_{p^{2\delta}}$ which costs $O(\log^3(p^{2\delta}))=O(D^3\log^3(p))$ elementary arithmetic operations by Cipolla-Lehmer's algorithm. $E$ being supersingular, point counting on $E$ to compute $M$ costs $O(1)$ and the scalar multiplication by $M/q_i^{e_i}$ costs $O(\log(M))=O(D\log(p))$ arithmetic operations over $\mathbb{F}_{p^{2\delta}}$ (costing $O(D^2 M(p))$ each). Testing that $(P_{i,1},P_{i,2})$ is a basis costs $O(D)$ elliptic curve additions so $O(D)$ arithmetic operations over $\mathbb{F}_{p^{2\delta}}$ (we compute the $[k]P_{i,1}$ and $[l]P_{i,2}$ for $1\leq k, l\leq q_i^{e_i}-1$ and conclude that we have a basis if these two sets are disjoint). Only $O(1)$ samplings of $P_{i,1}$ and $P_{i,2}$ are necessary before we find a basis. Hence, the overall complexity of the basis computation is
    \[O(D^3\log^3(p)+D^3M(p)\log(p)+D^3M(p))=O(D^3\log^3(p)).\]

    Line~\ref{line: phi(Pi12)} costs two evaluations of $\varphi$ and line $9$ costs eight scalar multiplications by the $a_i$, costing $O(\log(d))$ operations over $\mathbb{F}_{p^{2\delta}}$ each. Hence, the total cost of the loop of lines~\ref{line: for start1}--\ref{line: for end1} is 
    \[O(s(D^3\log^3(p)+D^2M(p)\log(d)))=O(\log(d)(D^3\log^3(p)+D^2M(p)\log(d)))\]
    elementary arithmetic operations and $2s=O(\log(d))$ evaluations of $\varphi$.

    Finally, computing each $F_i$ costs $O(e_iq_i^8)=O(D^8)$ arithmetic operations over $\mathbb{F}_{p^2}$ and computing the basis $\mathcal{C}_j$ ($i+1\leq j\leq s$) on line~\ref{line: Update Cj} costs $8(s-i)$ point evaluations, costing $O(e_iq_i^8)=O(D^8)$ arithmetic operations over an extension of degree $O(D)$ of $\mathbb{F}_{p^2}$. The total cost of the loop of lines~\ref{line: for start2} --\ref{line: for end2} is 
    \[O(sD^8M(p)+s^2D^{10}M(p))=O(D^{10}M(p)\log^2(d)).\]
\end{proof}

\subsection{Smoothness test and factorization with the ECM method}\label{sec: ECM}

In this section, we explain how to test if an integer $N$ is $B$-smooth and find its factorization if it is the case. 
A naive method would be to use trial division, but it is not optimal when $B$ is subexponential (which will be the case in the paper). An alternate method would be to factor $N$ with the General Number Field Sieve (GNFS) \cite{GNFS} and test if its prime factors are $\leq B$. 
However, GNFS underperforms with smooth integers so we propose a faster method relying on the elliptic-curve factorization method (ECM) due to Lenstra \cite{Lenstra_ECM}. 
Finding a prime factor of $N$ with this method takes time $L_\ell(1/2,\sqrt{2})$, where $\ell$ is the smallest prime divisor of $N$ and with the usual notation
 \[L_x(\alpha,\beta):=\exp\left((\beta+o(1))(\log(x))^\alpha(\log\log(x))^{1-\alpha}\right),\]
where $o(1)$ is for $x \to \infty$. 
Hence, to test the $B$-smoothness of $N$, we simply apply ECM to find a factor $k \mid N$ after expected time $L_B(1/2,\sqrt{2})$. If the running time exceeds what it should be, it means that $N$ is not $B$-smooth and we stop. Otherwise, we continue and try to factor $k$ and $N/k$ recursively until we have either completely factored $N$ or concluded it is not $B$-smooth. Algorithm \ref{algo: SmoothFact} follows.

\begin{algorithm}
\SetAlgoLined
\KwData{An integer $N\in\mathbb{Z}_{>0}$ and a smoothness bound $B>0$.}
\KwResult{$\bot$ if $N$ is not $B$-smooth, and primes $\ell_1,\cdots, \ell_r\leq B$ such that $N=\prod_{i=1}^r\ell_i$ otherwise.}
\eIf{$N$ is prime}{
\eIf{$N\leq B$}{
Return $N$\;
}{
Return $\bot$\;
}
}{
 Use ECM to find a strict divisor $k \mid N$ in time $L_B(1/2,\sqrt{2})$\;
 \eIf{ECM does not terminate in time $L_B(1/2,\sqrt{2})$}{
 Return $\bot$\;}{
 $R\longleftarrow \SmoothFact{}(k,B), R'\longleftarrow \SmoothFact{}(N/k,B)$\;
 \eIf{$R=\bot$ or $R'=\bot$}{
 Return $\bot$\;
 }{
 Return $R\cup R'$\;
 }
 }
}

\caption{$\SmoothFact{}$ determining if an integer is smooth and returning its factorization.}\label{algo: SmoothFact}
\end{algorithm}

If $r$ is the number of prime divisors of $N(\theta)$ (with multiplicity), then $r=O(\log(|\Delta_{\mathfrak{O}}|))$ and Algorithm \ref{algo: SmoothFact} can terminate with at most $r$ calls to ECM so in time $rL_B(1/2,\sqrt{2})=L_B(1/2,\sqrt{2})$.

\section{Reduction of $\mathfrak{O}$-orienting problem for special discriminants}\label{sec: smooth disc}


We begin with a nice assumption about the discriminant of our imaginary quadratic order $\mathfrak{O}$. The key ideas from this special case provide a foundation for the general cases we consider in Section~\ref{sec: decision prob with oracle}.

Let $d = \prod_{i=1}^r\ell_i$ be a product of small distinct primes. Let $\mathfrak{O}$ denote the maximal order of $K:=\Q(\sqrt{-d})$, so $\Delta_\mathfrak{O}=-d$ if $d\equiv -1 \mod 4$ and $-4d$ otherwise. In particular, $(\Delta_\mathfrak{O}/\ell_i) = 0$ for all $i = 1,\dots,r$ and $\mathfrak{O}$ is generated by $\omega:=(1+\sqrt{-d})/2$ if $d\equiv -1 \mod 4$ and by $\omega:=\sqrt{-d}$ otherwise. 
Hence, $\alpha:=\sqrt{-d}$ generates $\mathfrak{O}$ if $d\not\equiv -1 \mod 4$ or $(\mathbb{Z} + 2\mathfrak{O})$ if $d\equiv -1 \mod 4$. 
We use an oracle which solves Problem~\ref{prob:OOrientingDecision} to find an endomorphism $\varphi$ of $E$ to which we map $\alpha$, thus determining an embedding $\mathfrak{O}\hookrightarrow\End(E)$ either by mapping $\alpha=\omega$ to $\varphi$ or $(1+\alpha)/2=\omega$ to $(1+\varphi)/2$.  We use the fact that the primes $\ell_i$ are ramified in $K$.

We walk the $\mathfrak{O}$-oriented $\ell_i$-isogeny volcanoes in order to obtain the endomorphism $\varphi$ on $E$ which is the image of the generator $\omega$ under an embedding $\iota:\mathfrak{O}\hookrightarrow\End(E)$. 
The ideals $\mathfrak{l}_i$ above $\ell_i$ in $\mathfrak{O}$ determine horizontal degree-$\ell_i$ isogenies between $\mathfrak{O}$-oriented curves, beginning and ending with $E$. 
To see this, we need the following fact about horizontal isogenies of $\mathfrak{O}$-oriented elliptic curves:

\begin{proposition}\cite[Proposition 4.1]{OSIDH_Onuki}\label{prop: volcano structure}
Let $(E,\iota)$ be an $\mathfrak{O}$-oriented supersingular elliptic curve and $\ell$ be a prime number. Then:
\begin{description}
\item[(i)] If $\ell$ does not divide the conductor of $\mathfrak{O}$, there is no ascending, $(\Delta_{\mathfrak{O}}/\ell)+1$ horizontal and $\ell-(\Delta_{\mathfrak{O}}/\ell)$ descending $\ell$-isogenies.
\item[(ii)] If $\ell$ divides the conductor of $\mathfrak{O}$, there is one ascending, no horizontal and $\ell$ descending $\ell$-isogenies.
\end{description}
\end{proposition}

Let $\iota:\mathfrak{O}\hookrightarrow\End(E)$ be an orientation and $\varphi:=\iota(\alpha)$. Then $\deg(\varphi)=N(\alpha)=\prod_{i=1}^r\ell_i$ so we may write $\varphi:=\varphi_r\circ\cdots\circ\varphi_1$, where $\varphi_i$ is an isogeny of degree $\ell_i$ for all $i\in\{1,\cdots,r\}$. For each $i$, let $\mathfrak{O}\ell_i = (\mathfrak{l}_{i})^2$. The ideals $\mathfrak{l}_i$ determine the horizontal isogenies of $\mathfrak{O}$-oriented curves: 

\begin{lemma}\label{lemma: horizontal decomp}
    In the setting described above, all of the isogenies $\varphi_i$ in the decomposition of $\varphi$ are horizontal.
\end{lemma}

\begin{proof}
    Since $N(\alpha)=\prod_{i=1}^r\ell_i$, we have $\mathfrak{O}\alpha=\prod_{i=1}^r\mathfrak{l}_i$, $\mathfrak{l}_i$ being the unique prime ideal of $\mathfrak{O}$ lying above $\ell_i$ for all $i\in\{1,\cdots, r\}$. Hence, the $\varphi_i$ intervening in the decomposition of $\varphi=\iota(\alpha)$ are horizontal isogenies given by the action of $\mathfrak{l}_i$.
\end{proof}

Now, we describe the steps to obtain an endomorphism $\varphi=\varphi_r\circ\cdots\circ\varphi_1\in\End(E)$ which will be the image of $\alpha$. Let $E_0:= E$.

For $i=0$, we find the unique isogeny $\varphi_1:E_0\longrightarrow E_1$ which corresponds to the action of $[\mathfrak{l}_1]$ on $(E_0,\iota)$ by computing each of the $\ell_1 + 1$ outgoing isogenies and querying our oracle to find the one whose codomain $E_1$ is in fact orientable by $\mathfrak{O}$. We continue this process to compute each $\varphi_i$ by using the oracle to find the correct degree-$\ell_i$ isogeny to another $\mathfrak{O}$-orientable curve. At the last step, we compute the degree-$\ell_r$ isogeny from $E_{r-1}\longrightarrow E_r$ and then post-compose with an isomorphism $E_r\cong E_0$: We let $\varphi_r$ denote this composition. See Algorithm~\ref{algo: special reduction search to decision} for the algorithmic description of this process.

\begin{example}\label{ex: non unique orientation}
    Consider $p=41$ and $E_0:y^2=x^3+1$ defined over $\mathbb{F}_{41^2}=\mathbb{F}_{41}[\zeta]$ with $\zeta^2+\zeta+1=0$. Consider the Frobenius endomorphism $\pi:(x,y)\mapsto (x^p,y^p)$ and the automorphism $\tau: (x,y)\mapsto (\zeta x,y)$. Then $\varphi:=\pi+\tau$ satisfies the polynomial equation $\varphi^2+\varphi+42=0$ so it defines an orientation of the maximal order $\mathcal{O}_K:=\mathbb{Z}\left[(1+\sqrt{-167})/2\right]$ of the imaginary quadratic field $K:=\mathbb{Q}(\sqrt{-167})$, mapping $(1+\sqrt{-167})/2$ to $\varphi$. 
    
    The prime ideal $2$ splits in $K$ so there are two horizontal $2$-isogenies and one descending $2$-isogeny with domain $E_0$. However, all these isogenies have the same codomain $E_1$ (up to isomorphism) with $j$-invariant $j(E_1)=3$. So $E_1$ is both $\mathcal{O}_K$-oriented and $(\mathbb{Z}+2\mathcal{O}_K)$-oriented. 
\end{example}

Motivated by this example a question arises: 
If $\psi_i: E_{i-1}\longrightarrow E_i'$ is an $\ell_i$-isogeny with $E_i'$ $\mathfrak{O}$-orientable, how do we know that $\psi_i$ is the unique horizontal isogeny $\varphi_i$ given by the action of $\mathfrak{l}_i$ on $(E_{i-1},\iota)$? In fact, $\varphi_i$ and $\psi_i$ could be distinct horizontal isogenies for distinct primitive orientations $(E_{i-1},\iota)\neq (E_{i-1},\iota')$ (as in Example \ref{ex: non unique orientation}). Or $\psi_i$ could even be descending and $E_i'$ $\mathfrak{O}$-oriented by a different orientation than the one induced by $\psi_i$, $(E_i,(\psi_i)_*(\iota))$. 
To exclude those cases, we assume $p>|\Delta_{\mathfrak{O}}|\max_{1\leq i\leq r}\ell_i$ and  prove that there is precisely one (primitive) $\mathfrak{O}$-orientation $\iota$ on $E_{i-1}$, which ensures that there is only one isogeny $\varphi_i$ corresponding to the action of $[\mathfrak{l}_i]$ on $(E_{i-1},\iota)$. 
We also prove that codomains of descending isogenies are not $\mathfrak{O}$-orientable. These are consequences of \cite[Theorem 2']{Kaneko1989}, that we recall below.

\begin{theorem}\cite[Theorem 2']{Kaneko1989}\label{lemma: disc ineq}
    Let $\mathcal{O}\subset B_{p,\infty}$ be a maximal order in the quaternion algebra ramifying at $p$ and $\infty$. Let $j_i:\mathfrak{O}_i\lhook\joinrel\longrightarrow\mathcal{O}$ ($i\in\{1,2\}$) be two primitive embeddings of orders in the same imaginary quadratic field $K:=\mathbb{Q}\otimes\mathfrak{O}_1=\mathbb{Q}\otimes\mathfrak{O}_2$ of respective discriminants $\Delta_i$. Assume that $j_1(\mathfrak{O}_1)\neq j_2(\mathfrak{O}_2)$. Then $\Delta_1\Delta_2\geq p^2$.
\end{theorem}

\begin{corollary}\label{cor: unique orientation}
    Let $(E,\iota)$ be a (primitively) $\mathfrak{O}$-oriented curve. Then
    \begin{description}
        \item[(i)] If $|\Delta_\mathfrak{O}|<p$, $\iota$ and $\overline{\iota}: \alpha\longmapsto \iota(\overline{\alpha})$ are the only two (primitive) $\mathfrak{O}$-orientations of $E$.
        \item[(ii)] If $|\Delta_\mathfrak{O}|\ell<p$ and $\psi:(E,\iota)\longrightarrow (E',\iota')$ is a descending $\ell$-isogeny, then $E'$ is not $\mathfrak{O}$-orientable. 
    \end{description}
\end{corollary}

\begin{proof}
    \textbf{(i)} Let $\iota':\mathfrak{O}\hookrightarrow\End(E)$ be another $\mathfrak{O}$-orientation of $E$. Since $|\Delta_\mathfrak{O}|<p$, we must have $\iota'(\mathfrak{O})=\iota(\mathfrak{O})$ by Lemma \ref{lemma: disc ineq}. Hence, $\iota'^{-1}\circ\iota$ is an automorphism of $\mathfrak{O}$, so it is either the identity or the complex conjugation. The result follows.

    \textbf{(ii)} Suppose $E'$ is $\mathfrak{O}$ orientable and let $(E',\iota'')$ be an $\mathfrak{O}$-orientation. Let $\mathfrak{O}':=\Z+\ell\mathfrak{O}$. Then $\psi:(E,\iota)\longrightarrow (E',\iota')$ being descending, $(E',\iota')$ is an $\mathfrak{O}'$-orientation and $\iota'(\mathfrak{O}')\neq\iota''(\mathfrak{O})$, so that $\Delta_{\mathfrak{O}'}\Delta_\mathfrak{O}\geq p^2$ by Lemma \ref{lemma: disc ineq}. But $\Delta_{\mathfrak{O}'}\Delta_\mathfrak{O}=\ell^2\Delta_\mathfrak{O}^2<p^2$ by hypothesis. Contradiction.
\end{proof}

\begin{remark}
    Corollary~\ref{cor: unique orientation} holds for any imaginary quadratic order $\mathfrak{O}$, not only the special form we consider in this section.
\end{remark}

Assuming $p>|\Delta_{\mathfrak{O}}|\max_{1\leq i\leq r}\ell_i$, the orientation $\iota:\mathfrak{O}\hookrightarrow\End(E)$ is unique up to conjugation, the horizontal $\ell_1$-isogeny $\varphi_1:E_0\longrightarrow E_1$ given by the action of $\mathfrak{l}_1$ is uniquely determined, and it is the only $\ell_1$-isogeny with $\mathfrak{O}$-oriented codomain. In this case, $\varphi_1$ can be distinguished from other $\ell_1$-isogenies by an oracle query. 
Similarly for each further $i \in \{2,...,r\}$, the isogeny $\varphi_i:E_{i-1}\longrightarrow E_i$ given by the action of $[\mathfrak{l}_i]$ on $(E_{i-1},(\varphi_{i-1}\circ\cdots\circ\varphi_1)_*(\iota))$ by computing each of the $\ell_i + 1$ isogenies and querying the oracle to find the one whose codomain $E_i$ is orientable by $\mathfrak{O}$. In particular, the isogeny $\varphi_r:E_{r-1}\longrightarrow E_r$ corresponding to the action of $\mathfrak{l}_r$ on $(E_{r-1},(\varphi_{r-1}\circ\cdots\circ\varphi_1)_*(\iota))$ will have codomain $E_r \cong E_0$. Indeed, 
\[(E_{r},(\varphi_{r}\circ\cdots\circ\varphi_1)_*(\iota))=[\mathfrak{l}_1\cdots\mathfrak{l}_r]\cdot (E_0,\iota)=[\alpha\mathfrak{O}]\cdot (E_0,\iota)\cong (E_0,\iota)\]
Possibly post-composing with this isomorphism, we have an endomorphism $\varphi=\varphi_r\circ\cdots\circ\varphi_1\in \End(E)$ associated to the action of the ideal $\prod_{i=1}^r\mathfrak{l}_i=\alpha\mathfrak{O}$. It follows that $\varphi=\tau\circ\iota(\alpha)$ for some automorphism $\tau\in\Aut(E)$. We may post-compose $\varphi$ by $\tau\in\Aut(E)$ until the result has trace zero, as $\alpha$. The trace can be computed in polynomially many isogeny evaluations using Schoof's algorithm \cite[Section 5]{Schoof_1995}. 

\begin{remark}[Isomorphisms]
    Assuming we are working with elliptic curves in Weierstrass form, all isomorphism formulae are known. To find an isomorphism $\beta : E_r \longrightarrow E_0$, we check the codomain formula for each isomorphism from $E_r$ until $E_0$ is found.

    There additional automorphisms in the two special cases of $j=1728$ and $j=0$ \cite[Figure 3.1, Section 6]{ACLSST2022_orientations}. At each step $\varphi_i:E_{i-1}\to E_i$ where $j(E_i) = 0$ or $1728$, we must decide whether or not to post-compose with these automorphisms. The automorphisms $[\pm 1]$ will not affect the resulting trace, but we must check one nontrivial automorphism for $j = 1728$ and two for $j = 0$. This can be done after the algorithm is completed, as the oracle calls will remain unaffected.  

    The additional running time of choosing isomorphisms can be bounded by a constant, so does not contribute to the overall complexity.
\end{remark}

\begin{example}
    Let $p = 83$ and $\mathfrak{O} = \mathbb{Z}[\sqrt{-21}]$, the ring of integers of $K = \mathbb{Q}(\sqrt{-21})$. We see $p$ is inert in $\mathfrak{O}$ so $K$ embeds into the quaternion algebra $\End^0(E)$ for any supersingular $E$ over $\overline{\mathbb{F}_p}$. Now, let $E/\mathbb{F}_p^2$ be the $\mathfrak{O}$-oriented curve $y^2 = x^3 + x$, we find the orientation by finding an endomorphism $\omega$ with $N(\omega) = 21 = 3 \cdot 7$ and $Tr(\omega) = 0$. From $E$ we pick a $3$-isogeny to $y^2 = x^3 + 32x + 38\sqrt{-1}$, this is also $\mathfrak{O}$-oriented. Then we pick a horizontal $7$-isogeny which has codomain $y^2 = x^3 + 26x$. This curve is isomorphic to $E$. By composing maps we get $\omega : E \longrightarrow E$. Finally we notice $\omega \neq - \tilde{\omega}$ so the endomorphism has a non-zero trace. But by post-composing with an automorphism $\iota$ on $E$, we get a trace-zero endomorphism of degree $21$.
\end{example}

If $\alpha$ is a generator of $\mathfrak{O}$ ($d\not\equiv -1 \mod 4$), then $\varphi$ determines an $\mathfrak{O}$-orientation and we are done. Otherwise, $\Z[\alpha]$ has index $2$ in $\mathfrak{O}$ ($d\equiv -1 \mod 4$), and $\omega=(1+\alpha)/2$ generates $\mathfrak{O}$. Then $\varphi$ determines an imprimitive $\Z[\alpha]$-orientation of $E$. This orientation cannot be primitive, otherwise, we would have $\Delta_\mathfrak{O}\Delta_{\Z[\alpha]}\geq p^2$ i.e. $4\Delta_\mathfrak{O}^2\geq p^2$, which is a contradiction since we assumed that $p>|\Delta_{\mathfrak{O}}|\max_{1\leq i\leq r}\ell_i\geq 2|\Delta_{\mathfrak{O}}|$. It follows that $(\varphi+1)/2$ is well defined and induces an $\mathfrak{O}$-orientation on $E$: $\omega=(\alpha+1)/2\longmapsto (\varphi+1)/2$. 

\begin{remark}[Efficient representation]
    Knowing how to evaluate $\varphi$ (as the composition $\varphi_r\circ\cdots\circ\varphi_1$), we efficiently evaluate $(\varphi+1)/2$ as follows: if $P\in E(\F_{p^k})$, we find $P'\in E(\F_{p^{2k}})$ such that $[2]P'=P$ and compute $\varphi(P')+P'$. Assuming the $\ell_i$ are polynomial in $\log(d)$, the list of isogenies $(\varphi_r, \cdots, \varphi_1)$ defines an efficient representation of both $\varphi$ and $(\varphi+1)/2$.
\end{remark}

We summarize all the steps to determine an $\mathfrak{O}$-orientation in Algorithm \ref{algo: special reduction search to decision}. 

\begin{algorithm}
\SetAlgoLined
\KwData{A supersingular elliptic curve $E_0/\mathbb{F}_{p^2}$; the maximal order $\mathfrak{O}$ of $\Q(\sqrt{-d})$, where $d:=\prod_{i=1}^r\ell_i$ is a product of small distinct primes, where $p>|\Delta_{\mathfrak{O}}|\max_{1\leq i\leq r}\ell_i$; an oracle $\IsOrientable{\mathfrak{O}}$ for the Decision $\mathfrak{O}$-Orienting Problem \ref{prob:OOrientingDecision}.}
\KwResult{If $E_0$ is $\mathfrak{O}$-orientable, an efficient representation (as defined in \ref{def: efficient rep}) of an endomorphism $\varphi_0\in\End(E_0)$ defining an $\mathfrak{O}$-orientation of~$E_0$.}

\If{\textbf{not} \IsOrientable{\mathfrak{O}}($E_0$)}{
Return ``$E_0$ is not $\mathfrak{O}$-orientable"\;
}
Endo $\longleftarrow []$\;
\For{$i=1$  \KwTo $r$}{
Compute the set $\{E_{{i-1},k}\}_{k=1}^{\ell_i+1}$ of codomains of the $(\ell_i + 1)$ degree-$\ell_i$ isogenies from $E_{i-1}$\;\label{alg:special_disc:isog_list}

Looking $\longleftarrow$ True\;
$k \longleftarrow 1$\;

\While{Looking \textbf{and} $k\leq (\ell_i+1)$}{
\If{\IsOrientable{\mathfrak{O}(E_{{i-1},k})}}{
$E_i \longleftarrow E_{{i-1},k}$\;
Compute the degree-$\ell_i$ isogeny $\varphi_{i}:E_{i-1} \longrightarrow E_i$\;
Append $\varphi_{i}$ to Endo\;\label{alg:special_disc:isog_find}
Looking $\longleftarrow$ False\;
}
$k\longleftarrow k+1$\;
}
Test all isomorphisms $\beta: E_r\longrightarrow E_0$ until $\beta\circ\varphi_r\circ\cdots\circ\varphi_1$ has trace zero\;\label{line: trace}
Replace $\varphi_r$ by $\beta\circ\varphi_r$ in Endo\;
}
Return Endo\;

\caption{Algorithm to solve the $\mathfrak{O}$-Orienting Problem \ref{prob:OOrienting} with an oracle for the Decision $\mathfrak{O}$-Orienting Problem \ref{prob:OOrientingDecision}, special discriminant.}\label{algo: special reduction search to decision}
\end{algorithm}

\begin{theorem}\label{thm: special_red}
    Let $d:=\prod_{i=1}^r \ell_i$ be a product of small distinct primes, $\mathfrak{O}$ be the maximal order of $\Q(\sqrt{-d})$ and $p>|\Delta_\mathfrak{O}|\max_{1\leq i\leq r}\ell_i$. Then, over $\F_{p^2}$, Algorithm \ref{algo: special reduction search to decision} reduces the $\mathfrak{O}$-Orienting Problem (Problem \ref{prob:OOrienting}) to the Decision $\mathfrak{O}$-Orienting Problem (Problem \ref{prob:OOrientingDecision}) in polynomial time in $\log(p)$ and $\max_{1\leq i\leq r}\ell_i$.
\end{theorem}

\begin{proof}
 We justified above that this algorithm terminates and is correct. For all $i\in\{1,...,r\}$, this algorithm computes the $\ell_i+1$ curves which are $\ell_i$-isogenous to $E_{i-1}$, which costs $\tilde{O}(\ell_i^2\log(p))$ operations over $\F_{p^2}$ by Section \ref{sec: j-invariants}. It calls the oracle $\IsOrientable{\mathfrak{O}}$ $\ell_i+1$ times and computes one $\ell_i$-isogeny between $j(E_{i-1})$ and $j(E_{i})$, which costs on average $\tilde{O}(\ell_i^2\log(p))$ operations over $\F_{p^2}$ by Section \ref{sec: isogenies from j-invariants}. The number of isomorphisms $\beta:E_r\longrightarrow E_0$ is $O(1)$. Using \cite[Section~5]{Schoof_1995}, we compute the trace of $\beta\circ\varphi_r\circ\cdots\circ\varphi_1$ on line \ref{line: trace} of Algorithm \ref{algo: special reduction search to decision} in polynomial time in $\log(p)$, $r=O(\log(p))$ and $\max_{1\leq i\leq r}\ell_i$. The total cost is polynomial in $\log(p)$ and $\max_{1\leq i\leq r}\ell_i$.
\end{proof}

\section{Solving the $\mathfrak{O}$-orienting problem with a decision oracle}\label{sec: decision prob with oracle}

\subsection{Description of the algorithms}

Let $\mathfrak{O}$ be an imaginary quadratic order with general discriminant $\Delta_\mathfrak{O}$. Given access to an oracle $\IsOrientable{\mathfrak{O}}$ for the Decision $\mathfrak{O}$-Orienting Problem \ref{prob:OOrientingDecision}, we solve the $\mathfrak{O}$-Orienting Problem \ref{prob:OOrienting} finding a an $\mathfrak{O}$-orientation $\iota : \mathfrak{O}\hookrightarrow\End(E)$ of any given supersingular elliptic curve $E/\F_{p^2}$ if it exists. 

The idea is similar the case of special discriminant considered in Section \ref{sec: smooth disc}. We compute an endomorphism corresponding to a generator  of $\mathfrak{O}$ as a chain of horizontal isogenies of small degrees. However, two difficulties arise. First, the canonical generator $\omega:=(s+\sqrt{\Delta_\mathfrak{O}})/2$ with $s:=\Delta_\mathfrak{O} \mod 2$ of $\mathfrak{O}$ is not smooth in general. We have to find another smooth generator $\theta$ of~$\mathfrak{O}$. 
Second, if we denote $\varphi:=\iota(\theta)$ and decompose $\varphi:=\varphi_r\circ\cdots\circ\varphi_1$ as a product of horizontal isogenies of degrees $\ell_i, \cdots, \ell_r$ respectively, we may not be able to find the $\varphi_i$ simply by using the oracle $\IsOrientable{\mathfrak{O}}$ as in Section \ref{sec: smooth disc}. We are no longer guaranteed that $\ell_i \mid \Delta_\mathfrak{O}$, so there may be $1+(\Delta_\mathfrak{O}/\ell_i)=2$ horizontal isogenies of degree $\ell_i$ from a $\mathfrak{O}$-oriented elliptic curve. To search for $\varphi$, starting at root $E$ we fill a binary tree whose nodes are $\mathfrak{O}$-oriented elliptic curves and edges are horizontal isogenies. We call such a tree an $\mathfrak{O}$-oriented $(\ell_1,\cdots, \ell_r)$-isogeny tree, see Definition~\ref{def: isogeny_tree}. The endomorphism $\varphi$ is a branch of this tree with leaf $E$. 

\begin{definition}\label{def: isogeny_tree}
    An $\mathfrak{O}$-oriented $(\ell_1,\cdots, \ell_r)$-isogeny tree is a binary tree of height $r$ whose nodes are (primitively) $\mathfrak{O}$-oriented supersingular elliptic curves and such that every node $E_{i-1}$ of depth $i\in\{1,\cdots,r\}$ has children that are horizontally $\ell_{i}$-isogenous to $E_{i-1}$.
\end{definition}

To optimize the tree search, we propose a meet-in-the middle strategy where two half-depth such trees are computed starting at $E$ instead of a single one:

\begin{enumerate}
    \item Find a generator $\theta$ of $\mathfrak{O}$ of $B$-smooth norm $N(\theta):=\prod_{i=1}^r\ell_i$.
    \item Starting at $E$, compute $\mathfrak{O}$-oriented $(\ell_1,\cdots, \ell_{s})$-isogeny tree $\mathcal{T}_1$ and an $\mathfrak{O}$-oriented $(\ell_{s+1},\cdots, \ell_{r})$-isogeny tree $\mathcal{T}_2$ (with $s\simeq r/2$).
    \item Find a matching leaf in $\mathcal{T}_1$ and $\mathcal{T}_2$.
    \item Extract the corresponding endomorphism $\varphi\in\End(E)$.
    \item Infer from $\varphi=\iota(\theta)$ an efficient representation of the canonical generator $\varphi_0:=\iota(\omega)$ (in the sense of Definition \ref{def: efficient rep}).
\end{enumerate}

We explain each step in detail in the following paragraphs.

\subsubsection{Finding a smooth norm generator}\label{sec: smooth norm search}

Let $\mathfrak{O}$ be an imaginary quadratic order and $\omega$ be a generator. We want to find another generator $\theta$ of $\mathfrak{O}$ with smooth norm $N(\theta)=\prod_{i=1}^r\ell_i$. The computation of $\varphi=\varphi_r\circ\cdots\circ\varphi_1$ associated to $\theta$ is exponential in the $\ell_i$ and $r$, so we require the $\ell_i$ and $2^r$ to be subexponential in $\log(|\Delta_\mathfrak{O}|)$. For technical reasons (see Lemma \ref{lemma: horizontal decomp general}), $N(\theta)$ should also be non-square and coprime to $\Delta_\mathfrak{O}$. In summary, we look for a generator $\theta$ of ns-$(B,r_m,\Delta_\mathfrak{O})$-smooth norm, in the sense of Definition~\ref{def:nsBsmooth}, with $B$ and $2^{r_m}$ subexponential in $\log(|\Delta_\mathfrak{O}|)$. 

\begin{definition}\label{def:nsBsmooth}
    An integer $N\in\mathbb{N}$ is \emph{$(B,r_m,d)$-smooth} when its decomposition into prime factors $N=\prod_{i=1}^r\ell_i$ satisfies $r\leq r_{m}$, $\ell_i\leq B$, and $\ell_i\nmid d$ for all $i\in\{1,\cdots, r\}$. We say that $N$ is \emph{ns-$(B,r_m,d)$-smooth} when it is $(B,r_m,d)$-smooth and not a square.
\end{definition}

We look for $\theta$ of the form $\theta:=a+\omega$ with $a\in\Z$ to be determined. There is no better known method to find $a$ than sampling $a$ randomly and testing whether $N(a+\omega)$ is ns-$(B,r_m,\Delta_\mathfrak{O})$-smooth. To make sure $N(a+\omega)$ is close to $N(\omega)$, we sample $a\in \{-\lfloor\sqrt{N(\omega)}\rfloor, \cdots, \lfloor\sqrt{N(\omega)}\rfloor\}$. We have $N(\omega)=(|\Delta_{\mathfrak{O}}|+t^2)/4$ with $t:=\Tr(\omega)\in\{0,1\}$. It follows that:
\[\cfrac{|\Delta_{\mathfrak{O}}|}{4}=N(-t/2+\omega)\leq N(a+\omega)\leq N(-\sqrt{N(\omega)}+\omega)\leq |\Delta_{\mathfrak{O}}|\]
Since $B$ is subexponential in $\log(|\Delta_\mathfrak{O}|)$, the optimal known way to test the $B$-smoothness of $N(a+\omega)$ is the method introduced in Section \ref{sec: ECM} using ECM with time complexity $L_B(1/2,\sqrt{2})$. Algorithm \ref{algo: smooth norm} presenting the search for $\theta=a+\omega$ follows.

\begin{algorithm}
\SetAlgoLined
\KwData{The discriminant $\Delta_{\mathfrak{O}}$ of an imaginary quadratic order $\mathfrak{O}$ and smoothness parameters $B>0$ and $r_m\in\mathbb{Z}_{>0}$.}
\KwResult{A generator $\theta$ of $\mathfrak{O}$ having ns-$(B,r_m,\Delta_\mathfrak{O})$-smooth norm $N(\theta)$ and its prime factors $\ell_1,\cdots,\ell_r$ (with multiplicity).}
$s\longleftarrow \Delta_{\mathfrak{O}} \mod 2$\;
$\omega\longleftarrow(s+\sqrt{\Delta_{\mathfrak{O}}})/2$\;
\Repeat{$r\leq r_m$}{
\Repeat{$R\neq\bot$}{
\Repeat{$N(a+\omega)\wedge\Delta_\mathfrak{O}=1$ and $\sqrt{N(a+\omega)}\not\in\Z$}{
Sample $a\overset{\$}{\longleftarrow}\{-\lfloor\sqrt{N(\omega)}\rfloor,\cdots,\lfloor\sqrt{N(\omega)}\rfloor\}$\;
}
$R\longleftarrow\SmoothFact{}(N(a+\omega),B)$ (Algorithm \ref{algo: SmoothFact})\;
}
$\ell_1,\cdots,\ell_r\longleftarrow R$\;
}
Return $a+\omega$ and $\ell_1,\cdots,\ell_r$\;

\caption{$\FindSmoothGen{}$ finding a smooth generator of an imaginary quadratic order~$\mathfrak{O}$.}\label{algo: smooth norm}
\end{algorithm}

\subsubsection{Filling the $\mathfrak{O}$-oriented isogeny trees}

Let $E/\F_{p^2}$ be an $\mathfrak{O}$-orientable elliptic curve and splitting primes $\ell_1,\dots,\ell_s\leq B$. We explain here how to fill $\mathcal{T}$, the $\mathfrak{O}$-oriented $(\ell_1,\dots,\ell_s)$-isogeny tree starting at $E$.

We assume $p>B|\Delta_\mathfrak{O}|$ so any $\mathfrak{O}$-orientable curve admits a unique $\mathfrak{O}$-orientation up to conjugation by Corollary~\ref{cor: unique orientation}(i). Hence, every node of $\mathcal{T}$ can be represented by $j$-invariant (the root $E_0:=E$ included). 
If $E_{i-1}$ is a node of depth $i\in\{1,\cdots,s\}$ of $\mathcal{T}$, its children $E_{i,1}$ and $E_{i,2}$ are the only two $\mathfrak{O}$-orientable curves that are $\ell_i$-isogenous to $E_i$, given by the action of ideals $\mathfrak{l}_i,$ $\overline{\mathfrak{l}_i}$ above $\ell_i$. As in Section \ref{sec: smooth disc}, to find $E_{i,1}$ and $E_{i,2}$ we compute the codomain $j$-invariants of all degree-$\ell_i$ isogenies $E_i\longrightarrow E'$ and apply the decision oracle to see which are $\mathfrak{O}$-orientable.
Determining such $j$-invariants can be done using modular polynomials in $\tilde{O}(\ell_i^2\log(p))$ operations over $\F_{p^2}$, as in Section \ref{sec: j-invariants}. The tree filling algorithm \TreeFill{} (Algorithm \ref{algo: tree filling algorithm}) follows.

\begin{algorithm}
\SetAlgoLined
\KwData{\justifying An imaginary quadratic order $\mathfrak{O}$ such that $|\Delta_{\mathfrak{O}}|<p/B$, an $\mathfrak{O}$-orientable curve $E/\mathbb{F}_{p^2}$, splitting primes $\ell_1,\cdots,\ell_s\leq B$ and an oracle \IsOrientable{\mathfrak{O}} for the Decision $\mathfrak{O}$-Orienting Problem \ref{prob:OOrientingDecision}.}
\KwResult{\justifying The $\mathfrak{O}$-oriented $(\ell_1,\cdots,\ell_s)$-isogeny tree $\mathcal{T}$ starting at $E$.}

Initialize $\mathcal{T}$ at $E_0:=E$\;

\For{$i=1$ \KwTo $s$}{
\For{$j(E_{i-1})\in\Leaves(\mathcal{T})$}{
Compute $\Phi_{\ell_i}(j(E_{i-1}),Y)$\;
Find $S_i\subset\F_{p^2}$, the set of roots of $\Phi_{\ell_i}(j(E_{i-1}),Y)$\;
\For{$j(E_i)\in S_i$}{
\If{$\IsOrientable{\mathfrak{O}}(j(E_i))$}{
Append $j(E_i)$ as a child of $j(E_{i-1})$ in $\mathcal{T}$\;}
}
}
}

Return $\mathcal{T}$\;

\caption{\TreeFill{}, the $\mathfrak{O}$-oriented isogeny tree filling algorithm.}\label{algo: tree filling algorithm}
\end{algorithm}

\subsubsection{From a tree match to a generating endomorphism}

Assume we have found $\theta$, a generator of $\mathfrak{O}$ with ns-$(B,r_m,\Delta_\mathfrak{O})$-smooth norm $N(\theta)=\prod_{i=1}^r\ell_i$. Let {$\iota : \mathfrak{O}\hookrightarrow\End(E)$} denote the orientation with $\varphi:=\iota(\theta)$. Then, we may decompose $\varphi:=\varphi_r\circ\cdots\circ\varphi_1$, where $\varphi_i$ is an $\ell_i$-isogeny for all $i\in\{1,\cdots,r\}$. 

\begin{lemma}\label{lemma: horizontal decomp general}
    Assuming $N(\theta)=\deg(\varphi)$ coprime with $\Delta_\mathfrak{O}$, all the isogenies $\varphi_i$ in the decomposition of $\varphi$ are horizontal.
\end{lemma}

\begin{proof}
    By definition, $\varphi=\iota(\alpha)$ is a horizontal isogeny from $(E,\iota)$ to itself. 
    
    Since $N(\theta)=\deg(\varphi)$ is coprime with $\Delta_\mathfrak{O}$, $\ell_1$ does not divide the conductor of $\mathfrak{O}$ so $\varphi_1$ is horizontal or descending by Proposition \ref{prop: volcano structure}(i). Suppose $\varphi_1$ is descending. Since $\varphi$ is horizontal, there are as many descending as ascending $\ell_1$-isogenies in the decomposition of $\varphi$. 
    Reordering the $\varphi_i$ if necessary, we can assume $\ell_1 = \ell_2$ and $\varphi_2$ is ascending. Since there is only one ascending isogeny, $\varphi_2$ must be the dual of $\varphi_1$, up to post composition with an isomorphism. Thus, $\varphi_2\circ\varphi_1$ factors through $[\ell_1]$ and so does $\varphi$. Then $\varphi/[\ell_1]\in\End(E)$ and $\iota$ is not a primitive orientation, contradicting our original assumption. 
    
    By the same argument, the $\varphi_i$ are also horizontal for $i\geq 2$.
    
    
\end{proof}

Since the $\varphi_i$ are horizontal, we may use $\mathfrak{O}$-oriented isogeny trees to find these isogenies. Let $s:=\lfloor r/2\rfloor$, $\mathcal{T}_1$ the $\mathfrak{O}$-oriented $(\ell_1,\cdots, \ell_{s})$-isogeny tree starting at $E_0:=E$ and $\mathcal{T}_2$ the $\mathfrak{O}$-oriented $(\ell_{s+1},\cdots, \ell_{r})$-isogeny tree starting at $E$. Assume we have found a common leaf $E_s$ in $\mathcal{T}_1$ and $\mathcal{T}_2$. The branch of $\mathcal{T}_1$ of leaf $E_s$ is a chain of horizontal $\ell_i$-isogenies $\psi_i:E_{i-1}\longrightarrow E_i$ for $i\in\{1,\cdots, s\}$ and the branch of $\mathcal{T}_2$ of leaf $E_s$ (taken depth first) is a chain of horizontal $\ell_i$-isogenies $\psi_i:E_{i-1}\longrightarrow E_i$ for $i\in\{s+1,\cdots, r\}$, with $E_r=E_0=E$. The isogeny $\psi:=\psi_r\circ\cdots\circ\psi_1$ is a horizontal isogeny of degree $\prod_{i=1}^r \ell_i=N(\theta)$, but we do not know \emph{a priori} if $\psi=\varphi=\iota(\theta)$.

\begin{lemma}\label{lemma: explicit horizontal endomorphism}
    Let $(E_0,\iota)$ be an $\mathfrak{O}$-oriented supersingular elliptic curve and $\psi\in\End(E_0)$ a horizontal endomorphism of degree coprime to $p$. Then, there exists $\alpha\in\mathfrak{O}$ such that $\psi=\iota(\alpha)$.
\end{lemma}

\begin{proof}
    Since $\psi$ is horizontal, $\psi_*(\iota)$ defines an $\mathfrak{O}$-orientation on $E_0$, like $\iota$. Since $|\Delta_{\mathfrak{O}}|<p$, by Lemma \ref{lemma: disc ineq}, we must have $\psi_*(\iota)(\mathfrak{O})=\iota(\mathfrak{O})$, so that $\psi_*(\iota)=\iota$ or $\psi_*(\iota)=\overline{\iota}$, where $\overline{\iota}(\alpha):=\iota(\overline{\alpha})$ for all $\alpha\in K$.

    If $\psi_*(\iota)=\iota$, $\psi$ commutes with $\iota(K)$ ($K:=\mathfrak{O}\otimes_\Z\Q$), so $\psi\in\iota(K)\cap\End(E_0)=\iota(\mathfrak{O})$ and $\psi=\iota(\alpha)$ for some $\alpha\in\mathfrak{O}$.

    Suppose $\psi_*(\iota)=\overline{\iota}$. As Onuki proved in \cite[Proposition 3.3 and Theorem 3.4]{OSIDH_Onuki}, $(E,\overline{\iota})$ and $(E^{(p)},(\pi_p)_*(\iota))$ are in the same orbit of the action of $\Cl(\mathfrak{O})$ on the set $\SS_{\mathfrak{O}}^{pr}(p)$ of (primitively) $\mathfrak{O}$-oriented supersingular elliptic curves over $\mathbb{F}_{p^2}$ ($\pi_p:E\longrightarrow E^{(p)}$ being the $p$-Frobenius isogeny). Hence, there exists an ideal $\mathfrak{b}\subset\mathfrak{O}$ of norm coprime with $p$, such that  $(E,\overline{\iota})=\mathfrak{b}\cdot (E^{(p)},(\pi_p)_*(\iota))$, so that $\psi_*(\iota)=\overline{\iota}=(\varphi_{\mathfrak{b}}\circ\pi_p)_*(\iota)$. 
    Consequently, $\widehat{\pi_p}\circ\widehat{\varphi_{\mathfrak{b}}}\circ\psi$ commutes with $\iota(K)$, so there exists $\alpha\in\mathfrak{O}$ such that $\widehat{\pi_p}\circ\widehat{\varphi_{\mathfrak{b}}}\circ\psi=\iota(\alpha)$ and $p \mid N(\alpha)$. Since $\SS_{\mathfrak{O}}^{pr}(p)$ is not empty (it contains $E_0$), $p$ is either inert or ramified in $K$ by \cite[Proposition~3.2]{OSIDH_Onuki}. The prime $p$ cannot be ramified, otherwise we would have $p \mid \Delta_{\mathfrak{O}}$, so $|\Delta_{\mathfrak{O}}|\geq p$. If $p$ is inert and $p \mid N(\alpha)$, then $p \mid \alpha$ so that $p^2 \mid N(\alpha)$ and $p \mid \deg(\psi)N(\mathfrak{b})$. Since $N(\mathfrak{b})$ is coprime with $p$, $p \mid \deg(\psi)$ which contradicts our assumption.

    It follows that $\psi_*(\iota)=\iota$. 
\end{proof}

\begin{lemma}\label{lemma: values alpha}
    Let $\theta:=a+\omega\in\mathfrak{O}$, with $a\in\Z$, $|a|\leq\sqrt{N(\omega)}$. Assume $N(\theta)$ is not a square and $\Delta_{\mathfrak{O}}\neq -3, -4$. The only $\alpha\in\mathfrak{O}$ such that $N(\alpha)=N(\theta)$ are $\alpha = \pm\theta,\pm\overline{\theta}$.
\end{lemma}

\begin{proof}
    Let $\alpha:=b+c\omega\in\mathfrak{O}$ with $b,c\in\Z$ such that $N(\alpha)=N(\theta)$. Then
    \begin{equation}b^2-tbc+c^2n=N(\alpha)=N(\theta)=a^2-ta+n,\label{eq: norm a, b, c}\end{equation}
    with $t:=\Tr(\omega)\in\{0,1\}$ and $n:=N(\omega)=(t^2+|\Delta_{\mathfrak{O}}|)/4$. 
    
    If $c^2>1$, the minimum value of $b^2-tbc+c^2n$ is reached when $b=ct/2$, so
    \begin{equation}
        b^2-tbc+c^2n\geq\left(n-\frac{t^2}{4}\right)c^2=\frac{|\Delta_{\mathfrak{O}}|c^2}{4}\geq |\Delta_{\mathfrak{O}}|
    \label{eq: norm a, b, cII}
    \end{equation}
    But, by \eqref{eq: norm a, b, c} and since $|a|\leq\sqrt{n}$, we have: 
    \[N(\theta)\leq 2n+t\sqrt{n}<|\Delta_{\mathfrak{O}}|\] which contradicts \eqref{eq: norm a, b, cII}. 
    
    So $c^2\leq 1$ and $c\in\{0,\pm 1\}$. If $c=0$, then $N(\theta)$ is a square which is not possible. If $c=1$, then \eqref{eq: norm a, b, c} becomes $(b-a)(b+a-t)=0$ and we have $b=a$ or $b=t-a$. If $c=-1$, then \eqref{eq: norm a, b, c} becomes $(b+a)(b-a+t)=0$ and we have $b=-a$ or $b=a-t$. Hence, $(b,c)\in\{(a,1),(t-a,1),(-a,-1),(a-t,-1)\}$ and $\alpha\in\{\pm\theta,\pm\overline{\theta}\}$.
\end{proof}

\begin{remark}
    The cases of $\Delta_{\mathfrak{O}}=-3, -4$ are excluded from this lemma because in those cases, we have a very simple way to find the orientation: 
    
    If $\Delta_\mathfrak{O}=-3$, then $\mathfrak{O}=\mathbb{Z}[\zeta_3]$, with $\zeta_3:=(1+\sqrt{-3})/2$ so any elliptic curve $E$ that is $\mathfrak{O}$-oriented contains an automorphism of order $3$. By \cite[Theorem III.10.1]{Silverman}, we must have $j(E)=0$, so $E$ is given by the Weierstrass equation $y^2=x^3+1$ (up to isomorphism), and $\zeta_3$ corresponds to the automorphism $(x,y)\in E\longmapsto (\xi_3 x, y)\in E$, where $\xi_3$ is a primitive third root of unity in $\mathbb{F}_{p^2}$.

    Similarly, if $\Delta_\mathfrak{O}=-4$, then $\mathfrak{O}=\mathbb{Z}[i]$ so any elliptic curve $E$ that is $\mathfrak{O}$-oriented contains an automorphism of order $4$. By \cite[Theorem III.10.1]{Silverman}, we must have $j(E)=1728$, so $E$ is given by the Weierstrass equation $y^2=x^3+x$ (up to isomorphism), and $i$ corresponds to the automorphism $(x,y)\in E\longmapsto (x, a y)\in E$, where $a$ is a square root of $-1$ in $\mathbb{F}_{p^2}$.
\end{remark}

By Lemmas \ref{lemma: explicit horizontal endomorphism} and \ref{lemma: values alpha}, we must have $\psi=\pm\iota(\theta)=\pm\varphi$ or $\psi=\pm\iota(\overline{\theta})=\pm\widehat{\varphi}$. The sign can be determined by computing $\Tr(\psi)$ using Schoof's algorithm \cite[Section 5]{Schoof_1995} and comparing to $\Tr(\theta)$. We recover $\iota$ or $\overline{\iota}: \mathfrak{O}\hookrightarrow\End(E)$ by mapping $\theta$ to $\pm\psi$.

However, the factors $\psi_i$ of $\psi$ have subexponential degree so they do not provide an efficient representation of $\psi$ (enabling to evaluate $\psi$ in polynomial time for instance). We apply \EfficientRep{} Algorithm \ref{algo: efficient representation} to get an efficient representation of $\iota(\omega)$ or $\iota(\overline{\omega})=\pm\psi-[a]$. The search to decision reduction Algorithm \ref{algo: reduction search to decision} follows.

For efficiency, only $j$-invariants are stored in the trees and not the $\ell_i$-isogenies relating them so we use the method of Section \ref{sec: isogenies from j-invariants} to recover them in time $\tilde{O}(\ell_i^2\log(p))$. 

\begin{algorithm}
\SetAlgoLined
\KwData{\justifying A supersingular elliptic curve $E/\mathbb{F}_{p^2}$, smoothness parameters $B,r_{m},D$, an imaginary quadratic order $\mathfrak{O}$ of discriminant $\Delta_{\mathfrak{O}}\neq-3,-4$ such that $|\Delta_{\mathfrak{O}}|<p/B$ and canonical generator $\omega$ along with an oracle \IsOrientable{\mathfrak{O}} for the Decision $\mathfrak{O}$-Orienting Problem \ref{prob:OOrientingDecision}.}

\KwResult{\justifying If $E$ is $\mathfrak{O}$-orientable, an efficient representation $F$ (as defined in~\ref{def: efficient rep}) of an endomorphism $\varphi_0\in\End(E)$ such that $\deg(\varphi_0)=N(\omega)$ and $\Tr(\varphi_0)=\Tr(\omega)$, where $\omega$ is the canonical generator of $\mathfrak{O}$.}

\If{\emph{not} $\IsOrientable{\mathfrak{O}}(E)$}{Return $\bot$\;}
$\theta, \ell_1,\cdots, \ell_r\longleftarrow \FindSmoothGen{}(\Delta_{\mathfrak{O}},B,r_{m})$ (Algorithm \ref{algo: smooth norm})\; 

$s\longleftarrow\lfloor r/2\rfloor$\;
$\mathcal{T}_1\longleftarrow\TreeFill{}(\mathfrak{O},E,\ell_1,\cdots,\ell_s)$ (Algorithm \ref{algo: tree filling algorithm})\;

$\mathcal{T}_2\longleftarrow\TreeFill{}(\mathfrak{O},E,\ell_r,\ell_{r-1},\cdots,\ell_{s+1})$\;

Search for a matching leaf $j(E_{s})\in \Leaves(\mathcal{T}_1)\cap\Leaves(\mathcal{T}_1)$\;\label{step:leaf_match}

Recover from the branch of leaf $j(E_s)$ in $\mathcal{T}_1$ the $\ell_i$-isogeny $\psi_i:E_{i-1}\longrightarrow E_i$ for all $i\in\{1,\cdots, s\}$ (using Section \ref{sec: isogenies from j-invariants})\;\label{line: recover isogeny 1}

Recover from the branch of leaf $j(E_s)$ in $\mathcal{T}_2$ the $\ell_i$-isogeny $\psi_i:E_{i-1}\longrightarrow E_i$ for all $i\in\{s+1,\cdots, r\}$\;\label{line: recover isogeny 2}

Let $\psi:= \psi_r\circ\cdots\circ\psi_1$\; 

Compute $\Tr(\psi)$ using Schoof's algorithm \cite[Section 5]{Schoof_1995}\;

$s\longleftarrow \Delta_\mathfrak{O} \mod 2$\;

$\omega\longleftarrow (s+\sqrt{\Delta_\mathfrak{O}})/2$\;

Let $\epsilon:=\Tr(\psi)/\Tr(\theta)$ and $\theta:=a+\omega$\;

$F\longleftarrow\EfficientRep{}([\epsilon]\circ\psi-[a],D)$ (Algorithm \ref{algo: efficient representation})\;

Return $F$\;

\caption{Algorithm to solve the $\mathfrak{O}$-Orienting Problem \ref{prob:OOrienting} with an oracle for the Decision $\mathfrak{O}$-Orienting Problem \ref{prob:OOrientingDecision}.}\label{algo: reduction search to decision}
\end{algorithm}

\subsection{Complexity analysis}\label{sec: complexity analysis}

\subsubsection{Complexity of the smooth norm search (Algorithm \ref{algo: smooth norm})}

To estimate the complexity of Algorithm \ref{algo: smooth norm}, we need to determine the probability that $N(a+\omega)$ is ns-$(B,r_m,\Delta_\mathfrak{O})$-smooth. We have proven results on the distribution of $B$-smooth integers among random integers but not for random values of quadratic integer polynomials. For that reason, we introduce the following heuristic assumption.

\begin{heuristic}\label{heur: smoothness proba}
    Let $f:=X^2-tX+N\in\mathbb{Z}[X]$, $a$ following the uniform distribution in $\{-\lfloor\sqrt{N}\rfloor,\cdots,\lfloor\sqrt{N}\rfloor\}$, and $b$ following the uniform distribution in $\{0,\cdots, N\}$. Then there exist constants $C>0$, $c>0$ such that for all $N\in\mathbb{Z}_{>0}$, $\log^c(N)\leq B\leq N$, $\log(N)/log(B)\leq r\leq\log_2(N)$ and $d\leq 4N$, we have:
    \[\mathbb{P}(f(a) \mbox{ is ns-$(B,r,d)$-smooth})\geq C\cdot \mathbb{P}(b \mbox{ is ns-$(B,r,d)$-smooth})\]
\end{heuristic}

This heuristic assumption is supported by an estimate on $B$-smooth values of polynomials very similar to random integers. Such an estimate has been proved in \cite[Theorem 1.1]{Martin2002} under a dual hypothesis on the number of prime values of polynomials when $B$ is in a very tight range. 
It has been conjectured \cite[Equation 1.20]{Granville2008} that this result holds for broader values of $B$.

\begin{lemma}\label{lemma: number smooth numbers}
Let $\Psi_{r}(x,y,d)$ denote the number of $(y,r,d)$-smooth integers $\leq x$:
\[\Psi_r(x,y,d)=\#\left\{n\leq x\;\middle|\; n=\prod_{i=1}^s\ell_i, \ s\leq r \mbox{ and } \forall \, 1\leq i\leq s, \ \ell_i\leq y \mbox{ and } \ell_i\nmid d \right\}\]
Then, if $r\geq\log(x)/\log(y)$,
\[\Psi_r(x,y,d)\geq \binom{\pi(y)-\pi(z)-\omega_y(d)+\lfloor\frac{\log(x)}{\log(y)}\rfloor}{\pi(y)-\pi(z)-\omega_y(d)}\]
with $z:=x^{1/r}$, $\pi(t)$ the number of prime numbers $\leq t$ and $\omega_y(d)$ the number of distinct prime divisors $\leq y$ of $d$.
\end{lemma}

\begin{proof}
    The proof follows from \cite[§ 2]{DeBruijn1966}. We have the following inequalities (following from set inclusions):
\begin{align*}
\Psi_r(x,y,d)&=\#\Bigg\{(\alpha_\ell)_{\substack{\ell\leq y\\\ell\nmid d}}\in\mathbb{N}^{\pi(y)-\omega_y(d)}\;\Bigg|\; \#\{ \ell\leq y, \ell\nmid d\mid \alpha_\ell\neq 0\}\leq r\\
& \qquad \mbox{and} \ \sum_{\ell\leq y}\alpha_\ell\log(\ell)\leq\log(x) \Bigg\}\\
&\geq \#\left\{(\alpha_\ell)_{\substack{z<\ell\leq y\\\ell\nmid d}}\in\mathbb{N}^{\pi(y)-\pi(z)-\omega_y(d)}\;\middle|\; \sum_{\substack{z<\ell\leq y\\\ell\nmid d}}\alpha_\ell\log(\ell)\leq\log(x) \right\}\\
&\geq \#\left\{(\alpha_\ell)_{\substack{z<\ell\leq y\\\ell\nmid d}}\in\mathbb{N}^{\pi(y)-\pi(z)-\omega_y(d)}\;\middle|\; \sum_{\substack{z<\ell\leq y\\\ell\nmid d}}\alpha_\ell\leq\left\lfloor\frac{\log(x)}{\log(y)} \right\rfloor\right\}
\end{align*}

To conclude, we compute the cardinality of 
\[S(k,n):=\left\{(\alpha_1,\cdots,\alpha_k)\in\mathbb{N}^k\;\middle|\;\sum_{i=1}^k \alpha_i\leq n\right\}\]
for $k,n\in\mathbb{Z}_{>0}$ and apply it to the last set in the inequalities above. The set $S(k,n)$ is in bijection with the subsets of $k$ elements in $\{1,\cdots,n+k\}$, via the maps:
\[\{s_1<\cdots<s_k\}\longmapsto (s_1-1,s_2-s_1-1,\cdots,s_k-s_{k-1}-1)\]
\[(\alpha_1,\cdots,\alpha_k)\longmapsto \{\alpha_1+1,\alpha_1+\alpha_2+2,\cdots,\alpha_1+\cdots+\alpha_k+k\}.\]
It follows that 
\[\#S(k,n)=\binom{n+k}{k}.\]
\end{proof}

\begin{lemma}\label{lemma: DeBrujin}
Let $\psi(x,y)$ be the number of $y$-smooth numbers $\leq x$. Assume that $\log(y)\ll\log(x)$ and $\log(y)\gg\log\log(x)$. Then
\[\log\left(\frac{\psi(x,y)}{x}\right)\sim -\frac{\log(x)\log\log(x)}{\log(y)}.\]
\end{lemma}

\begin{proof}
    It follows from \cite[Theorem 1]{DeBruijn1966} that for all $2< y\leq x$:
    \begin{multline}
        \log\psi(x,y)=\left(\log\left(1+\frac{y}{\log(x)}\right)\frac{\log(x)}{\log(y)}+\log\left(1+\frac{\log(x)}{y}\right)\frac{y}{\log(y)}\right)\\
        \cdot\left(1+O\left(\frac{1}{\log(y)}\right)+O\left(\frac{1}{\log\log(x)}\right)+O\left(\left(1+\cfrac{\log(x)}{\log(y)}\right)^{-1}\right)\right)\label{eq: DeBrujin}
    \end{multline}

    Since $\log(y)\gg\log\log(x)$, we have $y\gg\log^2(x)$, so that
    \begin{align*}\log\left(1+\frac{y}{\log(x)}\right)\frac{\log(x)}{\log(y)}&=\left(\log\left(\frac{y}{\log(x)}\right)+\log\left(1+\frac{\log(x)}{y}\right)\right)\frac{\log(x)}{\log(y)}\\
    &=\log(x)-\frac{\log(x)\log\log(x)}{\log(y)}+\frac{\log^2(x)}{y\log(y)}(1+o(1))\\
    &=\log(x)-\frac{\log(x)\log\log(x)}{\log(y)}+o(1)
    \end{align*}
    And
    \[\log\left(1+\frac{\log(x)}{y}\right)\frac{y}{\log(y)}=\frac{\log(x)}{\log(y)}(1+o(1)).\]
    It follows finally by \ref{eq: DeBrujin} that
    \[\log\left(\frac{\psi(x,y)}{x}\right)\sim-\frac{\log(x)\log\log(x)}{\log(y)}.\]
\end{proof}

\begin{lemma}\label{lemma: smooth asymptotics}
    Let $\psi_r^*(x,y,d)$ be the number of ns-$(y,r,d)$-smooth integers $\leq x$. Let $z:=x^{1/r}$ and $\varepsilon:=z/y$. Assume that $r\geq \log(x)/\log(y)$, $d=O(x)$, $\log(y)\ll\log(x)$, $\log(y)\gg\log\log(x)$ and $\log(1-\varepsilon)\ll \log\log(y)$. Then 
    \[\log\left(\cfrac{\psi_r^*(x,y,d)}{x}\right)\sim -\frac{\log(x)\log\log(x)}{\log(y)}\]
    as $x,y,r,d\longrightarrow +\infty$.
\end{lemma}

\begin{proof}
    The number of squares $\leq x$ being $\leq\sqrt{x}$, we have
    \[\psi_r^*(x,y,d)\geq \psi_r(x,y,d)-\sqrt{x}.\]
    And by lemma \ref{lemma: number smooth numbers}, 
    \[\psi_r(x,y,d)\geq\binom{n+k}{k}\]
    with $k:=\pi(x)-\pi(z)-\omega_y(d)$, $n:=\lfloor \log(x)/\log(y)\rfloor$, so that
    \begin{multline}
        \log\psi_r(x,y,d)\geq \log\binom{n+k}{k}=(n+k)\log(n+k)-k\log(k)-n\log(n)\\
        +\frac{1}{2}\log(n+k)-\frac{1}{2}\log(k)-\frac{1}{2}\log(n)+O(1).\label{eq: binom}
    \end{multline}

    We have $\pi(t)=t/\log(t)+O(t/\log(t)^2)$ as $t\longrightarrow +\infty$ and 
    \[\omega_y(d)=O(\log(d))=O(\log(x))=o(y/\log^2(y)),\] 
    since $y\gg\log^\alpha(x)$ for all $\alpha>0$, because $\log(y)\gg\log\log(x)$. It follows that
    \[k=\pi(x)-\pi(z)-\omega_y(d)=\frac{(1-\varepsilon)y}{\log(y)}+O\left(\frac{y}{\log(y)^2}\right).\]
    Besides, since $\log(1-\varepsilon)\ll \log\log(y)$, we have $1-\varepsilon\gg 1/\log(y)$ and furthermore, $y\gg\log^\alpha(x)$ for all $\alpha>0$, so that:
    \[\frac{n^2}{k}=\frac{\log^2(x)}{(1-\varepsilon)y\log(y)}=o\left(\frac{\log^2(x)}{y}\right)=o(1),\]
    so \ref{eq: binom} becomes
    \begin{align*}\log\psi_r(x,y,d)&\geq n\log\left(\frac{k}{n}\right)+n-\frac{1}{2}\log(n)+O(1)\\
     &=\frac{\log(x)}{\log(y)}\log\left(\frac{(1-\varepsilon)y}{\log(x)}\right)+\frac{\log(x)}{\log(y)}-\frac{1}{2}\log\left(\frac{\log(x)}{\log(y)}\right)+O(1)\\
     &=\log(x)-\frac{\log(x)\log\log(x)}{\log(y)}+o\left(\frac{\log(x)\log\log(y)}{\log(y)}\right)\\
     & \qquad (\mbox{since } \log(1-\varepsilon)\ll \log\log(y))\\
     &=\log(x)-\frac{\log(x)\log\log(x)}{\log(y)}(1+o(1))
     \end{align*}
    It follows that
    \begin{align*}\frac{\psi_r(x,y,d)}{\sqrt{x}}&\geq\exp\left(\frac{1}{2}\log(x)-\frac{\log(x)\log\log(x)}{\log(y)}(1+o(1))\right)\\
    &=\exp\left(\frac{1}{2}\log(x)(1+o(1))\right)\longrightarrow +\infty,
    \end{align*}
    since $\log(y)\gg\log\log(x)$. Finally, we have
    \begin{align*}
        \log\left(\frac{\psi_r^*(x,y,d)}{x}\right)&=\log\left(\frac{\psi_r(x,y,d)}{x}\right)+\log\left(1-\frac{\sqrt{x}}{\psi_r(x,y,d)}\right)\\
        &\geq-\frac{\log(x)\log\log(x)}{\log(y)}(1+o(1))+o(1)
    \end{align*}

    Besides, $\psi_r^*(x,y,d)\leq\psi(x,y)$, so we conclude by Lemma \ref{lemma: DeBrujin}.
\end{proof}

\begin{proposition}\label{prop: complexity find smooth}
    Let $\Delta:=|\Delta_{\mathfrak{O}}|$ and $\varepsilon:=\Delta^{1/r_m}/B$. We assume that $B$ is subexponential in $\log(\Delta)$, $\varepsilon<1$ and $\log(1-\varepsilon)\ll \log\log(B)$. Then Algorithm \ref{algo: smooth norm} terminates in expected time
    \begin{multline*}T_{FS}(\Delta,B,r_m)= \exp\Bigg((1+o(1))\frac{\log(\Delta)\log\log(\Delta)}{\log(B)}\\
    +(\sqrt{2}+o(1))\sqrt{\log(B)\log\log(B)}\Bigg).\end{multline*}
\end{proposition}

\begin{proof}
    By Heuristic \ref{heur: smoothness proba} (since $\varepsilon<1$ i.e. $r_m\geq \log(\Delta)/\log(B)$), the probability to find an ns-$(B,r_m,\Delta)$-smooth value of $N(a+\omega)$ stisfies
    \[\mathbb{P}(B, r_m,\Delta)\geq C\cdot\frac{\psi_{r_m}^*(N(\omega),B,\Delta)}{N(\omega)},\]
    where $C>0$ is a constant. Since $N(\omega)=(\Delta+t^2)/2$ with $t:=\Tr(\omega)=\Delta\mod 2$ and $B$ is subexponential in $\log(\Delta)$, we have $\Delta=O(N(\omega))$, $N(\omega)=O(\Delta)$, $\log(B)\ll\log(N(\omega))$ and $\log(B)\gg\log\log(N(\omega))$. We also have $r_m\geq \log(\Delta)/\log(B)$ and $\log(1-\varepsilon)\ll \log\log(B)$, so we may apply Lemma \ref{lemma: smooth asymptotics}:
    \begin{align*}
        \log\left(\frac{\psi_{r_m}^*(N(\omega),B,\Delta)}{N(\omega)}\right)&\sim -\frac{\log(N(\omega))\log\log(N(\omega))}{\log(B)}\\
        &\sim\frac{\log(\Delta)\log\log(\Delta)}{\log(B)}(1+o(1)).
    \end{align*}
    
    Hence, Algorithm \ref{algo: smooth norm} terminates in expected time 
\begin{align*}
T_{FS}(\Delta, B, r_m)&=\frac{L_B(1/2,\sqrt{2})}{\mathbb{P}(B, r_m, \Delta)}=\exp\Bigg((1+o(1))\frac{\log(\Delta)\log\log(\Delta)}{\log(B)}\\
&\qquad +(\sqrt{2}+o(1))\sqrt{\log(B)\log\log(B)}\Bigg).
\end{align*}
\end{proof}

\subsubsection{Complexity of the tree filling algorithm (Algorithm \ref{algo: tree filling algorithm})}

\begin{proposition}\label{prop: tree filling complexity}
With inputs $B>0$, an imaginary quadratic order $\mathfrak{O}$ with $|\Delta_\mathfrak{O}|B<p$, primes $\ell_1,\cdots,\ell_s\leq B$ splitting in $\mathfrak{O}$ and an oracle \IsOrientable{\mathfrak{O}} for Problem \ref{prob:OOrientingDecision} running in constant time, Algorithm \ref{algo: tree filling algorithm} runs in time
\[O\left(2^sB^2\polylog(B)\log(p)M(p)\right),\]
where $M(p)$ is the time complexity of the multiplication in $\F_p$. It also uses $O(2^s\log(p))$ bits of memory.
\end{proposition}

\begin{proof}
    Filling-in tree $\mathcal{T}$ in Algorithm \ref{algo: tree filling algorithm} costs for all $1\leq i\leq s$, $2^{i-1}$ calls to $\IsOrientable{\mathfrak{O}}$ and the computation of $2^{i-1}$ sets of $j$-invariants $\ell_i$-isogenous to the same elliptic curve. Each call to $\IsOrientable{\mathfrak{O}}$ costs $O(1)$ and each $j$-invariants computation costs $O(\ell_i^2\polylog(\ell_i)\log(p))$ operations over $\F_{p^2}$ by Section \ref{sec: j-invariants}. Arithmetic operations over $\F_{p^2}$ cost $O(M(p))$. Hence, the total cost of filling tree $\mathcal{T}$~is
    \begin{align*}T_{tree}(s,B,p)&=\sum_{i=1}^{s}2^{i-1}O\left(\ell_i^2\polylog(\ell_i)\log(p)M(p)\right)\\
    &=\sum_{i=1}^s 2^{i-1}O(B^2\polylog(B)\log(p)M(p))\\
    &=O\left(2^sB^2\polylog(B)\log(p)M(p)\right).\end{align*}
    The memory used by Algorithm \ref{algo: tree filling algorithm} is the size of tree $\mathcal{T}$, which contains $\sum_{i=1}^s 2^{i-1}=2^s-1$ $j$-invariants defined over $\F_{p^2}$. Each $j$-invariant takes $2\log(p)$ bits to store, so the algorithm uses $O(2^s\log(p))$ bits of memory.
\end{proof}

\subsubsection{Complexity of the search to decision reduction algorithm (algorithm \ref{algo: reduction search to decision})}

\begin{theorem}\label{thm: complexity}
    Let $\Delta:=|\Delta_{\mathfrak{O}}|$. Then, with smoothness parameters 
    \[B:=L_\Delta\left(\frac{1}{2},\frac{\sqrt{2}}{2}\right), \quad r_m:=\left\lceil \sqrt{\frac{2\log(\Delta)}{\log\log(\Delta)}}\right\rceil+1 \quad \mbox{and} \quad D:=O(\log(p))\]
    and provided $B\Delta<p$, Algorithm \ref{algo: reduction search to decision} terminates in time 
    \[L_\Delta(1/2,\sqrt{2})\log(p)M(p),\]
    where $M(p)$ is the time complexity of multiplying over $\mathbb{F}_p$. It also requires 
    \[O\left(2^{\sqrt{2\log(\Delta)/\log\log(\Delta)}}\log(p)\right)\]
    bits of memory.
\end{theorem}

\begin{proof}
We already have proved the termination of Algorithm \ref{algo: reduction search to decision} when $B\Delta<p$. This is a consequence of Lemma \ref{lemma: horizontal decomp general}, Lemma \ref{lemma: explicit horizontal endomorphism} and Heuristic \ref{heur: smoothness proba} (which prove that \TreeFill{} and \FindSmoothGen{} terminate).

On the whole, the total time complexity of Algorithm \ref{algo: reduction search to decision} is
\[T(B,\Delta,r_m,p)=T_{FS}+2T_{tree}+T_{iso}+T_{trace}+T_{rep},\]
where:
\begin{itemize}
    \item $T_{FS}$ is the execution time of \FindSmoothGen{} (Algorithm \ref{algo: smooth norm}), given by Proposition \ref{prop: complexity find smooth}:
    \begin{multline*}T_{FS}(\Delta,B,r_m)= \exp\Bigg((1+o(1))\frac{\log(\Delta)\log\log(\Delta)}{\log(B)}\\
    +(\sqrt{2}+o(1))\sqrt{\log(B)\log\log(B)}\Bigg).\end{multline*}

    \item $T_{tree}$ is the execution time of \TreeFill{} (Algorithm \ref{algo: tree filling algorithm}), given by Proposition \ref{prop: tree filling complexity}:
    \[T_{tree}(B,s,p)=O\left(2^sB^2\polylog(B)\log(p)M(p)\right).\]
    with $s=r_m/2+O(1)$.

    \item $T_{iso}$ is the time taken in lines \ref{line: recover isogeny 1} and \ref{line: recover isogeny 2} of Algorithm \ref{algo: reduction search to decision} to recover the chain of $\ell_i$-isogenies $\psi_i:E_{i-1}\longrightarrow E_i$, given the sequence of $j$-invariants $j(E_0)=j(E),j(E_1),\cdots,j(E_r)=j(E)$. By Section \ref{sec: isogenies from j-invariants}, recovering an $\ell_i$-isogeny from the $j$-invariants of its domain and codomain costs $O(\ell_i^2\polylog(\ell_i)\log(p))$ operations over $\F_{p^2}$. Hence, we have
    \[T_{iso}=O(r_m B^2\polylog(B)\log(p)M(p))\]

    \item $T_{trace}$ is the time needed to compute the trace of $\psi=\psi_r\circ\cdots\circ\psi_1$. We use Schoof's algorithm \cite[Section~5]{Schoof_1995}. Namely, we look for primes $p_1, \cdots, p_t$ such that $\prod_{i=1}^t p_i>4\sqrt{\deg(\psi)}$ and evaluate $\psi$ on $E[p_i]$ to find $\tau_i\in\Z/p_i\Z$ such that $\psi^2-[\tau_i]\psi+[\deg(\psi)]$ is zero on $E[p_i]$ and recover $\Tr(\psi)$ by solving $\Tr(\psi)\equiv \tau_i \mod p_i$ for all $i\in\{1,\cdots, t\}$ via Chinese remainder theorem. Since $\deg(\psi)=N(\theta)\leq \Delta$, we can choose $t=O(\log(\Delta))$ and $p_i=O(\log(\Delta))$. Hence, the dominant cost is the evaluation via $\psi$ of $O(\log(\Delta))$ points all defined over an extension of degree $O(\log(\Delta))$ of $\mathbb{F}_{p^2}$ (by Lemma \ref{lemma: field definition torsion}). This cost amounts to
    \[T_{trace}(B,r_m,\Delta,p)=O(r_mB\log^3(\Delta)M(p)).\]

    \item $T_{rep}$ is the running time of \EfficientRep{} (Algorithm \ref{algo: efficient representation}). Since $\deg([\epsilon]\circ\psi-[a])=N(\omega)\leq(\Delta+1)/4$, we can find a $D$-powersmooth number coprime with $\deg([\epsilon]\circ\psi-[a])$ when $D=O(\log(\Delta))$ (line \ref{line: coprime with degree} of Algorithm \ref{algo: efficient representation}). Hence, by Proposition \ref{prop: complexity efficient rep}, the dominant cost of the call to \EfficientRep{} is given by $O(\log(\Delta))$ evaluations of $\psi$ on points defined over an extension of degree $O(\log(\Delta))$ of $\mathbb{F}_{p^2}$, which amounts to
    \[T_{rep}(\Delta,B,r_m,p)=O(r_m B\log^3(\Delta)M(p)).\]
\end{itemize}

It follows that:
\begin{align*}T(B,r_m,\Delta,p)&=T_{FS}+2T_{tree}+T_{iso}+T_{trace}+T_{rep}\\
&=\exp\Bigg((1+o(1))\frac{\log(\Delta)\log\log(\Delta)}{\log(B)}\\
&\qquad+(\sqrt{2}+o(1))\sqrt{\log(B)\log\log(B)}\Bigg)\\
&\qquad+M(p)\log(p)\exp\left(\frac{\log(2)r_m}{2}+2\log(B)\right)
\end{align*}

But by Proposition \ref{prop: complexity find smooth}, we have $r_m=\log(\Delta)/\log(B\varepsilon)$ with $\log(1-\varepsilon)\ll \log\log(B)$. We can impose that $\varepsilon\longrightarrow 0$, so that $\log(1-\varepsilon)\ll \log\log(B)$ and that  $\log(\varepsilon)\ll\log(B)$, so that $r_m\sim\log(\Delta)/\log(B)$. Heuristically, the quantity $T(\Delta,B,r_m,p)$ is minimal when the arguments of the two exponentials are close, i.e. when 
\[\frac{\log(\Delta)\log\log(\Delta)}{\log(B)}\simeq 2\log(B),\]
the other terms being negligible. Hence, we choose
\[B=\exp\left(\frac{\sqrt{2}}{2}\sqrt{\log(\Delta)\log\log(\Delta)}\right)=L_\Delta\left(\frac{1}{2},\frac{\sqrt{2}}{2}\right),\]
so that
\[T(B,r_m,\Delta,p)=M(p)\log(p)L_\Delta\left(\frac{1}{2},\sqrt{2}\right).\]
and
\begin{align*}r_m&=\sqrt{\frac{2\log(\Delta)}{\log\log(\Delta)}}\left(1+\frac{\sqrt{2}\log(\varepsilon)}{\sqrt{\log(\Delta)\log\log(\Delta)}}\right)^{-1}\\
&=\sqrt{\frac{2\log(\Delta)}{\log\log(\Delta)}}-\frac{2\log(\varepsilon)}{\log\log(\Delta)}.
\end{align*}
Hence, we can set $r_m:=\lceil \sqrt{2\log(\Delta)/\log\log(\Delta)}\rceil+1$, so that $\log(\varepsilon)=O(\log\log(\Delta))=o(\log(B))$. 

The space complexity is dominated by the trees $\mathcal{T}_1$ and $\mathcal{T}_2$, so Algorithm \ref{algo: reduction search to decision} uses 
\[O(2^{r_m/2}\log(p))=O\left(2^{\sqrt{2\log(\Delta)/\log\log(\Delta)}}\log(p)\right)\]
bits of memory by Proposition \ref{prop: tree filling complexity}. 
\end{proof}

\begin{corollary}
    Given an imaginary quadratic order $\mathfrak{O}$ of discriminant $\Delta_{\mathfrak{O}}$ and a prime $p>L_{|\Delta_\mathfrak{O}|}(1/2,\sqrt{2}/2)|\Delta_{\mathfrak{O}}|$, then, over $\F_{p^2}$ the $\mathfrak{O}$-orienting Problem (Problem~\ref{prob:OOrienting}) reduces to the Decision $\mathfrak{O}$-orienting Problem (Problem \ref{prob:OOrientingDecision}) in time 
    \[L_{|\Delta_\mathfrak{O}|}(1/2,\sqrt{2})\log(p)M(p)\] 
    using 
    \[O\left(2^{\sqrt{2\log(|\Delta_\mathfrak{O}|)/\log\log(|\Delta_\mathfrak{O}|)}}\log(p)\right)\] 
    bits of memory, $M(p)$ being the time complexity of multiplication over $\F_{p}$.
\end{corollary}

\section{$\mathfrak{O}$-orienting problem for quaternion orders}\label{sec: quaternion}

Isogeny problems can often be translated to quaternion problems via the Deuring correspondence, and in many cases the quaternion problems are easier to solve. In this section we consider the quaternion analogue of the $\mathfrak{O}$-Orienting Problem stated as follows:

\begin{problem}[Quaternion Order Embedding Problem]\label{prob: quat}
    Given a maximal quaternion order $\mathcal{O} \subset B_{p, \infty}$ and an imaginary quadratic order $\mathfrak{O}$ where an $\mathfrak{O}$-orientation of $\mathcal{O}$ exists, find the orientation.
\end{problem}
Similarly to the curve setting, we define an $\mathfrak{O}$-orientation of $\mathcal{O}$ to be an embedding $\iota: \mathfrak{O} \hookrightarrow \mathcal{O}$  which cannot be extended to a superorder of $\mathfrak{O}$, also known as an optimal embedding \cite[Chapter 30]{Voight}.


In this section we present a general algorithm, and analyse its complexity, noting special cases. For complexity analysis we assume an efficient factorization oracle exists, however we provide a practical alternative for running the algorithm without such an oracle. For embedding small discriminant quadratic orders $\mathfrak{O}$, our algorithm improves the state of the art being efficient up to $\disc(\mathfrak{O}) = O(p)$. 

Before moving on to the actual algorithms we give a brief technical overview of the main idea. First we compute a short prime norm $N$ ($\approx \sqrt{p}$) connecting ideal between a quaternion order $\mathcal{O'}$ isomorphic to $\mathcal{O}$ and a special extremal order.
Our goal is to compute an element of prescribed trace and norm in $\mathcal{O'}$ and then one can easily construct an element with said trace and norm in $\mathcal{O}$ as well. For simplicity assume that the prescribed trace is 0. The trace 0 part of $\mathcal{O'}$ is a rank 3 lattice and one can compute the Hermite Normal Form (HNF) of this lattice. This means that one has a basis of the form $e_{11}i+e_{12}j+e_{13}k,e_{22}j+e_{23}k,e_{33}k$ and even though $e_{ij}$ are not likely to be integers, their denominator is a divisor of $2N$. When looking for an element of trace 0 and norm smaller than $p$ the coefficients of this element with respect to this HNF basis will have a very specific structure. Namely the coefficient of $e_{11}i+e_{12}j+e_{13}k$ will be smaller than $p$ in absolute value and thus can be easily determined by looking at the norm modulo $p$. Then one only has to work out the two other coefficients which is equivalent to solving a binary quadratic form where the quadratic part is positive definite. This can then essentially be reduced to Cornacchia's algorithm \cite{sawilla2008new}. We can extend this to filter out imprimitive solutions.

\subsection{Finding General Embeddings}\label{ssec: quat_findingembeddings}

First we present an algorithm for finding embeddings, and in the next section we use this to define orientations. Suppose we are given a maximal quaternion order $\mathcal{O} \subset B_{p, \infty}$ in terms of a $\Z$-basis, and an imaginary quadratic order $\mathfrak{O} = \mathbb{Z}[\omega]$, by generator $\omega$ of reduced trace $t$ and reduced norm $d$. 

We start with an observation: suppose an embedding $\iota: \mathbb{Z}[\omega] \hookrightarrow \mathcal{O}$ exists and let $\alpha = \iota(\omega)$. Since $\omega^2 - t \omega + d = 0$ we must also have $\alpha^2 - t \alpha + d = 0$. Hence $\alpha$ also has trace $t$ and norm $d$. Finding any element $\alpha$ of norm $d$ and trace $t$ is enough to define an embedding $\iota$, solving Problem \ref{Problem 4}. This is the approach we take in Algorithm \ref{New Quaternion Search Algorithm}, finding $\alpha \in \mathcal{O}$ of a given norm and trace. We make the assumption $p \neq 2$ and conventionally use $1,i,j,k$ as a basis of $B_{p, \infty}$ with $i^2 = -q$ and $j^2 = -p$. If $p\equiv3\pmod{4}$ we take $q = 1$. If $p\equiv 5\pmod{8}$ we take $q = 2$. If $p\equiv 1\pmod{8}$ we take $q$ to be a prime $q\equiv 3\pmod{4}$ such that $p$ is not a quadratic residue modulo $q$. While $p \equiv 3 \mod 4$ is the most relevant for isogeny based cryptography, we consider general $p$. We fix a maximal order $\mathcal{O}_0$ in the following way:
\begin{proposition}\cite[Proposition 5.2]{pizer1980algorithm}\label{O_0}
    The following definitions give a maximal order in $B_{p, \infty}$ for any $p \neq 2$:
    \[
        \mathcal{O}_0 = 
        \begin{cases}
            \mathbb{Z}[\frac{1+j}{2}, \frac{i+k}{2}, j, k] & \text{if } p \equiv 3 \mod 4 \\
            \mathbb{Z}[\frac{1+j+k}{2}, \frac{i+2j+k}{4}, j, k] & \text{if } p \equiv 5 \mod 8 \\
            \mathbb{Z}[\frac{1+i}{2}, \frac{i + ck}{q}, \frac{j+k}{2}, k] & \text{if } p \equiv 1 \mod 8 \\
        \end{cases}
    \]
    where $c$ is an integer such that $q$ divides $c^2 p + 1$ where $q$ and $c$ exist by \cite[Proposition 1]{eisentrager2018supersingular}.
\end{proposition}

We address arbitrary trace in Remark~\ref{rmk:arbitrary_trace} and  Algorithm~\ref{New Quaternion Search Algorithm} has no restrictions on trace. However, for simplicity we first describe the algorithm under the assumption that the trace of $\omega$ is zero:

\begin{description}


    \item[Step 1: Compute HNF] Put the basis of $\mathcal{O}$ into column-style Hermite normal form (HNF). We denote the basis vectors $e_0, e_1, e_2, e_3$. Then we can write $\mathcal{O}$ as:
\begin{equation}\label{hnf1}\begin{split}
    \mathcal{O} = \langle e_0, e_1, e_2, e_3 \rangle_{\mathbb{Z}} = \langle & e_{00} + e_{01}i + e_{02}j + e_{03}k, \\
    & e_{11}i + e_{12}j + e_{13}k, \\
    & e_{22}j + e_{23}k, \\
    & e_{33}k \rangle_{\mathbb{Z}}
\end{split}\end{equation}
with coefficients $e_{mn} \in \mathbb{Q}$. For example see the orders in Proposition \ref{O_0} above. We know the basis is full rank, so $e_{nn} \neq 0$ for $n=0,1,2,3$, and we prove some additional properties:
\begin{lemma}\label{basis_properties} Given a maximal order $\mathcal{O} \subset B_{p, \infty}$ with a basis in the above form, the following properties hold (up to finding another basis in right form):
    \begin{enumerate}
        \item $e_{mn} \geq 0$ for all $n,m$
        \item For all $e_{mn}$, denominators divide $K \cdot N(I)$ where $K = 2, 4$ or $2q$ depending on whether $p \equiv 3 \mod 4$, or $\equiv 5 \mod 8$ or $\equiv 1 \mod 8$ respectively, and where $I:= N\mathcal{O}\mathcal{O}_0, N := [\mathcal{O} : \mathcal{O} \,\cap \,\mathcal{O}_0]$ is the connecting ideal from $\mathcal{O}_0$
        \item $e_{00} = \frac{1}{2}$
        \item $e_{22} e_{33} \leq N(I)$
        \item $e_{01} = 0$ or $e_{01} = 1/ (2 K e_{22}e_{33})$ where $K$ is defined in (2)
    \end{enumerate}
\end{lemma}
\begin{proof}
    \begin{enumerate}
        \item Requirement of HNF.
        \item As defined, $I$ is the connecting ideal between $\mathcal{O}_0$ and $\mathcal{O}$. $I$ is contained in both $\mathcal{O}_0$ and $\mathcal{O}$ and $N(I)\cdot\mathcal{O} = I \bar{I} \subseteq \mathcal{O}_0$ \cite[Proposition 16.6.15]{Voight}. Therefore the largest denominator of all $e_{mn}$s is at most $N(I)$ times the largest denominator of $\mathcal{O}_0$ as given in Prop \ref{O_0}.
        \item  The trace of any element must be integral hence $2 e_{00} \in \mathbb{Z}$. We must also have $1 \in \mathcal{O}$ hence $e_{00} \mid 1$ so either $e_{00} = \frac{1}{2}$ or $1$ and $\Tr(\mathcal{O}) = \mathbb{Z}$ or $2 \mathbb{Z}$ respectively. The (non-reduced) discriminant of any maximal order in $B_{p, \infty}$ is $p^2$, so by definition $p^2 = \det (\Tr(e_m e_n)) \in \Tr(\mathcal{O})$, but $p$ is odd, so $p^2 \not \in 2 \mathbb{Z}$ so we must have $e_{00} = \frac{1}{2}$.
        \item As $\mathcal{O}$ and $\mathcal{O}_0$ are maximal, they both have the same discriminant. Hence the change of basis matrix must have determinant 1 \cite[Lemma 15.2.5]{Voight}, which means $\prod e_{nn} = \prod f_{nn} = \frac{1}{2\cdot K}$, where $(f)_n$ is the basis of $\mathcal{O}_0$ specified in Proposition~\ref{O_0}. Then using (3) we have $e_{11} = 1/(Ke_{22}e_{33})$. The result follows from (2).
        \item $1 \in \mathcal{O}$ so there is some $n \in \mathbb{Z}$ such that $\frac{1}{e_{00}} e_{01} - n e_{11} = 0$. From above $e_{00} = \frac{1}{2}$, and $e_{11} = \frac{1}{Ke_{22}e_{33}}$ so $2 e_{01} = \frac{n}{K e_{22}e_{33}}$. But to be in HNF we must have already reduced $e_{01}$ as much as possible hence $n = 0$ or $1$.
    \end{enumerate}
\end{proof}

In general, we can replace the order $\mathcal{O}$ by an isomorphic order $\mathcal{O'}$, having denominator bounded by $N := K \cdot N(I') = O(\sqrt{p})$, where $I'$ is a connecting $(\mathcal{O}_0, \mathcal{O})$-ideal equivalent to $I$, and where $K$ is defined in (2) of Prop $\ref{basis_properties}$. Such an ideal always exists by the following lemma (taken from \cite[Lemma 5.2.2]{SQISignHD})

\begin{lemma}\label{N_sqrt_p}
Let $\mathcal{O} \subset B_{p, \infty}$ be a maximal order with connecting ideal $I = I(\mathcal
O_0, \mathcal{O})$, then there exists an equivalent ideal $J \sim I$ with $N(J) \leq \frac{2\sqrt{2}}{\pi} \sqrt{p}$
\end{lemma}

We return to this in Section \ref{subsec:rerandomization}, but for now, by passing to the isomorphic order ``closest'' to $\mathcal{O}_0$, we assume that $N$ is of size $O(\sqrt{p})$

    \item[Step 2: Fix trace] To find a trace zero element $\alpha$ of norm $d$, we may write an arbitrary element in the following form:
\begin{equation*}
    \alpha = \alpha_0 e_0 + \alpha_1 e_1 + \alpha_2 e_2 + \alpha_3 e_3
\end{equation*}
Note that since we are working in Hermite Normal Form only $e_0$ contributes to the trace of $\alpha$ so we set $\alpha_0 = 0$ to get $\Tr(\alpha) = 0$.

For the condition on the norm, consider the case $p \equiv 3 \mod 4$ for simplicity, however note that this generalizes for any prime $p \neq 2$. Then we have the rational ternary quadratic form:
\[
    (\alpha_1 e_{11})^2 + p(\alpha_1 e_{12} + \alpha_2 e_{22})^2 + p(\alpha_1 e_{13} + \alpha_2 e_{23} + \alpha_3 e_{33})^2 = \nrd(\alpha) = d
\]
    \item[Step 3: Find $\alpha_1 \mod p$] Since $\alpha_1$ controls the coefficient of $i$ it is the only term without a factor of $p$. Hence working modulo $p$ removes terms containing $\alpha_2$ and $\alpha_3$, and we can find $\alpha_1 \equiv r \mod p$.
\[
    r_{\pm} := \frac{\pm \sqrt{d}}{e_{11}} \mod p
\]
Fix the least positive residue class representative $r = r_{+}$, as we can execute the remainder of the algorithm a second time on $r_{-}$ if necessary. Then substitute $\alpha_1 = r + k p$ giving a rational ternary quadratic form in $k$, $\alpha_2$ and $\alpha_3$.
    \item[Step 4: A binary quadratic form] As defined in Step 1, we may multiply by the denominator $N^2$ to obtain integral coefficients. Rearranging we have:
\[
    p N^2 ( \gamma_1^2 + \gamma_2^2 ) = N^2(d - \alpha_1^2 e_{11}^2)
\]
where
\begin{equation*}
    \gamma_1 = \alpha_1 e_{12} + \alpha_2 e_{22},
    \qquad
    \gamma_2 = \alpha_1 e_{13} + \alpha_2 e_{23} + \alpha_3 e_{33}.
\end{equation*}
Let $v := N^2(\gamma_1^2 + \gamma_2^2)$ and notice $v \geq 0$. From the right-hand side above we see its value depends on $k$.
\[
    v = \frac{N^2 (d - (r + kp)^2 e_{11}^2)}{p}
\]
Clearly $v$ decreases as $k$ increases. Without loss of generality we can assume $k \geq 0$, and since $v \geq 0$ we get an upper bound on $k$. We can iterate over this range of $k$ which is precisely
\[
    k=0, ..., \left\lfloor \frac{\sqrt{d}}{p e_{11}} - \frac{r}{p} \right\rfloor
\]
where for each iteration, we compute $v$ using the above equation, and with $k$ fixed are left with the integral binary quadratic form $v = N^2(\gamma_1^2 + \gamma_2^2)$.
    \item[Step 5: Cornacchia's Algorithm] Writing the above form as $\beta_1^2 + \beta_2^2 = v$ we solve for integral pairs $(\beta_1, \beta_2)$ using Cornacchia's algorithm. For a valid solution we can write it in the form:
\begin{align*}
    \beta_1 &= N \gamma_1 = N\alpha_1 e_{12} + N\alpha_2 e_{22} \\
    \beta_2 &= N \gamma_2 = N\alpha_1 e_{13} + N\alpha_2 e_{23} + N\alpha_3 e_{33}
\end{align*}
and solve for $\alpha_2$ and $\alpha_3$
\begin{equation*}
    \alpha_2 = \frac{\beta_1 - N \alpha_1 e_{12}}{N e_{22}}
    \qquad
    \alpha_3 = \frac{\beta_2 - N \alpha_1 e_{13} - N \alpha_2 e_{23}}{N e_{33}}
\end{equation*}

Finally, we must check $\alpha_2, \alpha_3 \in \mathbb{Z}$. If this is true we have a valid solution $\alpha = \alpha_1 e_1 + \alpha_2 e_2 + \alpha_3 e_3$. If not we continue trying the next solution to Cornacchia's, or move on to the next iteration of $k$ in Step 4. If no solutions are found it means $\mathbb{Z}[\omega]$ does not embed into $\mathcal{O}$.
\end{description}


\begin{remark}[Arbitrary trace $t$]\label{rmk:arbitrary_trace}
    Suppose the element we are searching for does not have trace zero. We can always reduce the problem to finding an element of trace zero. Suppose $t \in 2 \mathbb{Z}$, then since $\mathcal{O}$ is a ring we have $1 \in \mathcal{O}$ so $\alpha - t/2 \in \mathcal{O}$ has trace zero and norm $d - t^2/4 \in \mathbb{Z}$. We can search for this trace zero element then translate back to find $\alpha$. Similarly if $t$ is odd we have trace zero element $2\alpha - t \in \mathcal{O}$ of norm $4d - t^2$, once found we translate back, divide by 2 and check $\alpha \in \mathcal{O}$ in Step 5 of the algorithm. If not we continue searching.

    Note for $t$ odd, this is not optimal as the scaling increases $d$ by a factor of $4$, and hence the number of iterations of $k$ by a factor of $2$, which can double the running time. Instead we can avoid this by incorporating additional constant terms for the non-zero trace. These details are included in Algorithm~\ref{New Quaternion Search Algorithm}, which we use for our implementation.
\end{remark}

The complete algorithm for arbitrary trace is summarised in Algorithm \ref{New Quaternion Search Algorithm}. Additionally we describe a few further generalisations and improvements:

\begin{algorithm}[h]
    \SetAlgoLined
    \KwData{Maximal order $\mathcal{O} \subset B_{p,\infty}$, given in terms of basis $e_0, e_1, e_2, e_3$. Quadratic order in the form $\mathbb{Z}[\omega]$ given by $\omega \in \mathbb{Q}(\sqrt{-z})$.}
    \KwResult{Returns element $\alpha \in \mathcal{O}$, which defines an embedding $\iota: \mathbb{Z}[\omega] \hookrightarrow \mathcal{O}$ by $\omega \mapsto \alpha$. Or returns $\bot$ if no element $\alpha$ exists.}

    
    Compute $d = \nrd(\omega)$ and $t = \Tr(\omega) \in \mathbb{Q}$\;

    Compute Hermite normal form of order, giving basis $e_0, e_1, e_2, e_3$ in form of Equation \eqref{hnf1}. Denote coefficient $n$ of vector $m$ as $e_{mn}$\;

    Compute $\alpha_0 := \frac{t}{2e_{00}}$\;

    \If{$d, \alpha_0 \not\in \mathbb{Z}$}{
        Return $\bot$\;
    }

    Compute $r_{\pm} := \frac{1}{e_{11}} \left( \pm \sqrt{d - (\alpha_0 e_{00})^2} - \alpha_0 e_{01}\right) \mod p$\;

    Set $r = r_{+}$\;

    Compute $N := lcm(\{ \text{Denom}(e_{mn}) : 0 \leq m \leq 3, m \leq n \leq 3 \})$ where $\text{Denom}(n)=b$ where $n=\frac{a}{b}$, with $b\geq 1$, is the simplest form of $n \in \mathbb{Q}$\;

    \For{$k= \left\lfloor \frac{1}{p e_{11}}\left(\sqrt{d - (\alpha_0 e_{00})^2} - \alpha_0 e_{01} - r e_{11} \right) \right\rfloor$ \textbf{decreasing} \KwTo $0$}{
        Compute $v = \frac{N^2 (d - (\alpha_0 e_{00})^2 - (\alpha_0 e_{01} + (r + kp)e_{11})^2)}{p}$\;

        Run Cornacchia's algorithm to solve $\beta_1^2 + \beta_2^2 = v$. Store solutions in array $C$\;

        \For{solution $(\beta_1, \beta_2)$ \textbf{in} $C$}{
            Set $\alpha_2 = \frac{\beta_1 - N \alpha_0 e_{02} - N \alpha_1 e_{12}}{N e_{22}}$\;
            
            Set $\alpha_3 = \frac{\beta_2 - N \alpha_0 e_{03} - N \alpha_1 e_{13} - N \alpha_2 e_{23}}{N e_{33}}$\;

            \If{$\alpha_2 \in \mathbb{Z}$ \textbf{and} $\alpha_3 \in \mathbb{Z}$}{
                Return $\alpha = \alpha_0 e_0 + \alpha_1 e_1 + \alpha_2 e_2 + \alpha_3 e_3$\;
            }
        }
    }

    Repeat from line 8 with $r = r_{-}$;

    Return $\bot$;

    \caption{Algorithm to find embeddings of quadratic order in quaternion order, for $B_{p,\infty}$, $p\neq 2$.} \label{New Quaternion Search Algorithm}
\end{algorithm}

\begin{remark}\label{QuaternionAlgRmk} Algorithm \ref{New Quaternion Search Algorithm} ... \begin{itemize}
    \item results in an embedding, but this does not necessarily define an $\mathbb{Z}[\omega]$-orientation. This is discussed in Section \ref{primitive_orientations}.
    \item can be adapted to work with any prime $p\neq 2$, not specifically $p \equiv 3 \mod 4$. For general, $B_{p, \infty} = \left(\frac{-q,-p}{\mathbb{Q}}\right)$, $q$ appears in the equations for $r, v$ and the maximum $k$, and you have to solve $\beta_1^2 + q\beta_2^2 = v$ instead of the sum of two squares. Cornacchia's still works since for $B_{p, \infty}$, $q$ and $p$ are always coprime.
    \item can be adapted to non-maximal orders. The value $N$ gains the conductor of the order as a factor.
    \item is more efficient iterating from largest $k$ to smallest, as this minimizes the values of $v$ used in Cornacchias.
    \item can be improved by using a congruence condition to rule out some cases where Cornacchia's does not have any solutions, before executing Cornacchia's. In the case $p \equiv 3 \mod 4$, we test for solutions by noting $v$ can be written as the sum of two squares if and only if, in it's prime factorization, every prime which is $3$ mod $4$ occurs an even number of times. For arbitrary $p$, a similar necessary but not sufficient congruence condition can test the splitting of $v$ to rule out some cases.
\end{itemize}\end{remark}

\subsection{Complexity Analysis of Algorithm \ref{New Quaternion Search Algorithm}}\label{ssec: quat_complexity}

In this section we give results on the asymptotic complexity of Algorithm \ref{New Quaternion Search Algorithm}, in particular giving average case results and a probabilistic worst case result. We start by attempting to give a worst case running time. Note there are a three reasons why Cornacchia's algorithm may not be efficient at finding all solutions to $\beta_1^2 + q \beta_2^2 = v$:
\begin{enumerate}
    \item It requires a factorization of $v$. To this end, we assume we have an efficient factorization oracle such as Shor's algorithm. See Section \ref{subsec:rerandomization} later on for a practical alternative to using a factorization oracle.
    \item Cornacchia's algorithm typically only refers to finding primitive solutions where $gcd(\beta_1, \beta_2)=1$. To also find imprimitive solutions we must run Cornacchia's on $\beta_1^2 + q \beta_2^2 = v/g^2$ for every square $g^2 \mid v$ and scale up the solutions $(g\beta_1, g\beta_2)$. The number of squares dividing $v$ can be subexponential in $v$. However, we can say the probability of this for random $v$ is very small, in fact asymptotically there is $\frac{\pi^2}{6} \sim 61\%$ chance $v$ is square-free.
    \item While just solving for primitive solutions, we must iterate over all the solutions Cornacchia gives. Internally Cornacchia must iterate over all solutions $x$ to the equation $x^2 \equiv -q \mod v$, where the number of solutions can be exponential in $v$ if $v$ has a large number of distinct prime factors. For example, experimentally with $p \equiv 3 \mod 4$ and $d \sim p$ we get some integers $v \sim p$ where if $v$ has lots of distinct prime factors, there can be as many as $v^{0.15} \sim p^{0.15}$ solutions which is exponential. We resolve this issue by bounding the number of factors of $v$ by the following probability estimate known as the fundamental theorem of probabilistic number theory:
\end{enumerate}
\begin{lemma}[Erd\H{o}s-Kac theorem]\label{Erdos_kac} For a natural number $n$, the number of distinct prime factors of $n$ follows the standard normal distribution with mean $\log\log n$ and standard deviation $\sqrt{\log \log n}$ as $n \to \infty$.    
\end{lemma}
This gives us the following result:

\begin{theorem}
    \label{worst_case_time}
    Let $0.5 \leq P < 1$. Assuming the heuristic that $v$ is distributed like random integers and hence the number of distinct prime factors follows Lemma \ref{Erdos_kac}, and given an efficient factorization oracle, the running time of Algorithm 5.1 is within
    \[
        O\left(T(\frac{P+1}{2}) \cdot \log(N^2 d)^{\mathcal{F}(\frac{P+1}{2})+1} \left\lceil\frac{N}{p}\sqrt{d - \frac{t^2}{4}}\right\rceil \cdot polylog(X) \right)
    \]
    with probability $P$. With $N = 2 N(I) = O(\sqrt{p})$ (Lemma \ref{N_sqrt_p}) this is
    \[
        O\left(T(\frac{P+1}{2}) \cdot \log(pd)^{\mathcal{F}(\frac{P+1}{2})+1} \left\lceil\frac{1}{\sqrt{p}}\sqrt{d - \frac{t^2}{4}}\right\rceil \cdot polylog(X) \right)
    \]
    where $X$ is the total size of the inputs, and $T(P)$ is a value large enough such that the asymptotic probability a random number has less than $T(P)$ perfect square divisors is larger than $P$. We define $\mathcal{F}$ as the inverse cumulative distribution function of the standard normal distribution where a sample is less than $\mathcal{F}(P)$ with probability $P$. For example, for $P = 0.95$ we have $\mathcal{F}(\frac{P+1}{2}) < 2$ and $T(\frac{P+1}{2}) \sim 4$.
\end{theorem}
\begin{proof}
    Steps 1-3 of the algorithm are efficient as polynomial time algorithms exist for computing Hermite normal form \cite{hafner1991asymptotically} \cite{Coppel2009}, and fixing $\alpha_0$ and solving $\alpha_1$ modulo $p$ is efficient. In Step 4 a worst case input will result in iterating $k$ over it's full range of values which is $O( \frac{1}{p e_{11}} \sqrt{d - (\alpha_0 e_{00})^2} )$, where the trace is fixed through $\alpha_0 = \frac{t}{2 e_{00}}$ so $(\alpha_0 e_{00})^2 = \frac{t^2}{4}$. And by Prop \ref{basis_properties} we have $\frac{1}{e_{11}} \leq N$. Then for each iteration over $k$, Cornacchia's algorithm is used in Step 5. To be efficient at finding primitive solutions we have to bound the number of distinct prime factors of $v$, by Lemma \ref{Erdos_kac} with probability $\frac{P+1}{2}$, $v$ is less than $\mathcal{F}(\frac{P+1}{2})$ standard deviations above the mean,
    \[
        \text{Number of factors of } v \leq \log\log(v) + \mathcal{F}(\frac{P+1}{2}) \sqrt{\log\log(v)}
    \]
    hence it is certainly true that
    \[
        \text{Number of factors of } v \leq (\mathcal{F}(\frac{P+1}{2}) + 1) \log\log(v).
    \]
    Then it follows that the number of square roots found in Cornacchia's algorithm is less than $O(2^{(\mathcal{F}(\frac{P+1}{2}) + 1) \log\log(v)}) = O(\log(v)^{(\mathcal{F}(\frac{P+1}{2}) + 1)})$, so we can bound the running time of Cornacchia by $O(\log(v)^{(\mathcal{F}(\frac{P+1}{2}) + 1)}) \cdot polylog(v)$, and clearly for each $v$ we have $v \leq N^2 d < O(pd)$ and hence $polylog(v) = polylog(X)$. The final consideration is for finding imprimitive solutions using Cornacchia's algorithm which requires repeating for every square dividing $v$. By definition this is at most $T(\frac{P+1}{2})$ repetitions with probability $\frac{P+1}{2}$. The probability both this condition and $v$ having the correct number of factors is at least $\frac{P+1}{2} + \frac{P+1}{2} -1 = P$.
\end{proof}

Now we give a result for the average case running time:

\begin{theorem}
    \label{quat_avg_time}
    Making the following assumptions, regarding iterating over $k$:
    \begin{itemize}
        \item Each $v_k$ is distributed like random integers and hence the expected number of distinct prime factors is $\log\log(v_k)$ by Lemma \ref{Erdos_kac}, and there is a high probability it only has a few square divisors.
        \item Additionally, the probability each $v_k$ is the sum of two squares is independent and at least the probability a random integer less than $(Nd)^2$ is the sum of two squares.
        \item The first solution to Cornacchia's algorithm has $\beta_1, \beta_2$ uniformly distributed modulo $e_{22} N$ and $e_{33} N$ respectively.
    \end{itemize}
    Then given an efficient factorization oracle, in the case $p \equiv 3 \mod 4$, the average case running time of Algorithm \ref{New Quaternion Search Algorithm} is $O(\min\{ N^3, \left\lceil \frac{N}{p}\sqrt{d - \frac{t^2}{4}}\right\rceil\} \times polylog(X))$ and substituting $N = O(\sqrt{p})$ from Lemma \ref{N_sqrt_p} it is    
    $O(\min\{ p \sqrt{p}, \left\lceil \frac{1}{\sqrt{p}}\sqrt{d - \frac{t^2}{4}}\right\rceil\} \times polylog(X))$ where $X$ is the total size of all inputs.
\end{theorem}

We use the following Lemma:
\begin{lemma}[Landau 1908]
    The number of integers representable as the sum of two squares from from $0$ to $n \in \mathbb{N}$ is the limit $C \frac{n}{\sqrt{\log n}}$ as $n \to \infty$, where $C \approx 0.764$ is the Landau-Ramanujan constant. Hence for sufficiently large $n$, the number of integers representable is greater than $\frac{1}{2} \frac{n}{\sqrt{\log n}}$. (In fact, experimentally this appears true for all $n \geq 0$).
\end{lemma}

\begin{proof}[Proof of Theorem \ref{quat_avg_time}]
    It's clear the running time is the product of the number of iterations over $k$, and the running time of Cornacchia's, because all other operations are polynomial time. In the case $p \equiv 3 \mod 4$ we are solving the sum of two squares, hence by Landau, and using the first assumption, we expect (for sufficiently large $d$) less than $2 \sqrt{\log((N d)^2)}$ iterations until we find a $k$ where Cornacchia's gives at least one solution.

    Now recall that finding one solution to Cornacchia's algorithm is not necessarily enough, since we need to satisfy the conditions $\alpha_2, \alpha_3 \in \mathbb{Z}$. This amounts to checking:
    \begin{align*}
        \beta_1 - N \alpha_0 e_{02} - N \alpha_1 e_{12} & \equiv 0 \mod e_{22} N \\
        \beta_2 - N \alpha_0 e_{03} - N \alpha_1 e_{13} - N \alpha_2 e_{23} & \equiv 0 \mod e_{33} N
    \end{align*}

    Therefore, by the second assumption we expect to have an integral solution after $e_{22} N \times e_{33} N$ solutions from Cornacchias. Noting that $e_{22} e_{33} \leq N/2$ from Prop \ref{basis_properties}, that's $N^3 /2$ solutions. In total we expect $O(N^3 \sqrt{\log(N d)})$ iterations of $k$. This is bounded above by the maximum number of iterations from Theorem \ref{worst_case_time}. Finally, Cornacchia's algorithm uses the efficient factorization oracle to factorize each $v_k$ and on average $v_k$ is expected to have $\log \log(v_k)$ distinct prime factors by the first assumption, hence internally Cornacchia's computes at most $2^{\log \log(v_k)} = \log(v_k)$ square roots which is efficient. Then to find imprimitive solutions, we only repeat Cornacchia's a constant number of times as the expected number of squares dividing $v$ is very small. Overall this takes time polylog in each $v_k \leq N^2 d = O(pd)$, so this term can be incorperated into $polylog(X)$.
\end{proof}

From this we observe the following:
\begin{corollary}[Efficient for orders close to $\mathcal{O}_0$]
    Given an efficient factorization oracle, consider the algorithm applied to the order $\mathcal{O}_0$. Here we have $N=2$, hence the algorithm is efficient; the average case running time is $polylog(X)$. For orders close to $\mathcal{O}_0$, such as a curve an $l$-isogeny from the curve with $j$-invariant $1728$, we gain a factor of $l$ in $N$, hence for small $l$ the algorithm is still efficient. However with each step from $\mathcal{O}_0$, $N$ gains a factor of the degree of the isogeny, so it gets exponentially harder the further you walk, until we reach the point $N \sim \sqrt{p}$.
\end{corollary}

For completeness, we now consider the case of arbitrary primes $p \neq 2$. Then the quaternion algebra containing order $\mathcal{O}$ is $B_{p, \infty} = \left( \frac{-q, -p}{\mathbb{Q}} \right)$ where $q$ is either $2$ or a prime $q \equiv 3 \mod 4$ with Legendre symbol $\left( \frac{q}{p} \right) = -1$.

By the same argument as Theorem \ref{worst_case_time}, with high probability $P$ the worst case running time is within
\[
    O\left(T(\frac{P+1}{2})\cdot\log(N^2 d)^{\mathcal{F}(\frac{P+1}{2})+1} \left\lceil\frac{N}{p}\sqrt{\frac{1}{q}(d - \frac{t^2}{4})}\right\rceil \cdot polylog(X) \right)
\]
which is the same as before except the additional factor of $\frac{1}{\sqrt{q}}$ appears requiring more iterations over $k$. Then from Lemma \ref{basis_properties} part (2), for a different value of $K$, we get $N = 2q N(I) = O(q\sqrt{p})$. Applying this along with Lemma \ref{N_sqrt_p} gives:
\[
    =
    O\left(T(\frac{P+1}{2})\cdot\log(q^2 p d)^{\mathcal{F}(\frac{P+1}{2})+1} \left\lceil\sqrt{\frac{q}{p}(d - \frac{t^2}{4})}\right\rceil \cdot polylog(X) \right)
\]

Typically $q$ is treated as a constant, so asymptotically the complexity is the same, however, $q$ is actually unbounded; you can construct a prime such that the minimum value for $q$ is larger than a given threshold. Hence we treat $q$ as a variable in our analysis.

We have a similar variation on the average time complexity, however the proof is more complex:

\begin{theorem}\label{quat_avg_time_arb_p}
    Given an efficient factorization oracle, for arbitrary $p \neq 2$, making the same assumptions as in Theorem \ref{quat_avg_time} (replacing sum of two squares with $x^2 + q y^2$), the average case running time of Algorithm \ref{New Quaternion Search Algorithm} is
    \[
        O\left(\min \left\{ \frac{q^2 p \sqrt{p}}{C(-4q)} , \left \lceil \sqrt{\frac{q}{p}(d - \frac{t^2}{4})} \right \rceil \right\} \cdot polylog(X)\right)
    \]
    where $C$ is a special function generalising the Landau-Ramanujan constant, and $X$ is the total size of all inputs.
\end{theorem}

\begin{proof}
    The proof is the same as Theorem \ref{quat_avg_time}, except instead of solving the sum of two squares, we are solving $x^2 + q y^2 = v_k$ using Cornacchia's algorithm. This means in the proof we cannot use Landau's result on the sum of two squares. However, Landau's result generalises.

    Bernays proved that the number of integers between $0$ and $n$ represented by a binary quadratic form $f(x,y)$ converges to $C(\Delta_f) \frac{n}{\sqrt{\log n}}$ as $n \to \infty$, where $\Delta_f$ is the discriminant of the form $f$ and $C(\Delta_f)$ is a constant depending on $\Delta_f$ \cite{bernays1912darstellung}. In our case $f(x,y) = x^2 + q y^2$, hence $\Delta_f = -4q$. 

    Additionally the bound $e_{22} e_{33} \leq \frac{N}{2}$ from the proof of Theorem \ref{quat_avg_time} becomes $e_{22} e_{33} \leq \frac{N}{2q}$ by Proposition \ref{basis_properties}. And we have $N = 2q N(I) = q \cdot O(1) \sqrt{p}$.
\end{proof}

For an explicit formula for $C(-4q)$, see the results of Moree and Osburn \cite{moree2006two}, and for a summary of results on $C(\cdot)$ see the work of Brink, Moree and Osburn \cite{brink2011principal}.

Next we note how the complexity changes in other contexts:

\begin{remark}[Suborders]
    Suppose you are given a quaternion order $\mathcal{O} \subset B_{p, \infty}$ which is not necessarily maximal. As stated in Remark \ref{QuaternionAlgRmk}, Algorithm~\ref{New Quaternion Search Algorithm} still works. The complexity is the same with the subtlety that $N$ is multiplied by the conductor of the suborder within a maximal order.
\end{remark}

\begin{remark}[$p$-extremal orders]
    Suppose $\mathcal{O}$ is a $p$-extremal order, and has suborder $R + jR \subseteq \mathcal{O}$ and we are trying to find an embedding into this suborder. For $\omega$ a generator of $R$ and $\alpha = \alpha_0 + \alpha_1 \omega$, $\alpha' = \alpha_0' + \alpha_1' \omega \in R$, the norm equation is:
\[
nrd(\alpha + j \alpha') = f(\alpha_0, \alpha_1) + p f(\alpha_0', \alpha_1') = d
\]
    where $f$ is a binary quadratic form of discriminant $disc(R)$. The approach in the KLPT algorithm \cite{KLPT}[Section 3.2] randomly samples $\alpha_0'$ and $\alpha_1'$ until $d - p f(\alpha_0', \alpha_1') = q$ is a prime which is split in $R$ where the ideal factors of $(q)$ are principal, and hence it's generator gives a solution to $q = f(\alpha_0, \alpha_1)$. Assuming sampled integers $q$ satisfy the distribution of primes less than $d$, this takes roughly $2 h(R) \log(d)$ iterations. Note that we require $d$ to be large enough for the set $d - p f(x, y)$ to contain enough primes.

    Our algorithm is very similar but works in reverse. Also assuming $d$ is sufficiently large ($d>p^{2+\epsilon}$), we sample $k$ until $q = \frac{d - f(\alpha_0, \alpha_1)}{p} \in \mathbb{Z}$ has a solution to the equation $q = f(\alpha_0', \alpha_1')$. By the same argument, if we wait until $q$ is a prime, split in $R$, with principal factor of $(q)$, we are guaranteed a solution, so we also expect $2 h(R) \log(d)$ iterations which is efficient.
\end{remark}

Finally note we do not necessarily need a factorization oracle using the technique presented in the next section.

\subsection{Rerandomization and Small Discriminant}\label{subsec:rerandomization}
Consider the special case of small discriminant orders $\Z[\omega]$ in Algorithm~\ref{New Quaternion Search Algorithm}. Previously, the best known efficient algorithm, as stated in \cite{Bweso_orientations}  was simply to look for small vectors in $\mathcal{O}$. This works for $|\disc(\Z[\omega])| < 2\sqrt{p} - 1$ and is stated below. Note that an alternative approach that heuristically works for $|\disc(\Z[\omega])|<p^{0.8}$ is outlined in \cite{isogenyfinding}.

\begin{proposition}\cite[Proposition 6]{Bweso_orientations}\label{Small Norms}
    Assume that $|\disc(\mathbb{Z}[\omega])|<2\sqrt{p}-1$. Then, there is a probabilistic polynomial time algorithm that solves Problem~\ref{Problem 4}.
\end{proposition}

Under certain heuristics, we can rerandomize Algorithm \ref{New Quaternion Search Algorithm} by considering isomorphic orders, potentially in different representations of $B_{p, \infty}$, to avoid factoring and bound the denominator $N < O(\sqrt{p})$. The result of this is a corollary (Corollary~\ref{Corollary: Efficient for d < p}) which gives a heuristic polynomial time algorithm for solving Problem \ref{Problem 4} for $\disc(\Z[\omega])$ in $O(p)$, or deciding that no solution exists.

The first step is to bound the number of values to try Cornacchia on in Algorithm \ref{New Quaternion Search Algorithm}:

\begin{lemma} \label{lem:bounded_number_iterations}
    Fix a positive integer $M$. As in the context of Algorithm \ref{New Quaternion Search Algorithm}, take $d = \nrd(\omega)$, $N$ as the common denominator of the HNF basis, and $k$ the variable we iterate over. Suppose we have
    \[
        d \leq \frac{q p^2 (\frac{M}{2}-1)^2}{N^2}
    \]
    then the algorithm consists of computing HNF and fast polynomial time arithmetic, and at most $M$ executions of Cornacchia's algorithm. By Cornacchia's algorithm, here we mean mean finding all solutions to $x^2 + qy^2 = v$, not just primitive ones.
\end{lemma}
\begin{proof}
    Naively we use Cornacchia's algorithm once for every $k$ we check. Iterating over $k$ happens twice, once using $r_{+}$ and once using $r_{-}$, therefore each time we want at most $\lfloor M/2 \rfloor$ iterations. As in proof of Theorem \ref{worst_case_time}, and taking into account general $p$, we can consider the maximum number of iterations over $k$ for each $r$ and bound it by $M/2$:
    \[
        \left\lfloor \frac{1}{p e_{11}}\left(\sqrt{\frac{d - (\alpha_0 e_{00})^2}{q}} - \alpha_0 e_{01} - r e_{11} \right) \right\rfloor + 1 \leq \frac{M}{2}
    \]
    which, since $e_{01} \geq 0$ by Prop \ref{basis_properties}, is certainly true if
    \[
        \frac{1}{p e_{11}}\left(\sqrt{\frac{d}{q}} \right) + 1 \leq \frac{M}{2}
    \]
    and noting $e_{11} \geq \frac{1}{N}$, we get the condition:
    \[
        d \leq \frac{q p^2 (\frac{M}{2}-1)^2}{N^2}
    \]
\end{proof}

From this, we obtain a result about $\disc(\Z[\omega])$ since we can translate generator $\omega$ to either $\sqrt{-\disc(\Z[\omega])}/2$ or $(1+\sqrt{-\disc(\Z[\omega])})/2$ hence $N(\omega)=d\leq(|\disc(\Z[\omega])|+1)/4$. Recalling that we can bound $N$, the denominator of $\mathcal{O}$, to $O(\sqrt{p})$, we see that Lemma \ref{lem:bounded_number_iterations} says that when $\disc(\Z[\omega])$ is in $O(p)$, and assuming $q$ in $O(1)$, the only potentially expensive part in Algorithm \ref{New Quaternion Search Algorithm} is Cornacchia on a constant number of instances.
 
 The general idea of our rerandomized version is then the common technique of only running Cornacchia on ``good'' instances. However, if the discriminant is small, then the embedding is with large probability unique, hence we might end up discarding the correct solution. Therefore, we need to rerandomize until \textit{all} $O(1)$ Cornacchia instances are good, before the algorithm can be sure that no embedding exists.

We define a ``good'' instance to be $x^2 + qy^2 = v_k$ where $v_k$ can be factorized in polynomial time, has $O(\log\log(N^2 d))$ distinct prime factors, and $O(\log(N^2 d))$ square divisors. The set of prime numbers satisfies these conditions, and with the heuristic that the events of each $v_k$ being prime are independent and follow from the density of primes, we expect at most some constant multiple of $\log(N^2 d)^{C}$ iterations until $C$ of the Cornacchia instances are primes. With $C = O(1)$ instances from the small discriminant condition, this is efficient.

Now we discuss how we rerandomization the order. Let $\mathcal{O}_0 \subseteq B_{p, \infty}$ be a maximal order with negligible denominator $K$ (e.g. the ``standard'' maximal order from Proposition \ref{O_0}). As has been pointed out, any maximal order $\mathcal{O} \subseteq B_{p, \infty}$ will have denominator bounded by $KN$, where $N$ is the norm of the connecting ideal from $\mathcal{O}_0$ to $\mathcal{O}$. Hence we can consider any equivalent ideal $J = I\gamma$ of small norm, and instead solve the problem in the isomorphic order $\mathcal{O}_R(J) = \gamma^{-1}\mathcal{O}\gamma$, before transferring the solution back to $\mathcal{O}$.

As we rerandomize, we might need to try many distinct $J \sim I$ of small norm. Heuristically, for random orders $\mathcal{O}$, we can expect there to be an abundance of small, equivalent ideals (i.e. with $N(J)$ in $O(\sqrt{p})$). The problem is that this heuristic fails completely if there exists some equivalent ideal $I'$ with $\nrd(I') \ll \sqrt{p}$, i.e. in the case that $\mathcal{O}$ is too ``close'' to $\mathcal{O}_0$. We can fix this by considering other maximal orders $\mathcal{O}_0'$ with negligible denominator in other representations of $B_{p, \infty}$.

Specifically, we can generate representations \hbox{$B_i = (-q_i, -p \mid \Q)$}, where we take $q_i$ to be the smallest primes satisfying 
\[
q_i \equiv 3 \pmod{4}, \quad \left(\frac{-q_i}{p}\right) = -1.
\]
These quaternion algebras are indeed ramified at $p$ and $\infty$ \cite[Proposition 14.2.7]{Voight}. In each of these representations, regardless of congruence conditions on $p$, we can take a standard choice of maximal order $\mathcal{O}_{0,i}$ with denominator $2q_i$ as

\[
\mathcal{O}_{0,i} := \mathbb{Z} \oplus \mathbb{Z}\frac{1 + i}{2} \oplus \mathbb{Z}j \oplus \mathbb{Z}\frac{(1+i)j}{2q_i},
\]
\cite[Exercise 15.5]{Voight}. Heuristically, these choices of maximal orders of the different quaternion algebra representations are ``independent'' in the sense that there is no reason that these should be close to each other, so it is enough to try a small, fixed number of such orders, as (heuristically), the probability that $\mathcal{O}$ is close to all $\mathcal{O}_{0, i}$ is negligible.

Finally, the explicit isomorphisms $B_i \cong B_j$ are also easy to find and compute, using \cite[Lemma 10]{DBLP:journals/iacr/EriksenPSV23}. The whole algorithm is summarized in Algorithm~\ref{alg:Rerandomization}.

\begin{algorithm}
    \SetAlgoLined
    \KwData{A $\Z[\omega]$-orientable maximal order $\mathcal{O} \subset B$, where $B$ is a quaternion algebra ramified at $p$ and $\infty$. $r \in \mathbb{N}$, a maximum number of randomizations to try.}
    \KwResult{An element $\alpha \in \mathcal{O}$, which defines an embedding $\iota: \mathbb{Z}[\omega] \hookrightarrow \mathcal{O}$ by $\omega \mapsto \alpha$.}

    Compute $r$ representations $B_i = (-q_i, -p, \Q)$ for $q_i \in O(1)$ of $B_{p, \infty}$\;
    \For{$i = 1, \dots, r$}{
        Set $\mathcal{O}_{0,i} \subseteq B_i$ to be a maximal order with denominator dividing $2q_i$\;
        
        Compute an isomorphism $\varphi_i : B \rightarrow B_i$\;
        
        Set $\mathcal{O}_i = \varphi(\mathcal{O})$\;
        
        Let $I$ be a connecting $(\mathcal{O}_{0,i}, \mathcal{O}_i)$-ideal\;
        \For{$J = I\gamma$, with $N(J)$ in $O(\sqrt{p})$}{
            Compute $\beta$ by running Algorithm~\ref{New Quaternion Search Algorithm} on $\mathcal{O}_R(J)$, only running Cornacchia on prime numbers\;
            \If{$\beta \neq \bot$}{
                Set $\alpha' = \gamma \beta \gamma^{-1}$\;
                
                Return $\varphi_i^{-1}(\alpha')$\;
            }
        }
    }
    
    \caption{Rerandomized version of Algorithm~\ref{New Quaternion Search Algorithm}} \label{alg:Rerandomization}
\end{algorithm}

This gives the following corollary:
\begin{corollary} \label{Corollary: Efficient for d < p}
    For arbitrary $p \neq 2$, given a maximal order $\mathcal{O} \subseteq B_{p, \infty}$, and a quadratic order $\mathfrak{O}$ with $|\disc(\mathfrak{O})|$ in $O(p)$, Algorithm~\ref{alg:Rerandomization} heuristically computes an $\mathfrak{O}$-orientation of $\mathcal{O}$, or decides that none exists, in probabilistic polynomial time in $\log(p)$, under the heuristics discussed above.
\end{corollary}
\begin{proof}
    By Lemma \ref{lem:bounded_number_iterations}, and the subsequent discussion, the only potentially expensive step of running Algorithm~\ref{New Quaternion Search Algorithm} are the (constant number of) Cornacchia instances. By only running the Cornacchia on prime instances (or more generally, ``good'' instances), we expect to have to run Algorithm~\ref{New Quaternion Search Algorithm} at most $O(\log(p)^{O(1)})$ times by the prime number theorem, and the fact that $v$ is in $O(p)$. Further, it is clear that if all $O(1)$ values to try Cornacchia on is prime, and a solution is still not found, a solution cannot exist, hence the algorithm can conclude that no solution exists.
\end{proof}

\subsection{From Embeddings to Orientations of Superorders}\label{ssec: quat_FromEmbeddings}

Algorithms~\ref{New Quaternion Search Algorithm} and \ref{alg:Rerandomization} find all possible embeddings $\iota : \mathbb{Z}[\omega] \hookrightarrow \mathcal{O}$. Every embedding gives an orientation for some order, namely an $\mathbb{Z}[\omega']$-orientation where $\mathbb{Z}[\omega] \subseteq \mathbb{Z}[\omega']$.
We split embeddings into two cases: those which give a $\mathbb{Z}[\omega]$-orientation (where $\mathbb{Z}[\omega] = \mathbb{Z}[\omega']$) we call \textit{primitive embeddings}, and those which give superorder orientations $\mathbb{Z}[\omega] \subsetneq \mathbb{Z}[\omega']$ which we call \textit{imprimitive embeddings}. 
Finding primitive embeddings solves Problem~\ref{Problem 4}, and we consider this question in Section~\ref{primitive_orientations}. In this section, we focus on embeddings which are orientations by a strict superorder. First we explain how to differentiate between primitive and imprimitive embeddings.

For any element $\alpha \in \mathcal{O}$, write $\tilde \alpha$ for its class in the lattice $\mathcal{O}/\Z$.
Consider the discriminant form 
\begin{align*}
    \Delta : \mathcal{O}/\Z &\longrightarrow \Q\\
    \tilde{\alpha} &\longmapsto \Tr(\alpha) - 4 \nrd(\alpha),
\end{align*}
an integral quadratic form of rank $3$.
It does not depend on the choice of a representative $\alpha$ of the class $\tilde \alpha$, and we also write $\Delta(\alpha)$.
Let $d$ be an integer.
We say that a solution $\alpha \in \mathcal{O}$ of $\Delta(\alpha) = d$ is \emph{primitive} if $\tilde{\alpha}$ is a primitive element of the lattice $\mathcal{O}/\Z$, i.e., it is not of the form $\tilde{\alpha} = b\tilde{\beta}$ for some element $\tilde\beta \in \mathcal{O}/\Z$ and integer $b > 1$. We now show primitive solutions correspond directly to primitive orientations.

\begin{lemma}
An element $\alpha \in \mathcal{O}$ is a primitive solution of $\Delta(\alpha) = d$ if and only if $\Z[\alpha] \subseteq \mathcal{O}$ is a primitive embedding.
\end{lemma}

\begin{proof}
Suppose $\alpha$ is an imprimitive solution of $\Delta(\alpha) = d$, i.e., there are integers $a$ and $b>1$ such that $(\alpha - a)/b \in \mathcal{O}$. Then, $\Z[\alpha] \subsetneq \Z[(\alpha - a)/b] \subseteq \mathcal{O}$, hence $\Z[\alpha] \subseteq \mathcal{O}$ is not primitive. Conversely, suppose $\Z[\alpha] \subseteq \mathcal{O}$ is not primitive, so there exists $\beta \in (\Q[\alpha] \cap \mathcal{O})\setminus \Z[\alpha]$. There exists integers $a$ and $b>1$ such that $\alpha = a + b\beta$. In particular, $\tilde\alpha = b\tilde\beta$, so $\alpha$ is not a primitive solution.
\end{proof}

Now we know primitive embeddings come from primitive solutions, we can determine if an embedding is primitive, and if not extend it to its superorder, very fast using a $\gcd$ computation:

\begin{lemma}
    Given a maximal order $\mathcal{O}$ with basis $e_0, e_1, e_2, e_3$ and an element $\alpha \in \mathcal{O}$ of trace $t=\Tr(\omega)$ and norm $d=\nrd(\omega)$, there is a polynomial time algorithm on order $\mathcal{O}$ which:
    \begin{itemize}
        \item determines whether embedding $\iota$ defined by extending $\omega \mapsto \alpha$ is a primitive or imprimitive embedding of $\mathbb{Z}[\omega]$,
        \item and if it's imprimitive outputs $(a, b, \alpha')$ defining a superorder $\mathbb{Z}[\frac{\omega - a}{b}] \supset \mathbb{Z}[\omega]$ which $\iota$ can be extended to, through the map $\frac{\omega - a}{b} \mapsto \alpha'$.
    \end{itemize}
\end{lemma}
\begin{proof}
    First convert the upper diagonal basis $e_0, e_1, e_2, e_3$ of $\mathcal{O}$ into an lower diagonal basis $f_0, f_1, f_2, f_3$,
    \begin{equation}\begin{split}
        \label{hnf}
        \mathcal{O} = \langle f_0, f_1, f_2, f_3 \rangle_{\mathbb{Z}} = \langle & f_{00}, \\
        & f_{10} + f_{11}i, \\
        & f_{20}i + f_{21}i + f_{22}j, \\
        & f_{30} + f_{31}i + f_{32}j + f_{33}k \rangle_{\mathbb{Z}}
    \end{split}\end{equation}
    and to compute a function applying the change of basis transformation taking coefficients of $e_i$s onto coefficients of $f_i$s. This is polynomial time as a variant of the Hermite normal form algorithm, and can be seen as precomputation. Then given a solution $\alpha$, we change the basis to obtain:
    \[
        \alpha = \gamma_0 f_0 + \gamma_1 f_1 + \gamma_2 f_2 + \gamma_3 f_3
    \]
    Since $\mathcal{O}$ is a ring we have $1 \in \mathcal{O}$, and every norm is integral so it cannot contain a rational number less than one. Hence $f_{0} = f_{00} = 1$. For $\alpha$ to be a primitive solution there should be no $a,b \in \mathbb{Z}$ with $b > 1$ such that $\alpha - a = b\tau$ where $\tau \in \mathcal{O}$. Equivalently, for any $a$, when expressing $\alpha - a$ in terms of $f_i$s, the coefficients should not all be divisible by any $b > 1$. Note that we have:
    \[
        \alpha - a= (\gamma_0 - a) f_0 + \gamma_1 f_1 + \gamma_2 f_2 + \gamma_3 f_3
    \]
    Where $a = \gamma_0$ can be chosen, setting the first coefficient to zero. Then the solution is primitive if and only if $\gamma_1, \gamma_2, \gamma_3$ share no factor. This can be checked with a $\gcd$ computation. Note that if the solution is imprimitive, so we have $gcd(\gamma_1, \gamma_2, \gamma_3) = b > 1$, we return $(\gamma_0, b, \frac{\gamma_1 f_1 + \gamma_2 f_2 + \gamma_3 f_3}{b})$ defining an embedding giving an $\mathbb{Z}[\frac{\omega - \gamma_1}{b}]$-orientation.
\end{proof}

Algorithm \ref{Primitive Check Algorithm} is a concise version of this. Hence using Algorithm \ref{New Quaternion Search Algorithm} for the embedding we find we always get a superorder $\mathbb{Z}[\omega']$-orientation without effecting asymptotic time complexities. Furthermore, if we iterate over the full range of $k$ in Algorithm \ref{New Quaternion Search Algorithm}, we find all embeddings and hence all superorder orientations.

\begin{algorithm}
    \SetAlgoLined
    \KwData{Element $\alpha \in \mathcal{O}$ in terms of a basis $e_i$ which is a solution to the discriminant form of $\omega$.}
    \KwResult{True if it is a primitive solution, otherwise False and output $(a,b,\alpha')$ where $\alpha'$ is a primitive solution to the discriminant form of $\frac{\omega - a}{b}$.}
    Precompute lower diagonal Hermite normal form basis $f_i$. And store operations performed giving a linear change of basis transformation matrix $M$\;

    Apply transformation $M$ to $\alpha$, giving coefficients $\gamma_i$ such that $\alpha = \gamma_0 f_0 + \gamma_1 f_1 + \gamma_2 f_2 + \gamma_3 f_3$\;

    Let $S = \{\gamma_1, \gamma_2, \gamma_3\} \setminus \{0\}$\;

    \lIf{ $|S| = 0$}{
        \Return{False, $(\gamma_0, \infty, 0)$}
    }

    Compute $g = \gcd(S)$ using Euclidean algorithm\;

    \uIf{ $g == 1$}{
        \Return{True}\;
    }
    \Else{
        \Return{False, $(\gamma_0, g, \frac{\gamma_1}{g} f_1 + \frac{\gamma_2}{g} f_2 + \frac{\gamma_3}{g} f_3)$}\;
    }

    \caption{Checking solution is primitive, and getting primitive superorder orientation.}\label{Primitive Check Algorithm}
\end{algorithm}

\subsection{Finding $\mathbb{Z}[\omega]$-orientations - Solving Problem \ref{Problem 4}}
\label{primitive_orientations}

To solve Problem \ref{Problem 4} we must find \textit{primitive} embeddings giving $\mathbb{Z}[\omega]$-orientations. We have Algorithm \ref{New Quaternion Search Algorithm} for finding embeddings, and we have Algorithm \ref{Primitive Check Algorithm} which can check if an embedding is primitive.

To combine them, we modify Algorithm \ref{New Quaternion Search Algorithm} to include the pre-computation of basis $f_i$ and the change of basis transformation at the start, then when each solution is found, check if it is primitive using Algorithm \ref{Primitive Check Algorithm} and only stop searching if it is. The worst case running time doesn't change, since finding a primitive embedding takes at most as long as finding all embeddings. However, the average case running time does increase, heuristically by the total number of solutions divided by the number of primitive solutions. We now provide a further heuristic argument that this ratio can be bounded above.

Let $f(\gamma_1, \gamma_2, \gamma_3) = \lambda$ be the solution to the ternary quadratic norm form of the order defined in \ref{hnf} with fixed trace. We have shown the solution is primitive if and only if $gcd(\gamma_1, \gamma_2, \gamma_3) = gcd(|\gamma_1|, |\gamma_2|, |\gamma_3|) = 1$.

Consider the rational solution $x_1 = x_2 = x_3 \in \mathbb{Q}$ then as it is rational we can write $f(x_1, x_2, x_3) = w x_1^2 = \lambda$ for some $w \in \mathbb{Q}$, so $|x_1| = |\sqrt{\lambda / w}|$. Since the norm form is positive definite, this means if one variable were to increase, another must decrease in absolute value. Therefore $\min(|\gamma_1|, |\gamma_2|, |\gamma_3|) \leq \lfloor \sqrt{\lambda / w} \rfloor$.

Now reconsider our integral solution. Suppose the solution is not primitive so $gcd(|\gamma_1|, |\gamma_2|, |\gamma_3|) \neq 1$, then there is a prime number $\geq 2$ that divides all three numbers. This prime factor must be in the set $S = \{2, 3, 5, ...,\lfloor \sqrt{\lambda / w} \rfloor\} \cap \{\text{Primes } p\}$ as it must divide the smallest of these three numbers.

Heuristically, we assume that $\gamma_1, \gamma_2, \gamma_3$ are distributed like random numbers in the sense that some $q \in S$ divides one of them with uniformly random probability $1/q$. And assume independence of the probabilities of different factors $q_1, q_2 \in S$ occurring. Then the probability $2$ divides all three numbers is $1/2^3$, the probability $3$ divides them is $1/3^3$, and the probability any $q \in S$ divides them is $1/q^3$. Combined, the probability a number in $S$ divides all three is:
\[
    \mathbb{P}[(\gamma_1, \gamma_2, \gamma_3) \text{ imprimitive}] =  \sum_{\text{primes } q\in S} \frac{1}{q^3} \leq \sum_{\text{all primes } q \in \mathbb{N}} \frac{1}{q^3} = P(3) \leq 0.175
\]
where $P(3)$ is the prime zeta function at 3. Hence the probability a random solution is primitive is over $80\%$ so if we find $5$ independent solutions we would expect at least one to be primitive. Therefore assuming the heuristics above, if we modify algorithm \ref{New Quaternion Search Algorithm} to ignore imprimitive solutions, the average running time should only increase by at most a factor of $5$.

Note that this argument makes some strong assumptions. Experimentally for some parameters we see the probability the first solution found is primitive is around $80\%$, but on other parameter choices it is considerably lower. With all parameters we tested, we found the probability is always over $50\%$, which still suggests the average running time is only worsened by a small factor, but it gives reason to doubt these assumptions. In particular consider independence. Existence of embeddings come with symmetry hence we may find two solutions where there is only a change of sign in the defining formulae. This means the probability $q$ divides a coefficient of one solution might have a strong dependence on whether $q$ divides the coefficient of the second solution. We leave a more complete analysis of the probability of finding primitive embeddings to future research.

\bibliographystyle{abbrv}  
\bibliography{biblio} 

\begin{thebibliography}{10}

\bibitem{arpin2022orienteering}
S.~Arpin, M.~Chen, K.~E. Lauter, R.~Scheidler, K.~E. Stange, and H.~T. Tran.
\newblock Orienteering with one endomorphism.
\newblock {\em La Matematica}, 2023.
\newblock \url{https://link.springer.com/article/10.1007/s44007-023-00053-2}.

\bibitem{ACLSST2022_orientations}
S.~Arpin, M.~Chen, K.~E. Lauter, R.~Scheidler, K.~E. Stange, and H.~T.~N. Tran.
\newblock Orientations and cycles in supersingular isogeny graphs, 2022.
\newblock To appear in the Proceedings of Women in Number Theory 5.

\bibitem{isogenyfinding}
B.~Bencina, P.~Kutas, S.-P. Merz, C.~Petit, M.~Stopar, and C.~Weitkämper.
\newblock Improved quantum algorithms for finding fixed-degree isogenies
  between supersingular elliptic curves.
\newblock personal communication, 2023.

\bibitem{bernays1912darstellung}
P.~Bernays.
\newblock {\em {\"U}ber die Darstellung von positiven: ganzen Zahlen durch die
  primitiven, bin{\"a}ren quadratischen Formen einer nicht-quadratischen
  Diskriminante}.
\newblock Dieterich, 1912.

\bibitem{SqrtVelu}
D.~J. Bernstein, L.~De~Feo, A.~Leroux, and B.~Smith.
\newblock Faster computation of isogenies of large prime degree.
\newblock In {\em Proceedings of the Fourteenth Algorithmic Number Theory
  Symposium}, pages 39--55, 798 Evans Hall \#3840, c/o University of
  California, Berkeley CA 94720-3840, 2020. MSP.

\bibitem{Isogeny_comp}
A.~Bostan, F.~Morain, B.~Salvy, and E.~Schost.
\newblock Fast algorithms for computing isogenies between elliptic curves.
\newblock {\em Mathematics of Computations}, (77):1755--1778, 2008.

\bibitem{brink2011principal}
D.~Brink, P.~Moree, and R.~Osburn.
\newblock Principal forms x\^{}2 + ny\^{}2 representing many integers.
\newblock In {\em Abhandlungen aus dem Mathematischen Seminar der
  Universit{\"a}t Hamburg}, volume~81, pages 129--139. Springer, 2011.

\bibitem{GNFS}
J.~P. Buhler, H.~W. Lenstra, and C.~Pomerance.
\newblock Factoring integers with the number field sieve.
\newblock In A.~K. Lenstra and H.~W. Lenstra, editors, {\em The development of
  the number field sieve}, pages 50--94, Berlin, Heidelberg, 1993. Springer
  Berlin Heidelberg.

\bibitem{CastryckDecruSIDH}
W.~Castryck and T.~Decru.
\newblock {An efficient key recovery attack on {SIDH} (preliminary version)}.
\newblock Cryptology ePrint Archive, Paper 2022/975, 2022.

\bibitem{castryck2018csidh}
W.~Castryck, T.~Lange, C.~Martindale, L.~Panny, and J.~Renes.
\newblock {CSIDH}: an efficient post-quantum commutative group action.
\newblock In {\em Advances in Cryptology--ASIACRYPT 2018: 24th International
  Conference on the Theory and Application of Cryptology and Information
  Security, Brisbane, QLD, Australia, December 2--6, 2018, Proceedings, Part
  III 24}, pages 395--427. Springer, 2018.

\bibitem{CGL}
D.~X. Charles, E.~Z. Goren, and K.~E. Lauter.
\newblock Cryptographic hash functions from expander graphs.
\newblock {\em J. Cryptology}, 22(1):93--113, 2009.

\bibitem{OSIDH}
L.~Col\`{o} and D.~Kohel.
\newblock Orienting supersingular isogeny graphs.
\newblock Cryptology ePrint Archive, Report 2020/985, 2020.
\newblock \url{https://eprint.iacr.org/2020/985}.

\bibitem{Coppel2009}
W.~A. Coppel.
\newblock {\em The Arithmetic of Quadratic Forms}, pages 291--326.
\newblock Springer New York, New York, NY, 2009.

\bibitem{couveignes2006hard}
J.-M. Couveignes.
\newblock Hard homogeneous spaces.
\newblock {\em Cryptology ePrint Archive}, 2006.

\bibitem{OSIDH_security}
P.~Dartois and L.~De~Feo.
\newblock On the security of {OSIDH}.
\newblock In G.~Hanaoka, J.~Shikata, and Y.~Watanabe, editors, {\em Public-Key
  Cryptography -- PKC 2022}, pages 52--81, Cham, 2022. Springer International
  Publishing.

\bibitem{SQISignHD}
P.~Dartois, A.~Leroux, D.~Robert, and B.~Wesolowski.
\newblock {SQISignHD}: New dimensions in cryptography.
\newblock Cryptology ePrint Archive, Paper 2023/436, 2023.
\newblock \url{https://eprint.iacr.org/2023/436}.

\bibitem{DeBruijn1966}
N.~G. de~Bruijn.
\newblock {On the number of positive integers $\leq x$ and free of prime
  factors $>y$, II}.
\newblock {\em Proceedings of the Koninklijke Nederlandse Akademie van
  Wetenschappen}, Series A: Mathematical Sciences(3):239--247, 1966.

\bibitem{SETA}
L.~De~Feo, C.~Delpech~de Saint~Guilhem, T.~B. Fouotsa, P.~Kutas, A.~Leroux,
  C.~Petit, J.~Silva, and B.~Wesolowski.
\newblock {\em Séta: Supersingular Encryption from Torsion Attacks}, pages
  249--278.
\newblock Advances in Cryptology – ASIACRYPT 2021. Springer International
  Publishing, Cham, 2021.

\bibitem{de2020sqisign}
L.~De~Feo, D.~Kohel, A.~Leroux, C.~Petit, and B.~Wesolowski.
\newblock {SQISign}: compact post-quantum signatures from quaternions and
  isogenies.
\newblock In {\em Advances in Cryptology--ASIACRYPT 2020: 26th International
  Conference on the Theory and Application of Cryptology and Information
  Security, Daejeon, South Korea, December 7--11, 2020, Proceedings, Part I
  26}, pages 64--93. Springer, 2020.

\bibitem{DelfsGalbraith2016}
C.~Delfs and S.~D. Galbraith.
\newblock Computing isogenies between supersingular elliptic curves over
  $\mathbb{F}_p$.
\newblock {\em Designs, Codes and Cryptography}, 78(2):425--440, 2016.

\bibitem{Deu41}
M.~Deuring.
\newblock {Die Typen der Multiplikatorenringe elliptischer Funktionenk\"orper.}
\newblock {\em Abh. Math. Sem. Hansischen Univ.}, 14:197–272, 1941.

\bibitem{eisentrager2018supersingular}
K.~Eisentr{\"a}ger, S.~Hallgren, K.~Lauter, T.~Morrison, and C.~Petit.
\newblock Supersingular isogeny graphs and endomorphism rings: reductions and
  solutions.
\newblock In {\em Advances in Cryptology--EUROCRYPT 2018: 37th Annual
  International Conference on the Theory and Applications of Cryptographic
  Techniques, Tel Aviv, Israel, April 29-May 3, 2018 Proceedings, Part III 37},
  pages 329--368. Springer, 2018.

\bibitem{Eisentrager_al_2018}
K.~Eisentr{\"a}ger, S.~Hallgren, K.~Lauter, T.~Morrison, and C.~Petit.
\newblock Supersingular isogeny graphs and endomorphism rings: Reductions and
  solutions.
\newblock In J.~B. Nielsen and V.~Rijmen, editors, {\em Advances in Cryptology
  -- EUROCRYPT 2018}, pages 329--368, Cham, 2018. Springer International
  Publishing.

\bibitem{Elkies_1997}
N.~D. Elkies.
\newblock Elliptic and modular curves over finite fields and related
  computational issues.
\newblock In {\em Computational perspectives on number theory}, 1997.

\bibitem{DBLP:journals/iacr/EriksenPSV23}
J.~K. Eriksen, L.~Panny, J.~Sot{\'{a}}kov{\'{a}}, and M.~Veroni.
\newblock Deuring for the people: Supersingular elliptic curves with prescribed
  endomorphism ring in general characteristic.
\newblock {\em {IACR} Cryptol. ePrint Arch.}, page 106, 2023.

\bibitem{SCALLOP}
L.~D. Feo, T.~B. Fouotsa, P.~Kutas, A.~Leroux, S.-P. Merz, L.~Panny, and
  B.~Wesolowski.
\newblock {SCALLOP}: scaling the {CSI-FiSh}.
\newblock Cryptology ePrint Archive, Paper 2023/058, 2023.
\newblock \url{https://eprint.iacr.org/2023/058}.

\bibitem{Granville2008}
A.~Granville.
\newblock Smooth numbers: Computational number theory and beyond.
\newblock {\em Math. Sci. Res. Inst. Publ.}, 44:267--323, 2008.

\bibitem{hafner1991asymptotically}
J.~L. Hafner and K.~S. McCurley.
\newblock Asymptotically fast triangularization of matrices over rings.
\newblock {\em SIAM Journal on Computing}, 20(6):1068--1083, 1991.

\bibitem{SIDH}
D.~Jao and L.~De~Feo.
\newblock Towards quantum-resistant cryptosystems from supersingular elliptic
  curve isogenies.
\newblock In B.-Y. Yang, editor, {\em Post-Quantum Cryptography}, pages 19--34,
  Berlin, Heidelberg, 2011. Springer Berlin Heidelberg.

\bibitem{Kaneko1989}
M.~Kaneko.
\newblock {Supersingular $j$-invariants as singular moduli ${\rm mod}\, p$}.
\newblock {\em Osaka Journal of Mathematics}, 26(4):849 -- 855, 1989.

\bibitem{Kani1997}
E.~Kani.
\newblock The number of curves of genus two with elliptic differentials.
\newblock {\em Journal f\"ur die reine und angewandte Mathematik},
  1997(485):93--122, 1997.

\bibitem{KLPT}
D.~Kohel, K.~Lauter, C.~Petit, and J.-P. Tignol.
\newblock On the quaternion-isogeny path problem.
\newblock {\em LMS Journal of Computation and Mathematics}, 17(A):418--432,
  2014.

\bibitem{Lang_EF}
S.~Lang.
\newblock {\em Elliptic Functions}.
\newblock Springer-Verlag, 1987.

\bibitem{Lenstra_ECM}
H.~W. Lenstra.
\newblock Factoring integers with elliptic curves.
\newblock {\em Annals of Mathematics}, 126(3):649--673, 1987.

\bibitem{Leroux_mod_pol}
A.~Leroux.
\newblock Computation of hilbert class polynomials and modular polynomials from
  supersingular elliptic curves.
\newblock Cryptology ePrint Archive, Paper 2023/064, 2023.
\newblock \url{https://eprint.iacr.org/2023/064}.

\bibitem{DRfastisogenies}
D.~Lubicz and D.~Robert.
\newblock Fast change of level and applications to isogenies.
\newblock volume~9. Springer, 2023.

\bibitem{Maino_et_al_SIDH}
L.~Maino, C.~Martindale, L.~Panny, G.~Pope, and B.~Wesolowski.
\newblock A direct key recovery attack on {SIDH}.
\newblock Springer-Verlag, 2023.

\bibitem{Martin2002}
G.~Martin.
\newblock An asymptotic formula for the number of smooth values of a
  polynomial.
\newblock {\em Journal of Number Theory}, 93:108--182, 2002.

\bibitem{moree2006two}
P.~Moree and R.~Osburn.
\newblock Two-dimensional lattices with few distances.
\newblock {\em arXiv preprint math/0604163}, 2006.

\bibitem{OSIDH_Onuki}
H.~Onuki.
\newblock On oriented supersingular elliptic curves, 2020.
\newblock \url{https://arxiv.org/abs/2002.09894}.

\bibitem{pizer1980algorithm}
A.~Pizer.
\newblock An algorithm for computing modular forms on $\gamma_0(n)$.
\newblock {\em Journal of algebra}, 64(2):340--390, 1980.

\bibitem{Pollack2018}
P.~Pollack and E.~Trevi{\~n}o.
\newblock {Finding the Four Squares in Lagrange's Theorem}.
\newblock {\em Integers}, 18A:A15, 2018.

\bibitem{Rob_HDR}
D.~Robert.
\newblock Efficient algorithms for abelian varieties and their moduli spaces,
  2021.
\newblock \url{ http://www.normalesup.
  org/~robert/pro/publications/academic/hdr.pdf}.

\bibitem{RobSIDH}
D.~Robert.
\newblock Breaking {SIDH} in polynomial time.
\newblock Cryptology ePrint Archive, Paper 2022/1038, 2022.

\bibitem{Robert2022evaluating}
D.~Robert.
\newblock Evaluating isogenies in polylogarithmic time.
\newblock Cryptology ePrint Archive, Paper 2022/1068, 2022.

\bibitem{Rob_mod_pol}
D.~Robert.
\newblock Some applications of higher dimensional isogenies to elliptic curves
  (overview of results).
\newblock Cryptology ePrint Archive, Paper 2022/1704, 2022.
\newblock \url{https://eprint.iacr.org/2022/1704}.

\bibitem{robert2023breaking}
D.~Robert.
\newblock Breaking {SIDH} in polynomial time.
\newblock In {\em Annual International Conference on the Theory and
  Applications of Cryptographic Techniques}, pages 472--503. Springer, 2023.

\bibitem{RostStol}
A.~Rostovtsev and A.~Stolbunov.
\newblock Public-key cryptosystem based on isogenies.
\newblock Cryptology ePrint Archive, Paper 2006/145, 2006.
\newblock \url{https://eprint.iacr.org/2006/145}.

\bibitem{sawilla2008new}
R.~E. Sawilla, A.~K. Silvester, and H.~C. Williams.
\newblock A new look at an old equation.
\newblock In {\em Algorithmic Number Theory: 8th International Symposium,
  ANTS-VIII Banff, Canada, May 17-22, 2008 Proceedings 8}, pages 37--59.
  Springer, 2008.

\bibitem{Schoof_1995}
R.~Schoof.
\newblock Counting points on elliptic curves over finite fields.
\newblock {\em Journal de Théorie des Nombres de Bordeaux}, (7):219--254,
  1995.

\bibitem{Siegel_1935}
C.~L. Seigel.
\newblock Über die classenzahl quadratischer zahlkörper.
\newblock {\em Acta Arithmetica}, 1(1):83--86, 1935.

\bibitem{Silverman}
J.~H. Silverman.
\newblock {\em The Arithmetic of Elliptic Curves}.
\newblock Springer-Verlag, New York, N.Y., 2009.

\bibitem{sage}
W.~Stein et~al.
\newblock {\em {S}age {M}athematics {S}oftware ({V}ersion 10.0)}.
\newblock The Sage Development Team, 2023.
\newblock {\tt http://www.sagemath.org}.

\bibitem{Voight}
J.~Voight.
\newblock Quaternion algebras.
\newblock v.0.9.23, August 2020.
\newblock \url{https://math.dartmouth.edu/~jvoight/quat.html}.

\bibitem{Modern_comput_alg}
J.~von~zur Gathen and J.~Gerhard.
\newblock {\em Modern Computer Algebra}.
\newblock Cambridge University Press, 3 edition, 2013.

\bibitem{Velu}
J.~Vélu.
\newblock Isogénies entre courbes elliptiques.
\newblock {\em Comptes-rendus de l'Académie des Sciences}, 273:238--241, july
  1971.
\newblock Available at \url{https://gallica.bnf.fr}.

\bibitem{Bweso_orientations}
B.~Wesolowski.
\newblock Orientations and the supersingular endomorphism ring problem.
\newblock In O.~Dunkelman and S.~Dziembowski, editors, {\em Advances in
  Cryptology -- EUROCRYPT 2022}, pages 345--371, Cham, 2022. Springer
  International Publishing.

\bibitem{Wesolowski_EquivProbs}
B.~Wesolowski.
\newblock The supersingular isogeny path and endomorphism ring problems are
  equivalent.
\newblock In {\em 2021 IEEE 62nd Annual Symposium on Foundations of Computer
  Science (FOCS)}, pages 1100--1111, 2022.

\end{thebibliography}

\newpage

\appendix

\section{Singular points on the modular curve $\Phi_\ell(X,Y)=0$}\label{sec: modular singular points}

If $K$ is a field, we denote by $X_0(\ell,K)$ the modular curve $\Phi_\ell(X,Y)=0$ over the field $K$. We also define:
\begin{multline*}
S_0(\ell,\F_{p^2}):=\Bigg\{j\in\F_{p^2} \mbox{ supersingular } \Bigg| \ \exists j'\in\F_{p^2}, \  \ \Phi_\ell(j,j')=\frac{\partial\Phi_\ell}{\partial X}(j,j')
\\
=\frac{\partial\Phi_\ell}{\partial Y}(j,j')=0\Bigg\}
\end{multline*}

\begin{lemma}\label{lemma: cardinality singular}
    Assume that $2\ell<p$. Then, $\# S_0(\ell,\F_{p^2})=O(\ell^{3+o(1)})$.
\end{lemma}

\begin{proof}
Let $j(E)\in S_0(\ell,\F_{p^2})$ and $j(E')\in\F_{p^2}$ such that $(j(E),j(E'))$ is a singular point of $X_0(\ell,\F_{p^2})$ i.e. such that
\[\Phi_\ell(j(E),j(E'))=\frac{\partial\Phi_\ell}{\partial X}(j(E),j(E'))=\frac{\partial\Phi_\ell}{\partial Y}(j(E),j(E'))=0.\]
Schoof proved in \cite[Section 7]{Schoof_1995} that there exists a lift $(j(\widetilde{E}),j(\widetilde{E'}))$ over $\C$ of $(j(E),j(E'))$ that is also a singular point of the modular curve $X_0(\ell,\C)$. Schoof deduced that there exists two $\ell$-isogenies $\phi, \psi: \widetilde{E}\longrightarrow \widetilde{E'}$ over $\C$ that are not equal up to pre or post composition by an isomorphism, so that $\varphi:=\widehat{\psi}\circ\phi$ is a cyclic endomorphism of $\widetilde{E}$ of degree $\ell^2$. Hence, $\varphi$ is non-scalar and $\End(\widetilde{E})$ is isomorphic to an imaginary quadratic order $\mathfrak{O}$ and $\Z[\varphi]$ is mapped to a suborder of $\mathfrak{O}$ via this isomorphism. It follows that $\disc(\mathfrak{O}) \mid \disc(\Z[\varphi])$. But
\[\disc(\Z[\varphi])=\Tr(\varphi)^2-4\deg(\varphi)=\Tr(\varphi)^2-4\ell^2\]
Since $\mathfrak{O}$ is imaginary quadratic, so is $\Z[\varphi]$ and $\disc(\Z[\varphi])<0$, so that $|\disc(\mathfrak{O})|\leq|\disc(\Z[\varphi])|\leq 4\ell^2$.

Since $2\ell<p$, $p$ does not divide the conductor of $\mathfrak{O}$ (otherwise, $p^2 \mid \disc(\mathfrak{O})$ so $p^2\leq 4\ell^2$). It follows by \cite[Lemma 3.1]{OSIDH_Onuki} (generalizing \cite[Chapter 13, Theorem 12]{Lang_EF}), that $E$ is (primitively) $\mathfrak{O}$-oriented. Besides, by \cite[Proposition 3.3 and Theorem~3.4]{OSIDH_Onuki} there are at most $2\#\Cl(\mathfrak{O})$ $j$-invariants of supersingular $\mathfrak{O}$-oriented curves. By Siegel's theorem \cite{Siegel_1935}, we have $\#\Cl(\mathfrak{O})=O(|\disc(\mathfrak{O})|^{1/2+o(1)})=O(\ell^{1+o(1)})$, so there are at most $O(\ell)$ $j$-invariants of $\mathfrak{O}$-oriented supersingular elliptic curves. Taking into account all possible imaginary quadratic orders $\mathfrak{O}$ of discriminant $|\disc(\mathfrak{O})|\leq 4\ell^2$, we conclude that $j(E)$ lies in a set of cardinality $O(\ell^{3+o(1)})$, which completes the proof.
\end{proof}

\begin{lemma}\label{lemma: expected isogeny computation time}
    Assume that $\log(\ell)\ll \log(p)$. Then, computing an $\ell$-isogeny between two $\ell$-isogenous supersingular $j$-invariants chosen uniformly at random takes on average $\tilde{O}(\ell^2\log(p))$ operations over $\F_{p^2}$.
\end{lemma}

\begin{proof}
As discussed in Section \ref{sec: isogenies from j-invariants}, the average number of operations over $\F_{p^2}$ to compute an $\ell$-isogeny between supersingular $\ell$-isogenous $j$-invariants is:
\begin{align*}
    N&:=(1-\mathbb{P}((j(E),j(E')) \mbox{ singular})-\mathbb{P}(j(E')=0,1728))\tilde{O}(\ell^2\log(p))\\
    &\qquad + (\mathbb{P}((j(E),j(E')) \mbox{ singular})+\mathbb{P}(j(E')=0,1728))O(\ell^{7/2})
\end{align*}
Since there are $\sim p/12$ supersingular $j$-invariants by \cite[Theorem V.4.1.c]{Silverman}, we obtain by Lemma \ref{lemma: cardinality singular}:
\[\mathbb{P}((j(E),j(E')) \mbox{ singular})\leq \mathbb{P}(j(E)\in S_0(\ell,\F_{p^2}))=O\left(\frac{\ell^{3+o(1)}}{p}\right)\]
Besides, $\mathbb{P}(j(E')=0,1728)=O(1/p)$. Since $\log(\ell)\ll\log(p)$, we have $\ell^{13/2+o(1)}\ll p$ so the dominant term of $N$ is $\tilde{O}(\ell^2\log(p))$ and all other terms are negligible. Hence, $N=\tilde{O}(\ell^2\log(p))$ and the proof is complete.
\end{proof}

\end{document}